\documentclass[a4paper,leqno]{amsart}

\usepackage[margin=3cm]{geometry}

\usepackage{amsmath,amssymb,amsfonts,amsthm,mathrsfs}
\usepackage{enumitem,xcolor,microtype,lmodern,url}
\usepackage[overload]{empheq}

\usepackage[colorlinks=true,linkcolor=black,citecolor=black]{hyperref}

\newcommand{\A}{\mathbb A}

\newcommand{\D}{\mathbb D}
\newcommand{\R}{\mathbb R}
\newcommand{\N}{\mathbb N}
\newcommand{\T}{\mathbb T}

\newcommand{\F}{\mathcal F}

\newcommand{\DD}{\mathcal D}
\newcommand{\EE}{\mathcal E}

\newcommand{\Ereg}{{\mathcal E}_{\rm reg}}
\newcommand{\Ereginit}{{\mathcal E}_{{\rm reg}\mid t=0}}
\newcommand{\Edrag}{{\mathcal E}_{\rm drag}}
\newcommand{\Edraginit}{{\mathcal E}_{{\rm drag}\mid t=0}}
\newcommand{\EBD}{{\mathcal E}_{\rm BD}}

\newcommand{\EBDreg}{{\mathcal E}_{{\rm BD,reg}}}
\newcommand{\EBDdrag}{{\mathcal E}_{\rm BD, drag}}
\newcommand{\EBDdraginit}{{\mathcal E}_{{\rm BD, drag}\mid t=0}}
\newcommand{\Dreg}{{\mathcal D}_{\rm reg}}
\newcommand{\Ddrag}{{\mathcal D}_{\rm drag}}
\newcommand{\DBD}{{\mathcal D}_{\rm BD}}
\newcommand{\DBDreg}{{\mathcal D}_{\rm BD,reg}}
\newcommand{\DBDdrag}{{\mathcal D}_{\rm BD,drag}}

\def\({\left(}
\def\){\right)}
\def\<{\left\langle}
\def\>{\right\rangle}

\DeclareMathOperator{\Div}{div}

\def\Eq#1#2{\mathop{\sim}\limits_{#1\rightarrow#2}}
\def\Tend#1#2{\mathop{\longrightarrow}\limits_{#1\rightarrow#2}}

\newcommand{\eps}{\epsilon}
\newcommand{\dd}{\mathrm d}
\renewcommand{\r}{\varrho}

\theoremstyle{plain}
\newtheorem{theorem}{Theorem} [section]
\newtheorem{lemma}[theorem]{Lemma}
\newtheorem{corollary}[theorem]{Corollary}
\newtheorem{proposition}[theorem]{Proposition}

\theoremstyle{remark}
\newtheorem{remark}[theorem]{Remark}

\theoremstyle{definition}
\newtheorem{definition}[theorem]{Definition}

\DeclareMathOperator{\RE}{Re}
\DeclareMathOperator{\IM}{Im}

\numberwithin{equation}{section}

\title[Global solutions for isothermal fluids]{Global weak solutions
  for quantum isothermal fluids}

\begin{document}

\author[R. Carles]{R\'emi Carles}
\author[K. Carrapatoso]{Kleber Carrapatoso}
\author[M. Hillairet]{Matthieu Hillairet}

\address[R. Carles]{Univ Rennes, CNRS\\ IRMAR - UMR 6625\\ F-35000
  Rennes, France} 
\address[K. Carrapatoso]{CMLS\\\'Ecole
  polytechnique\\ Institut Polytechnique de Paris\\ 91128 Palaiseau
  cedex\\ France}
\address[M. Hillairet]{Institut Montpelli\'erain
  Alexander Grothendieck\\ Univ. Montpellier\\ CNRS\\ Montpellier\\ France}
\email{Remi.Carles@math.cnrs.fr}
\email{kleber.carrapatoso@polytechnique.edu}
\email{Matthieu.Hillairet@umontpellier.fr}
\begin{abstract}
We construct global weak solutions to isothermal quantum Navier-Stokes
equations, with or without Korteweg term, in the whole space of
dimension at most three. Instead of working on the initial set of
unknown functions, we consider an equivalent reformulation, based on a
time-dependent rescaling, that we
introduced in a previous paper to study the large time behavior, and
which provides suitable a priori estimates, as opposed to the initial
formulation where the potential energy is not signed. We proceed by
working on tori whose size eventually becomes infinite. On each 
fixed torus, we consider the equations in the presence
of drag force terms. Such equations are solved by regularization, and
the limit where the drag force terms vanish is treated by resuming the
notion of renormalized solution developed by I.~Lacroix-Violet and
A.~Vasseur. We also establish global existence of weak solutions for
the isothermal Korteweg equation   (no viscosity), when
initial data are well-prepared, in the sense that they stem from a
Madelung transform. 
\end{abstract}

\maketitle

\tableofcontents

\section{Introduction}\label{sec:intro}
In this paper we consider the isothermal fluid equations in $\R^d$ ($d
\le 3$):
\begin{subequations}\label{fluide}
\begin{align}[left = \empheqlbrace\,]
& \partial_t \r + \Div(\r u) = 0, \label{fluide1} \\
& \partial_t(\r u) + \Div (\r u \otimes u) + \nabla \r
= \frac{\eps^2}{2} \r \nabla \left( \frac{\Delta \sqrt{\r}}{\sqrt{\r}}\right) 
+ \nu \Div(\r \D u) ,   \label{fluide2}
\end{align}
\end{subequations}
on some time interval $(0,T)$. Here, the unknowns are the density $\r: (0,T) \times \mathbb R^d \to   [0,\infty)$ and the velocity field $u: (0,T) \times \mathbb R^d \to \mathbb R^d$ of the fluid.
We denote by $\D u = \tfrac12 (\nabla u + \nabla u^{\top}),$ the symmetric part of $\nabla u$, and $\eps \geq 0,$  $\nu \ge 0$ (with $(\eps, \nu) \neq (0,0)$) are given parameters.
When $\eps = 0$ and $\nu >0$, the system \eqref{fluide} corresponds to the
isothermal quantum Navier--Stokes equations; the case $\eps, \nu >0$
corresponds to the isothermal quantum Navier--Stokes--Korteweg
equations; the case $\eps >0$ and $\nu=0$ to the quantum Euler
equation.   The term $\nabla \rho$ on the left-hand side corresponds to the gradient of the pressure of an isothermal fluid. 
Analytically, this corresponds to a limiting case of equations for
polytropic gases where the pressure is given by a power-law $P(\rho) =
a \rho^{\gamma}$ with $\gamma > 1$ and $a>0$. Such isothermal models are
marginally studied in the literature (see
\cite{Jungel} for the quantum Navier-Stokes equations on $\T^d$, $d\le
2$, and \cite{PlotnikovWeigant,VaPl17} for the 2D Newtonian Navier-Stokes case
on a bounded domain) whereas they have been derived in a quantum
context \cite{BruMe10}. {\color{black} We emphasize in the case of the
  Euler equation 
  ($\eps=\nu=0$),in space dimension $d=1$, the existence of
global weak solution is obtained in \cite{LeFlochShelukhin2005} by the
vanishing viscosity method, under weak assumptions on the initial
data: $0\le \r_0\in L^\infty(\R)$ and $|u_0(x)|\lesssim 1+|\ln
\r_0(x)|$.} 
In a previous paper \cite{CCH}, we studied the large-time behavior of
solutions to \eqref{fluide} with $\eps,\nu\ge 0$, 
under the assumption that sufficiently integrable solutions do exist
globally in time. To our knowledge, the question of the existence of
such solutions remains open,
specifically in the isothermal case.  We answer this question herein
by proving that \eqref{fluide} admits weak solutions globally in time.
The main part of this paper addresses the Navier-Stokes  case $\nu>0$
(with $\eps\ge 0$) for general initial data, while the Korteweg  case
$\nu=0$, $\eps>0$ is considered for well-prepared initial data
(stemming from a Madelung transform), and is much more straightforward.
\smallbreak

Formally, solutions to \eqref{fluide} enjoy the energy
equality
\begin{equation*} 
  E(t) +\int_0^t  D(s) \dd s =E(0),\quad t\ge 0,
\end{equation*}
where the energy is defined by
\begin{equation} \label{Eru}
 E(t) = \frac{1}{2}\int_{\R^d} \left( \r |u|^2 + \eps^2 \left|\nabla
    \sqrt\r\right|^2 \right) +\int_{\R^d} \r \log \r,
\end{equation}
and the dissipation is given by
\begin{equation} \label{Dru}
  D(t) = \nu \int_{\R^d} \r |\D u|^2.
\end{equation}
A feature of the isothermal case is that the pressure part of the
energy, 
\begin{equation*}
  \int_{\R^d} \r \log \r,
\end{equation*}
involves a functional which has no definite sign, as opposed to 
\begin{equation*}
 \frac{1}{\gamma-1} \int_{\R^d} \r^\gamma 
\end{equation*}
in the polytropic case. 
This is one of the reasons
why there are fewer results regarding the global existence of solutions in
the case $\gamma=1$ than in the case $\gamma>1$. Also, because we
consider the case of an unbounded domain $x\in \R^d$, nonzero constant
densities cannot provide finite-energy
solutions to \eqref{fluide}, ruling out
natural candidates for an approach based on relative entropy like in
e.g.\ \cite{BrNoVi17}.
\smallbreak

Following \cite{CCH}, we circumvent this difficulty by considering the
auxiliary unknowns $(R,U)$  as defined by
\begin{equation}
  \label{eq:uvFluid}
  \r(t,x) =
  \frac{1}{\tau(t)^d}R\(t,\frac{x}{\tau(t)}\)
\frac{\|\r_0\|_{L^1}}{\|\Gamma\|_{L^1}},\quad 
  u(t,x) = \frac{1}{\tau(t)} U \(t,\frac{x}{\tau(t)}\) +\frac{\dot \tau(t)}{\tau(t)}x,
\end{equation}
where $\Gamma(y) = e^{-|y|^2}$ and  the function $\tau$ is the global
solution to the nonlinear ODE 
\begin{equation*}
  \ddot \tau = \frac{2}{\tau},\quad \tau(0)=1,\quad \dot\tau(0)=0.
\end{equation*}
We recall (see \cite{CaGa18}) that there exists a unique global
solution $\tau\in C^\infty([0,\infty))$
to this system. This solution remains uniformly bounded from below by a strictly positive constant 
and its  large time behavior is known:
\begin{equation*}
  \tau(t) \Eq t \infty 2 t\sqrt{\log t},\quad \dot \tau(t)\Eq t \infty
  2\sqrt{\log t}.
\end{equation*}
By convention, the space variable for unknowns with
  capital letters will be denoted by $y$, in contrast with the initial
  space variable $x$.
System \eqref{fluide} becomes, in the terms of the new unknown $(R,U)
= (R(t,y), U(t,y))$, 
\begin{subequations}\label{RU}
\begin{align}[left = \empheqlbrace\,]
& \partial_t R+\frac{1}{\tau^2}\Div (R U )=0    \label{RU1}  \\
& \partial_t (R U) +\frac{1}{\tau^2}\Div ( R U \otimes U) +2 y R  + \nabla R  
=\frac{\eps^2}{2\tau^2}R\nabla \( \frac{\Delta  \sqrt{R}}{\sqrt{R}}\)   
+\frac{\nu}{\tau^2} \Div (R \D U) + \frac{\nu \dot \tau}{\tau} \nabla R. \label{RU2}
\end{align}
\end{subequations}
Since the change of unknowns \eqref{eq:uvFluid} preserves the
integrability properties of  density and velocity
unknowns locally in time (we consider velocity \emph{and} space
momenta), we focus in the whole paper on system \eqref{RU}.   

\smallbreak

 An interesting feature of \eqref{RU} is that it is again associated with a  natural energy 
 dissipation estimate, but the new energy involved in this estimate is
 sign-definite and provides 
 important controls for the unknowns. Indeed, as
  exploited in \cite{CCH}, the energy associated to \eqref{RU} reads
  \begin{equation}
    \label{eq:ERU}
    \EE(R,U) = \frac{1}{2\tau^2} \int_{\R^d} \left( R|U|^2 + \eps^2 |\nabla \sqrt{R}|^2 \right)
+ \int_{\R^d} \( R|y|^2 + R\log R\),
  \end{equation}
so that, formally, solutions to \eqref{RU} satisfy the energy equality
\begin{equation}\label{eq:EE+DD=EE0}
  \EE(R,U)(t) +\int_0^t \DD(R , U)(s)  \dd s =\EE(R_0,U_0) - \nu \int_0^t \frac{\dot \tau}{\tau^3} \int_{\R^d} R \Div U    ,\quad t\ge 0,
\end{equation}
where the nonnegative dissipation is given by
\begin{equation} \label{eq:DRU}
  \DD(R,U) 
= \frac{\dot \tau}{\tau^3} \int_{\R^d}\left( R|U|^2 + \eps^2 |\nabla
  \sqrt{R}|^2 \right) 
+ \frac{\nu}{\tau^4} \int_{\R^d} R |\D U|^2.
\end{equation}
In view of the conservation of mass, $\|R(t)\|_{L^1}=\|\Gamma
\|_{L^1}=\pi^d$ for all $t\ge 0$,  we see that the functional $\EE$ is
positive by writing
\begin{equation*}
  \int_{\R^d} \( R|y|^2 + R\log R\) = \int_{\R^d}
R\log\frac{R}{\Gamma}\ge \frac{1}{2\pi^d}\|R-\Gamma\|_{L^1}^2,
\end{equation*}
where the last inequality stems from  Csisz\'ar-Kullback inequality (see
e.g. \cite[Th.~8.2.7]{LogSob}).


\smallbreak

The construction of a positive-definite energy which is dissipated with time is a first building-block to construct 
solutions to \eqref{RU}. However, it is classical in compressible fluid mechanics that
\eqref{eq:EE+DD=EE0} must be completed. 
For instance, studies on compactness of finite-energy solutions to \eqref{RU} 
require to handle the viscous stress $R \mathbb DU.$ Yet, the information
provided by \eqref{eq:EE+DD=EE0} is insufficient (when $\eps =0$) to pass
to the limit in this term (see e.g. \cite{Bre-De-CKL-03,VasseurYu}),
because we lack 
information on the regularity of the density $R$.
More specifically, in the case of \eqref{RU}, with \eqref{eq:EE+DD=EE0} 
alone, it is not clear also how to define
the Korteweg term when $\eps >0.$  
Another important quantity, known as BD-entropy,  introduced  in
 \cite{BD04,Bre-De-CKL-03}, is now standard to handle these difficulties. In the case of \eqref{RU}, it reads
\begin{equation*}
  \EBD (R,U)
= \frac{1}{2\tau^2} \int_{\R^d} 
\left( R|U + \nu \nabla \log R|^2 + \eps^2 |\nabla \sqrt{R}|^2  \right)
+ \int_{\R^d} \left( R|y|^2 + R\log R \right).
\end{equation*}
Exactly as above, the second integral defines a non-negative
functional.  The evolution of this BD-entropy is given formally by
\begin{equation}\label{eq:EBD}
\begin{aligned}
  \EBD(R , U)(t)  &+\int_0^t\DBD(R , U)(s)  \dd s \\
   &=\EBD(R_0 , U_0) + \nu \int_0^t \frac{2 d}{\tau^2} \int_{\R^d} R
  + \nu \int_0^t \frac{\dot \tau}{\tau^3} \int_{\R^d} R \Div U , \quad t\ge 0,
  \end{aligned}
\end{equation}
where the above dissipation is defined by
\begin{equation}\label{eq:DBD}
\begin{aligned}
  \DBD(R,U)
&= \frac{\dot \tau}{\tau^3} \int \left( R|U|^2 + \eps^2 |\nabla
  \sqrt{R}|^2 \right) 
+ \frac{\nu}{\tau^4} \int_{\R^d} R |\A U|^2  \\
&\quad
+\frac{\nu \eps^2}{\tau^4}   \int R|\nabla^2  \log R|^2
+\frac{4 \nu}{\tau^2} \int |\nabla \sqrt{R}|^2 ,
\end{aligned}
\end{equation}
with $\mathbb AU := \frac12 (\nabla U - \nabla U^\top)$ the skew-symmetric part of $\nabla U.$
Hence putting together the energy and the BD-entropy equalities, it holds
\begin{equation}\label{eq:EE+EBD}
  \EE(t)  + \EBD(t) +\int_0^t \(\DD(s) + \DBD(s) \) \dd s = \EE(0) + \EBD(0) 
  + \nu \int_0^t \frac{2 d}{\tau^2} \int_{\R^d} R , \quad t\ge 0,
\end{equation}
and thanks to the conservation of mass and the fact that $\int_0^\infty \tau^{-2}(t) \, \dd t < \infty$, the last term is uniformly bounded.  We note that, in view of \eqref{eq:EBD}, we gain information on the regularity of $R$
when $\nu >0$ which may help in the compactness issue of weak solutions to \eqref{RU}. 
To define the Korteweg term, we may also apply the classical identity:
\begin{equation} \label{eq_Ktwidentity}
R\nabla \( \frac{\Delta  \sqrt{R}}{\sqrt{R}}\)  = {\rm div} (\sqrt{R} \nabla^2 \sqrt{R} - \nabla \sqrt{R} \otimes \nabla \sqrt{R}),
\end{equation}
in view of 
 \begin{equation}\label{eq:equivJungel}
\int_\Omega|\nabla^2 \sqrt
   R|^2 + \int_\Omega|\nabla R^{1/4}|^4\lesssim 
   \int_\Omega R|\nabla^2  \log R|^2\lesssim \int_\Omega|\nabla^2 \sqrt
   R|^2 + \int_\Omega|\nabla R^{1/4}|^4,
 \end{equation}
 which holds true for $\Omega=\R^d$ or  $\T^d$ (see \cite{Jungel,VasseurYu}). 

\smallbreak

The estimates provided by the above energy and BD-entropy turn out to
be fundamental in the construction of a weak solution, and motivate
the following definition: 
\begin{definition} \label{def:weak}
Assume $\nu > 0$ and $\eps \ge 0$. 
Let $(\sqrt{R_0} , \Lambda_0 = (\sqrt{R} U)_0 ) \in L^2(\R^d) \times L^2(\R^d)$.
We call global weak
solution to \eqref{RU}, associated to the initial data $(\sqrt{R_0} , \Lambda_0= (\sqrt{R} U)_0)$, any pair $(R,U)$ such that there exists a collection $(\sqrt{R},\sqrt{R}U,\mathbf S_{K},\mathbf
T_{N})$ satisfying 
\begin{itemize}
\item[i)] The  following regularities:
\begin{align*}
&\(\<y\>+|U|\) \sqrt{R} \in L^{\infty}_{\rm loc}\(0,\infty; L^2 (\R^d)\),\quad \nabla
  \sqrt R \in L^{\infty}_{\rm loc}\(0,\infty; L^2 (\R^d)\),\\ 
& \eps \nabla^2 \sqrt{R} \in L^2_{\rm loc}(0,\infty;L^2(\mathbb
  R^d)) ,\quad 
  \sqrt{\eps} \nabla R^{1/4}\in L^4_{\rm loc}(0,\infty;L^4(\mathbb
  R^d)),\\
&\mathbf T_N \in L^2_{\rm loc}(0,\infty; L^2(\mathbb R^d)) ,
\end{align*}
with the compatibility conditions
\[
\sqrt{R} \ge 0 \text{ a.e. on } (0,\infty)\times \R^d,  \quad   \sqrt{R}U=
0 \text{ a.e. on } \{\sqrt{R} = 0 \}.
\]
\item[ii)] The following equations in $\mathcal D'((0,\infty)\times \mathbb R^d)$
\begin{equation}\label{eq:NSKrevu}
  \left\{
    \begin{aligned}
  & \partial_t\sqrt{R}+\frac{1}{\tau^2}\Div (\sqrt{R} U )=
  \frac{1}{2\tau^2}{\rm Trace}(\mathbf T_N),\\
    &\partial_t ({R}U) +\frac{1}{\tau^2}\Div ( \sqrt{R}U \otimes \sqrt{R}U)
      +2 y |\sqrt{R}|^2 +\nabla \left( |\sqrt{R}|^2 \right)
      \\
      & \phantom{\partial_t (\sqrt{R} \sqrt{R}U) +\frac{1}{\tau^2}\Div }=\Div \left(\dfrac{\nu}{\tau^2} \sqrt R \mathbf S_N
        + \dfrac{\eps^2}{2\tau^2} \mathbf S_K\right) +
      \dfrac{\nu \dot{\tau}}{\tau} \nabla R,
    \end{aligned}
\right.
\end{equation}
with $\mathbf S_N$ the symmetric part of $\mathbf T_N$ and the compatibility conditions:
\begin{align} \label{eq_compnewton}
& \sqrt{R}\mathbf T_{N} = \nabla(\sqrt{R}  \sqrt{R} U) - 2 \sqrt{R}U
  \otimes \nabla \sqrt{R}\,, \\[6pt] 
& \mathbf S_K 
=\sqrt{R}\nabla^2 \sqrt{R} -  \nabla \sqrt{R} \otimes \nabla \sqrt{R}
  \,. \label{eq_compkorteweg} 
\end{align}
\item[iii)] For any $\psi\in C_0^\infty(\R^d)$, 
  \begin{align*}
    &\lim_{t\to 0}\int_{\R^d} \sqrt{R}(t,y)\psi(y) \, \dd y=
    \int_{\R^d} \sqrt{R_0} (y) \psi(y)  \, \dd y , \\
& \lim_{t\to 0}\int_{\R^d} \sqrt{R}(t,y)    (\sqrt{R}U)(t,y)\psi(y) \, \dd y=
    \int_{\R^d} \sqrt{R_0} (y) \Lambda_0(y) \psi(y) \, \dd y.
  \end{align*}
\end{itemize}

\end{definition}

A specific feature of the previous statement is that we define weak solutions to \eqref{RU} in terms of $\sqrt{R}$
and $\sqrt{R}U.$ This is related to the fact that these are the natural quantities that are involved in the energy and 
entropy estimates. By construction, we shall have $\sqrt{R}U = 0$
where $\sqrt{R}=0$ so that, whenever  
$U$ is mentioned, it should be understood as: 
\[
U = \dfrac{\sqrt{R}U}{\sqrt{R}} \mathbf{1}_{\sqrt{R}>0}.
\]
Also, thanks to the regularity estimates obtained on the density, the above weak formulation implies the classical continuity equation (see \cite[Lemma 2.2]{CCH}). On the other hand,
we mention that a solution $(\sqrt{R},\sqrt{R}U)$ in the sense of distributions enjoying the regularity of i)
satisfies furthermore that $\sqrt{R} \in C([0,\infty),L^2(\mathbb R^d)-w)$ and $RU \in C([0,\infty) ; L^1(\mathbb R^2) -w).$ Consequently, we may require the initial conditions in terms of item iii). 
Finally, we do not claim for an energy estimate in our definition,
however we shall derive these solutions  
from approximate finite-energy, finite-entropy solutions, so that the
global weak solutions we construct satisfy: 
There exist absolute constants $C,C'$ such that, for almost all $t
\geq 0$,  there holds: 
\begin{align}  \label{eq:EE+DD<=EE0}
\mathcal E(t) + \int_0^t \mathcal D(s) \,  \dd s &  \leq C (\mathcal E(0)), \\
\label{eq_EBD_weak}
\mathcal E_{\mathrm{BD}}(t) + \int_0^t \mathcal D_{\mathrm{BD}}(s) \,  \dd s &\leq C' (\mathcal E(0), \mathcal E_{\mathrm{BD}}(0)) ,
\end{align}
with $\mathcal E,\mathcal D,\mathcal E_{\mathrm{BD}},\mathcal D_{\mathrm{BD}}$ as defined in \eqref{eq:ERU}-\eqref{eq:DRU}-\eqref{eq:EBD}-\eqref{eq:DBD}.  In terms of our weak solutions, 
the term $R |\mathbb D U|^2$ appearing in these estimates must be
understood as $|\mathbf S_N|^2$ (and, similarly, $R |\mathbb A U|^2$
as $|\mathbf T_N - \mathbf S_N|^2$ , and $R |\nabla U|^2$ as $|\mathbf
T_N|^2$). 
In addition, item  i) along with \eqref{eq:NSKrevu} imply
the conservation of mass, 
  \begin{equation*}
    \int_{\R^d}R(t,y)\dd y = \int_{\R^d} R_0(y) \dd y ,\quad \forall
    t\ge 0,
  \end{equation*}
which is hence fixed through all the paper. The extra integral terms
present on the right hand side 
of \eqref{eq:EE+DD=EE0} and \eqref{eq:EBD} do not appear in the 
estimates \eqref{eq:EE+DD<=EE0} and \eqref{eq_EBD_weak}: thanks to
Cauchy-Schwarz inequality, and the conservation of mass, they can be
controlled by   
the dissipation $\mathcal D$ (see \cite[Remark~2.13]{CCH} as well as
the proof of Proposition~\ref{prop:sol-1} below).
Note that in the previous
definition, the entropy of $R$ is not mentioned. The reason is the
following lemma.
\begin{lemma}\label{lem:SigmaLlogL}
  Let $d\ge 1$. For all $M>0$, there exists $C(M)$ such that for all
  $f\in H^1\cap \F(H^1)(\R^d)$
  satisfying
  \begin{equation*}
    \int_{\R^d}\(1+|y|^2\)|f(y)|^2\dd y +\int_{\R^d}|\nabla
    f(y)|^2\dd y \le M,
  \end{equation*}
  the $L\log L$ norm of $|f|^2$ is controlled by
  \begin{equation*}
    \int_{\R^d}|f(y)|^2\left|\log \(|f(y)|^2\)\right|\dd y \le C(M).
  \end{equation*}
\end{lemma}
\begin{proof}[Sketch of proof]
  We distinguish the regions where $|f|$ is smaller or larger than
  one,
  \begin{align*}
    \int_{\R^d}|f(y)|^2\left|\log \(|f(y)|^2\)\right|\dd y & \le
    \int_{|f|<1}|f(y)|^2\left|\log \(|f(y)|^2\)\right|\dd y  +
     \int_{|f|>1}|f(y)|^2\left|\log \(|f(y)|^2\)\right|\dd y \\
    &\lesssim
    \int_{\R^d} |f(y)|^{2-\beta}\dd y + \int_{\R^d} |f(y)|^{2+\beta}\dd y ,
  \end{align*}
  where $\beta>0$ is arbitrarily small. We then invoke the  localization estimate in the former
  region,
  \begin{equation*}
  \int_{\R^d}|f|^{2-\beta} \le C_\beta
  \|f\|_{L^2}^{2-\beta-d\beta/2}\||y|f\|_{L^2}^{d\beta/2},\quad
0<\beta<\frac{4}{d+2},
\end{equation*}
which is easily established by distinguishing the regions $|y|
<\kappa$ and $|y|>\kappa$, introducing $|y|^2/|y|^2$ in the latter,
using H\"older  
inequality, and eventually 
optimizing in $\kappa$. We may take $\beta=\frac{2}{d+2}$, and the term $\int
|f|^{2+\beta}$ is then controlled by the $H^1$-norm of $f$ thanks to
Sobolev embedding.
\end{proof}
Of course if $H^1\cap \F(H^1)$ is replaced by $H^1$, the above lemma
is no longer true. In view of the above discussion, we will apply this
lemma to $\sqrt R$. Recalling that the presence of a space momentum is
natural when working with the unknown $(R,U)$ (due to \eqref{RU2},
implying the definition \eqref{eq:ERU}), this yields another
motivation for working with $(R,U)$ instead of $(\r,u)$: we definitely
gain coercivity properties. 
\smallbreak

With the above definition, the main result of this paper reads:
\begin{theorem}\label{theo:main}
   Assume $\nu>0$, $\eps\ge 0$.
   Let $(\sqrt{R_0} , \Lambda_0 = (\sqrt{R} U)_0 ) \in L^2(\R^d) \times L^2(\R^d)$ satisfy $\EE(0) < \infty$, $\EBD(0) < \infty$,
as well as the compatibility conditions
\[
\sqrt{R_0} \ge 0 \text{ a.e.\ on } \R^d,  \quad   (\sqrt{R}U)_0=
0 \text{ a.e.\ on } \{\sqrt{R_0} = 0 \}.
\]
There exists  at least one global weak solution to \eqref{RU}, which satisfies moreover the energy and BD-entropy inequalities \eqref{eq:EE+DD<=EE0} and \eqref{eq_EBD_weak}. 
\end{theorem}
In view of \cite{CCH}, we readily infer the following corollary:
\begin{corollary}\label{cor:tempslong}
  Under the assumptions  of Theorem~\ref{theo:main}, every global weak
  solution to \eqref{RU}  enjoying the energy  inequality
  \eqref{eq:EE+DD<=EE0}   satisfies
  \begin{equation*}
    R(t)\rightharpoonup \Gamma \quad\text{in }L^1(\R^d), \text{ as }t\to\infty.
  \end{equation*}
\end{corollary}

To construct solutions of \eqref{RU}, we consider various
levels of approximation, by resuming the approach of
\cite{VasseurYuInventiones} (summarized in \cite{RoussetBBK}) in the
case $\gamma>1$. The first approximation
consists in adding two new terms in the left hand side of \eqref{RU2},
leading to more 
dissipation, hence better a priori estimates,
\begin{equation*}
  \frac{r_0}{\tau^2}U +\frac{r_1}{\tau^2}R|U|^2U.
\end{equation*}
This yields the following system in $\R^d$, for $r_0,r_1 \ge 0$:
\begin{subequations}\label{RU-drag}
\begin{align}[left = \empheqlbrace\,]
& \partial_t R+\frac{1}{\tau^2}\Div (R U )=0    \label{RU-drag1}  \\
& \partial_t (R U) +\frac{1}{\tau^2}\Div ( R U \otimes U) +2 y R  + \nabla R  
+ \frac{r_0}{\tau^2} U + \frac{r_1}{\tau^2} R |U|^2 U  \label{RU-drag2} \\
&\qquad\qquad\qquad\qquad\qquad
=\frac{\eps^2}{2\tau^2}R\nabla \( \frac{\Delta  \sqrt{R}}{\sqrt{R}}\)   
+\frac{\nu}{\tau^2} \Div (R \D U) + \frac{\nu \dot \tau}{\tau} \nabla R.\nonumber 
\end{align}
\end{subequations}
When $r_0 , r_1 >0$ we call this system the \emph{isothermal fluid system with drag forces}, whereas when $r_0=r_1=0$ we recover the original system \eqref{RU}. 
When the factor $1/\tau^2$ is absent, these terms
correspond to physical models; see
e.g. \cite{BD03,BD07} and references therein. 

\smallbreak

The change of unknown functions \eqref{eq:uvFluid} involves a
time-dependent spatial rescaling, an aspect which essentially forces
us to consider the geometrical framework $x\in \R^d$. On the other
hand, construction of weak solutions in the context of compressible
fluid mechanics is often performed in the periodic case $x\in \T^d$:
this geometry provides compactness in space more easily, and
integrations by parts are harmless. The periodic case is also rather
convenient for approximating, among others in Lebesgue spaces, the
initial density by a density bounded away from zero (see
\eqref{RUinitial-theta2} below), a step which
would be more delicate on $\R^d$. Note also that this property is
classically propagated  by the flow in a suitable regularized continuity equation
(see e.g. \cite{Fei04,Jungel}), and such a property is needed in the
presence of cold pressure 
and regularizing terms (see e.g. \cite{Gis-VV15,VasseurYu}). For these reasons, the second step in
our approach consists in replacing  $\R^d$ with a box $\T^d_\ell$ of
size $\ell>0$, where $\ell$ is aimed at going to infinity at the last
step of the construction of solutions to the system with drag forces \eqref{RU-drag} with $r_0, r_1 >0$. The most delicate step turns out to be the
adaptation of the initial data, given on $\R^d$, in order to fit in
the periodic framework. Details are given in Section~\ref{sec:tore}.

\smallbreak

We also emphasize another important difference whether the space
variable belongs to $\T^d$ or to $\R^d$. In the former case, it is
possible to overcome the lack of positivity in the energy \eqref{Eru}
by introducing an intermediary constant density, as in
e.g. \cite{BGL19,BrNoVi17,Jungel}. This strategy cannot be carried out in the
case $x\in \R^d$, since no non-zero constant belongs to $L^1(\R^d)$.

\smallbreak

To solve  \eqref{RU-drag} on the torus $\T^d_\ell$, we proceed as in \cite{VasseurYu}
and introduce regularizing terms in \eqref{RU-drag1} and \eqref{RU-drag2}.
 This regularized system hence becomes
\begin{subequations}\label{RU-drag-reg}
\begin{align}[left = \empheqlbrace\,]
& \partial_t R+\frac{1}{\tau^2}\Div (R U )= \frac{\delta_1}{\tau^2} \Delta R,    \label{RU-drag-reg1}  \\
& \partial_t (R U) +\frac{1}{\tau^2}\Div ( R U \otimes U) +2 y R  + \nabla R - \eta_1 \nabla R^{-\alpha}\label{RU-drag-reg2} \\ 
&\quad\qquad 
+ \frac{r_0}{\tau^2} U + \frac{r_1}{\tau^2} R |U|^2 U + \frac{\delta_1}{\tau^2}(\nabla R \cdot \nabla) U  \nonumber \\\notag
&\quad\qquad
=\frac{\eps^2}{2\tau^2}R\nabla \( \frac{\Delta  \sqrt{R}}{\sqrt{R}}\)   
+\frac{\nu}{\tau^2} \Div (R \D U) + \frac{\nu \dot \tau}{\tau} \nabla R
+\frac{\delta_2}{\tau^2} \Delta^2 U  + \frac{\eta_2}{\tau^2} R \nabla
   \Delta^{2s+1} R  ,
\end{align}
\end{subequations}
where the regularization parameters verify $0< \delta_1, \delta_2, \eta_1, \eta_2 <1$; $\alpha,s >0$ are 
chosen sufficiently large (to be fixed later on); and the drag forces parameters $r_0,r_1$ as well as the Korteweg parameter $\eps$ are positive $r_0 , r_1 , \eps >0$. 
Such solutions are constructed in Section~\ref{sec:regularised}. Next,
passing to the limit $\delta_1,\delta_2\to 0$, then $\eta_1,\eta_2\to
0$,
we obtain a solution to the system with drag forces \eqref{RU-drag} with $r_0, r_1 , \eps >0$ on the torus $\T^d_\ell$. 
This is achieved in Section~\ref{sec:drag}. 

\smallbreak 

To pass to the limits  $\theta\to 0$, where $\theta>0$ measures the
fact that the initial density is bounded away from zero (see
\eqref{RUinitial-theta2}), $r_0,r_1\to 0$ and $\ell\to \infty$
(simultaneously), we proceed as in \cite{LacroixVasseur}, and 
consider an adapted notion of renormalized solutions, which is
equivalent to our notion of weak solution in the presence of drag
forces terms, and provides a weak solution when
$r_0=r_1=0$. We thus obtain a solution to
\eqref{RU} on the whole space. Note that this step
has to be the final one, insofar as the case with drag forces requires
to control $r_0\(\log R\)_-$ in $L^1$ (see e.g. \cite{VasseurYu}),
which is inconsistent with the property $\sqrt R\in H^1$ in the
case $y\in \R^d$.
These steps are performed in Section~\ref{sec:tore}. 
\smallbreak

We note that these final limits, $\theta\to 0$,  $r_0,r_1\to 0$, and
$\ell\to \infty$ could be performed in a more independent fashion, by
letting first $\theta,r_0,r_1\to 0$, thus obtaining a global weak solution to
\eqref{RU} on $\T_\ell^d$, and then letting $\ell\to \infty$
(recalling that $H^1\cap \F(H^1)$ provides more compactness than the mere $H^1$
space). We choose to unify these steps in order to shorten the overall
presentation, and also since \eqref{RU} is meaningful on $\R^d$ in
view of \eqref{eq:uvFluid}, but not necessarily on a (time-independent) torus. 

\smallbreak

We explain now the outcome of our main theorem in terms 
of the initial system \eqref{fluide}. This is the content of the following corollary:

\begin{corollary}
   Assume $\nu>0$ and $\eps\ge 0$. 
Let $(\sqrt{\r_0} , \lambda_0 = (\sqrt{\r} u)_0 ) \in
H^1  \cap \F(H^1) (\R^d) \times L^2(\R^d)$ satisfy the compatibility conditions
\[
\sqrt{\r_0} \ge 0 \text{ a.e.\ on } \R^d,  \quad   (\sqrt{\r}u)_0=
0 \text{ a.e. on } \{\sqrt{\r_0} = 0 \},
\]
and assume that  the associated functions $(\sqrt{R_0} , \Lambda_0 =
(\sqrt{R} U)_0 )$ obtained via  \eqref{eq:uvFluid} satisfy $\EE(0) <
\infty$ and $\EBD(0) < \infty$. 
Then there exists a global weak solution to \eqref{fluide} in
the following sense: there exists a collection $(\sqrt\r,\sqrt\r u , \mathbb T_N, \mathbb S_K)$ such that

\smallskip

\item[i)] The  following regularities are satisfied:
\begin{align*}
&\(\<x\>+|u|\) \sqrt{\r} \in L^{\infty}_{\rm loc}\(0,\infty; L^2 (\R^d)\),\quad \nabla
  \sqrt \r \in L^{\infty}_{\rm loc}\(0,\infty; L^2 (\R^d)\),\\ 
& \eps \nabla^2 \sqrt{\r} \in L^2_{\rm loc}(0,\infty;L^2(\mathbb
  R^d)) ,\quad 
  \sqrt{\eps} \nabla \r^{1/4}\in L^4_{\rm loc}(0,\infty;L^4(\mathbb
  R^d)),\\
&\mathbb T_N \in L^2_{\rm loc}(0,\infty; L^2(\mathbb R^d)) ,
\end{align*}
with the compatibility conditions
\[
{\sqrt{\r}} \ge 0 \text{ a.e. on } (0,\infty)\times \R^d,  \quad   \sqrt{\r}u=
0 \text{ a.e. on } \{\sqrt{\r} = 0 \}.
\]
\item[ii)] The following equations hold in $\mathcal D'((0,\infty)\times \mathbb R^d)$
\begin{equation}\label{eq:NSKrevu_petit}
  \left\{
    \begin{aligned}
  & \partial_t\sqrt{\r}+ \Div (\sqrt{\r} u )=
  \frac{1}{2}{\rm Trace}(\mathbb T_N),\\
    &\partial_t ({\r}u) + \Div ( \sqrt{\r}u \otimes \sqrt{\r}u)
     +\nabla \left( |\sqrt{\r}|^2 \right)
      =\Div \left(\dfrac{\nu}{\tau^2} \sqrt \r \mathbb S_N
        + \dfrac{\eps^2}{2} \mathbb S_K\right),
    \end{aligned}
\right.
\end{equation}
with $\mathbb S_N$ the symmetric part of $\mathbb T_N$ and the compatibility conditions:
\begin{align} \label{eq_compnewton_petit}
& \sqrt{\r}\mathbb T_{N} = \nabla(\sqrt{\r}  \sqrt{\rho} u) - 2 \sqrt{\r}u
  \otimes \nabla \sqrt{\r}\,, \\[6pt] 
& \mathbb S_K 
=\sqrt{\r}\nabla^2 \sqrt{\r} -  \nabla \sqrt{\r} \otimes \nabla \sqrt{\r}
  \,. \label{eq_compkorteweg_petit} 
\end{align}
\item[iii)] For any $\psi\in C_0^\infty(\R^d)$, 
  \begin{align*}
    &\lim_{t\to 0}\int_{\R^d} \sqrt{\r}(t,x)\psi(x) \, \dd x=
    \int_{\R^d} \sqrt{\r_0} (x) \psi(x)  \, \dd x , \\
& \lim_{t\to 0}\int_{\R^d} \sqrt{\r}(t,x)    (\sqrt{\r}u)(t,x)\psi(x) \, \dd x=
    \int_{\R^d} \sqrt{\r_0} (x) \lambda_0(x) \psi(x) \, \dd x.
  \end{align*}
\end{corollary}

The main shortcoming of this construction is that we do not get the energy inequality corresponding to \eqref{Eru}
for the initial system
(but the regularity obtained ensures that, at any
time $t\ge 0$, the energy $E(t)$ is well defined).
Indeed, we remark that, if $U$ should be going to $0$ at infinity, then, our solution $u$
would then be a perturbation of the affine velocity field $(\dot{\tau}/\tau) x$ which increases at infinity.
In particular, performing back the change of variable \eqref{eq:uvFluid} 
in the energy estimate \eqref{eq:EE+DD<=EE0}, in the case $\|\rho_0\|_{L^1(\mathbb R^d)} = \|\Gamma\|_{L^1(\mathbb R^d)} $ we obtain:
\begin{align*}
& \dfrac{1}{2}\left[ \int_{\mathbb R^d} \rho(t,x) \left|u - \dfrac{\dot{\tau}}{\tau} x\right|^2 {\rm d}x  +
 \int_{\mathbb R^d} |\nabla \sqrt{\rho}(t,x)|^2{\rm d}x \right] 
 \\ 
&  + \int_{\mathbb R^d} \rho(t,x) \ln (\rho(t,x)) {\rm d}x+ d  \left( \ln(\tau(t)) + \frac{1}{\tau(t)^2}\right)\int_{\mathbb R} \rho(t,x){\rm d}x 
\\
&  + \int_0^t \left[ \int_{\mathbb R^d}\dfrac{\dot{\tau}}{\tau}\rho\left|u - \dfrac{\dot{\tau}}{\tau} x\right|^2 {\rm d}x + \nu \int_{\mathbb R^d} \rho |\mathbb D u - \dfrac{\dot{\tau}}{\tau} |^2 \right]{\rm d}x{\rm d}s  \leq C_0 .
\end{align*}
Another point of view consists in recalling that in \cite{CCH},  the
large time convergence of the second order momentum of $R$ is
established by using the a priori bounds provided by
\eqref{eq:EE+DD<=EE0}, \emph{and} the information that the energy $E$
defined in \eqref{Eru} is $o(\log t)$ as $t\to \infty$: even though
this information is weaker than the expected boundedness of $E$ (and
even, decay), it seems to
be needed in the proof, suggesting that either some tools are missing in the
study of $(R,U)$ to recover the energy inequality corresponding to \eqref{Eru}
for the initial system, or that it is just not possible.
\bigbreak

We complement the above results, valid for $\nu>0$, with a global
existence result in the case of the isothermal Korteweg
equation ($\eps>0$ and $\nu=0$). The proof is fairly different from
the case $\nu>0$, since it is based on nonlinear Schr\"odinger
equations, but is rather short. We choose to present this case so that the
family of results in this paper is consistent. 
Mimicking Definition~14 from \cite{AnMa09}, we set:
\begin{definition}\label{def:sol-korteweg}
  Let $d\ge 1$.
  Assume $\nu = 0$ and $\eps > 0$. 
Let $(\sqrt{\r_0} , \lambda_0) \in L^2(\R^d) \times L^2(\R^d)$.
We call global weak
solution to \eqref{fluide}, associated to the initial data
$(\sqrt{\r_0} , \lambda_0)$, any pair $(\sqrt\r,\sqrt\r u)$  such that if we define
$\r:=\(\sqrt{\r}\)^2$, 
$j:=\sqrt \r \times \sqrt\r u$, then we have:
\begin{itemize}
\item[i)] The  following regularities:
\begin{equation*}
  \sqrt \r \in L^{\infty}_{\rm loc}\(0,\infty; H^1 (\R^d)\),\quad
\sqrt\r u\in L^{\infty}_{\rm loc}\(0,\infty; L^2 (\R^d)\),
\end{equation*}
with the compatibility condition
\begin{equation*}
  \sqrt\r \ge 0\text{ a.e. on }(0,\infty)\times\R^d,\quad
  \sqrt\r u=0 \text{ a.e. on }\{\r=0\}. 
\end{equation*}
\item[ii)] For every $T>0$, for any test function $\varphi\in
  C_0^\infty([0,T[\times \R^d)$,
  \begin{equation*}
    \int_0^T \int_{\R^d}\(\r \partial_t \varphi +j\cdot \nabla \varphi\)\dd
    t \dd x + \int_{\R^d} \r_0\varphi(0)\dd x=0,
  \end{equation*}
  and for any test function $\eta\in C_0^\infty([0,T[\times
  \R^d;\R^d)$,
  \begin{align*}
    \int_0^T \int_{\R^d}&\( j\cdot \partial_t \eta
    +(\sqrt\r u)\otimes (\sqrt\r u):\nabla \eta +\nabla \r \Div \eta
    +\eps^2\nabla \sqrt\r \otimes \nabla \sqrt\r :\nabla
                          \eta - \frac{\eps^2}{4} \r \Delta \Div \eta \)\dd t\dd x\\
 &   +\int_{\R^d}\lambda_0\cdot \eta(0)\dd x=0.
  \end{align*}
\item[iii)] (Generalized irrotationality condition) For almost every
  $t\ge 0$,
  \begin{equation*}
    \nabla \wedge j = 2\nabla \sqrt\r \wedge \(\sqrt\r u\)
  \end{equation*}
  holds in the sense of distributions.
\end{itemize}
\end{definition}
Note that in the second point,  the quantum pressure (right hand side
of \eqref{fluide2}) 
has been recast in view of \eqref{eq_Ktwidentity}. Like before,
whenever $u$ is mentioned, it should be understood as
\begin{equation*}
  u = \frac{\sqrt\r u}{\sqrt\r}\mathbf 1_{\sqrt\r>0}. 
\end{equation*}
The generalized irrotationality condition, explained in
\cite[Remark~2]{AnMa09},  is the generalization of the property $\r
\nabla\wedge u=0$ of the smooth case $j=\r u$, to the notion of weak
solution. 

Also, Definition~\ref{def:sol-korteweg} is readily adapted to the case
of \eqref{RU} in the following statement. The first part of this result is the analogue
of \cite[Proposition~15]{AnMa09} in the isothermal case. 
\begin{proposition}\label{prop:korteweg}
 Let $d\ge 1$.  Assume $\nu=0$ and $\eps>0$. Let $\psi_0\in
 \textcolor{black}{H^1}\cap \F(H^\alpha)(\R^d)$ for some $0<\alpha\le 1$, and assume that
 the initial data for \eqref{fluide} are \emph{well-prepared} in the
 sense that
 \begin{equation*}
   \r_0 = |\psi_0|^2,\quad j_0 = \eps\IM \(\bar \psi_0\nabla
   \psi_0\). 
 \end{equation*}
 $(1)$ Then there exists a global  weak solution to
 \eqref{fluide}. Furthermore, the energy $E(t)$ defined by \eqref{Eru}
 is conserved for all time $t\ge 0$. \\
 $(2)$ If $\psi_0\in
 \textcolor{black}{H^1}\cap \F(H^1)(\R^d)$, then $(\sqrt R,\sqrt R U)$ defined by
 \begin{equation}\label{eq:uvKorteweg}
   \begin{aligned}
    &   \sqrt\r(t,x) =
  \frac{1}{\tau(t)^{d/2}}\sqrt R\(t,\frac{x}{\tau(t)}\)
\(\frac{\|\r_0\|_{L^1}}{\|\Gamma\|_{L^1}}\)^{1/2},\\
 &\sqrt \r u (t,x) =
\frac{1}{\tau(t)} \sqrt R U\(t,\frac{x}{\tau(t)}\)
\(\frac{\|\r_0\|_{L^1}}{\|\Gamma\|_{L^1}}\)^{1/2} +
  \frac{\dot\tau(t)}{\tau(t)} x \sqrt{\r(t,x)} ,
  \end{aligned}
 \end{equation}
 is
 a global weak solution to \eqref{RU}. The pseudo-energy $\mathcal E$,
 defined in \eqref{eq:ERU}, solves \eqref{eq:EE+DD=EE0}, where the
 dissipation is given by \eqref{eq:DRU}.
Equivalently, setting 
 \begin{equation*}
  \EE = \frac{1}{2\tau^2} \int_{\R^d} \left( |\sqrt R U|^2 + \eps^2
    |\nabla \sqrt{R}|^2 \right) 
+ \int_{\R^d} \( R|y|^2 + R\log R\),
  \quad  \DD =\frac{\dot \tau}{\tau^3}\int_{\R^d} |\sqrt R U |^2 +
   \eps^2|\nabla\sqrt R|^2,
 \end{equation*}
we have 
\begin{equation*}
  \EE(t) +\int_0^t \DD(s)   \dd s =\EE(0)  ,\quad \forall t\ge 0.
\end{equation*}
\end{proposition}
\color{black}The proof of Proposition~\ref{prop:korteweg} relies on properties of the logarithmic Schr\"odinger
equation, which is the natural candidate to provide solutions to
\eqref{fluide}, as opposed to the nonlinear Schr\"odinger equation
with power-like nonlinearity in the polytropic case. The specificity of
this nonlinearity explains the presence of a (fractional)
momentum in the first part of the statement. We emphasize the fact
that the special structure of the initial data (due to the use of
Madelung transform) implies that the flow is irrotational (see also
the last point of Definition~\ref{def:sol-korteweg} and
\cite[Remark~2]{AnMa09} where it is discussed).
\color{black}
In view of \cite{CCH}, we readily infer the following corollary, which
is stronger than Corollary~\ref{cor:tempslong}: 
\begin{corollary}
  In the second case of Proposition~\ref{prop:korteweg}, every such
  global weak solution satisfies
  \begin{equation*}
    \int_{\R^d}
\begin{pmatrix}
    1\\
y\\
|y|^2
  \end{pmatrix}
R(t,y)\dd y\to
 \int_{{\mathbb R}^d}
  \begin{pmatrix}
    1\\
y\\
|y|^2
  \end{pmatrix}
\Gamma(y)\dd y ,\quad
    R(t)\rightharpoonup \Gamma \quad\text{in }L^1(\R^d), \text{ as }t\to\infty.
  \end{equation*}
\end{corollary}

\begin{remark}
  In view of the proof of Proposition~\ref{prop:korteweg},
  Theorem~1.12 in \cite{CaGa18} implies that
  Proposition~\ref{prop:korteweg} and its corollary (from \cite{CCH}) remain valid in
  the case where the above pressure law $p(\r)= \r$ is replaced for
  instance by
  \[p(\r) = c_0\r+ \sum_{j=1}^Nc_j \r^{\gamma_j},\quad c_j>0,\ 0\le
    j\le N,\quad 1<\gamma_j<\frac{d+2}{(d-2)_+}.\]
  \end{remark}

\color{black}

\begin{remark}
  Since our reformulation of \eqref{fluide} in terms of the unknowns
  $(R,U)$ provides extra positivity properties, one may ask if the
  isothermal case can be obtained as the limit $\gamma\to 1$ in the
  barotropic case, where the pressure law is $p(\r)=\r^\gamma$,
  $\gamma>1$. A first aspect is that such a limit might be possible
  only \emph{locally in time}, for as proven in \cite{CCH} (isothermal
  case) and \cite{CCH-p} (barotropic case), $\r$ enjoys dispersive
  properties with a rate that changes precisely for the value
  $\gamma=1$.  For bounded time, it is plausible that the limit $\gamma\to 1$ might be handled
  in terms of $(R,U)$ (adapted to the case  $\gamma>1$) when $\epsilon >0$ because of a further uniform bound $\sqrt{R} \in H^1(\mathbb R^d)$  due to the Korteweg term.  On the other hand, having proven
  Theorem~\ref{theo:main}, one may ask if the solutions from
  Proposition~\ref{prop:korteweg} can be obtained through the inviscid
  limit $\nu\to 0$. Such a convergence has been proven in \cite{BGL19}
  for the barotropic case, and \cite{ChaAA} for the (damped)
  isothermal case, both times in a periodic setting $x\in \T^d$. The
  damping in \cite{ChaAA} can easily be removed, but in order to
  consider the case $x\in \R^d$, the order of the limits $\ell\to
  \infty$ and $\nu\to 0$ is certainly a delicate issue, which we leave
  out at this stage.
Finally, both limits $\gamma \to 1$ and $\nu \to 0$ seem highly singular when $\epsilon =0$ (or goes simultaneously to $0$) even in terms of $(R,U)$.  Concerning the limit $\gamma \to 1$ for instance, the estimates established in \cite{CCH-p} are then not  uniform in $\gamma$.  
  
\end{remark}
\color{black}

\subsubsection*{Organization of the paper}
Until the end of Section~\ref{sec:tore}, we assume $\nu>0$. In
Section~\ref{sec:reg}, we construct solutions to \eqref{RU-drag-reg} on
the torus $\T_\ell^d$ with strictly positive densities. In
Section~\ref{sec:drag}, we obtain solutions to \eqref{RU-drag}
in the presence of drag forces, $r_0,r_1>0$, by passing to the limit
$\delta_1,\delta_2,\eta_1,\eta_2\to $ in
\eqref{RU-drag-reg}. Theorem~\ref{theo:main} 
is proved in Section~\ref{sec:tore}, where we let $r_0,r_1\to 0$ and
$\ell\to \infty$ (with possibly $\eps\to
0$). Section~\ref{sec:korteweg} is devoted to the proof of
Proposition~\ref{prop:korteweg} ($\nu=0$, $\eps>0$). In an appendix,
we give more details 
about the derivation of an identity appearing in
Section~\ref{sec:tore}.

\section{Construction of solutions to the regularized system}\label{sec:reg}
We start this study by constructing weak solutions to the system 
\eqref{RU-drag-reg} on the torus $\T_\ell^d$ with strictly positive
densities and deriving 
further properties satisfied by these solutions. 
We recall that in system \eqref{RU-drag-reg} the parameters $r_0, r_1 , \eps >0$ are positive, which will be hence assumed through this section.

System \eqref{RU-drag-reg} is endowed with some estimates.
 We first note that, integrating \eqref{RU-drag-reg1} we obtain the conservation of mass:
 \begin{equation} \label{mass_conservation}
\int_{\T^d_\ell}R(t) = \int_{\T^d_\ell}R_0 .
\end{equation}
Then,  by multiplying formally \eqref{RU-drag-reg2} with $U/\tau^2$ and combining with equation \eqref{RU-drag-reg1}, we obtain that reasonable solutions to \eqref{RU-drag-reg} should satisfy
the energy estimate: 
\begin{equation} \label{RU-drag-reg-estimation}
\begin{aligned}
\frac{\dd}{\dd t} \Ereg(R,U) + \Dreg(R,U) = \frac{2d \delta_1}{\tau^2}\int_{\T^d_\ell} R - \frac{\nu \dot \tau}{\tau^3} \int_{\T^d_\ell} R \Div U,
\end{aligned}
\end{equation}
where 
\begin{equation} \label{def_energie-reg}
\begin{aligned}
\Ereg(R,U) 
&= \frac{1}{2\tau^2} \int_{\T^d_\ell} \left( R|U|^2 + \eps^2 |\nabla \sqrt{R}|^2 \right)
+ \int_{\T^d_\ell} \left( R|y|^2 + R\log R + \frac{\eta_1}{\alpha+1} R^{-\alpha} \right)\\
&\quad
+ \frac{\eta_2}{2\tau^2} \int_{\T^d_\ell} |\nabla \Delta^s R|^2,
\end{aligned}
\end{equation}
and
\begin{equation*} 
\begin{aligned}
\Dreg(R,U) 
&= \frac{\dot \tau}{\tau^3} \int_{\T^d_\ell}\left( R|U|^2 + \eps^2 |\nabla \sqrt{R}|^2 +\eta_2 |\nabla \Delta^s R|^2\right)
 + \frac{\nu}{\tau^4} \int_{\T^d_\ell} R |\D U|^2 \\
&\quad
+ \frac{\delta_2}{\tau^4}   \int_{\T^d_\ell} |\Delta U|^2
+ \frac{\delta_1 \eta_2}{\tau^4} \int_{\T^d_\ell} |\Delta^{s+1} R|^2 
+ \frac{4\delta_1}{\tau^2}  \int_{\T^d_\ell} |\nabla \sqrt{R_N}|^2
\\
& \quad + \frac{4\delta_1 \eta_1}{ \alpha \tau^2} \int_{\T^d_\ell} |\nabla R^{-\alpha/2}|^2 
+ \frac{r_0}{\tau^4} \int_{\T^d_\ell} |U|^2
+ \frac{r_1}{\tau^4} \int_{\T^d_\ell}  R |U|^4 \\
& \quad + \frac{\delta_1 \eps^2}{2\tau^4} \int_{\T^d_\ell} R| \nabla^2 \log R|^2.
\end{aligned}
\end{equation*}
Note that the term appearing on the last line is obtained thanks to the exact formula: 
\[
\frac12 \int R |\nabla^2 \log R|^2 = \int \frac{\Delta \sqrt{R}}{\sqrt{R}} \Delta R.
\]
On the other hand, multiplying formally \eqref{RU-drag-reg1} by
a smooth function $\Psi$ and \eqref{RU-drag-reg2} by a smooth vector field $\Phi $ yields respectively 
\begin{equation}\label{RU-drag-reg1-weak}
\int_{\T^d_\ell} R_0 \Psi(0) 
+\int_0^T \!\!\! \int_{\T^d_\ell} R \partial_t \Psi   
+  \int_0^T \!\!\! \int_{\T^d_\ell} \frac{1}{\tau^2} R U \cdot \nabla \Psi 
+  \delta_1  \int_0^T \!\!\! \int_{\T^d_\ell}  \frac{1}{\tau^2}   R \Delta  \Psi  =0,
\end{equation}
and
\begin{equation}\label{RU-drag-reg2-weak}
\begin{aligned}
&\int_{\T^d_\ell} R_0 U_0 \Phi(0) 
+\int_0^T \!\!\! \int_{\T^d_\ell} RU \cdot \partial_t \Phi  
+  \int_0^T \!\!\! \int_{\T^d_\ell} \frac{1}{\tau^2} R U \otimes U : \nabla \Phi \\
&\qquad\quad
=  \int_0^T \!\!\! \int_{\T^d_\ell} R (2y \cdot \Phi - \Div \Phi)   
+  r_0 \int_0^T \!\!\!  \int_{\T^d_\ell}  \frac{1}{\tau^2}   U \cdot \Phi 
+  r_1 \int_0^T \!\!\! \int_{\T^d_\ell}  \frac{1}{\tau^2}  R|U|^2  U \cdot \Phi \\
&\qquad\qquad 
+ \eps^2 \int_0^T \!\!\! \int_{\T^d_\ell}  \frac{1}{2\tau^2} \left[ 2 \Delta \sqrt R \nabla \sqrt R \cdot \Phi + \Delta \sqrt R \sqrt R \Div \Phi \right] \\
&\qquad\qquad  
+ \nu  \int_0^T \!\!\! \int_{\T^d_\ell} \frac{1}{\tau^2}  R \D U : \nabla \Phi   
+   \nu  \int_0^T \!\!\! \int_{\T^d_\ell}  \frac{\dot \tau}{\tau}  R \Div \Phi \\
&\qquad\qquad
+  \delta_1  \int_0^T \!\!\! \int_{\T^d_\ell}  \frac{1}{\tau^2}  \nabla U : \nabla R \otimes \Phi   
+ \delta_2  \int_0^T \!\!\! \int_{\T^d_\ell} \frac{1}{\tau^2}  \Delta U \cdot \Delta \Phi \\
&\qquad\qquad  
+ \eta_1  \int_0^T \!\!\! \int_{\T^d_\ell}  R^{-\alpha} \Div \Phi   
+ \eta_2  \int_0^T \!\!\!  \int_{\T^d_\ell} \frac{1}{\tau^2} \Delta^{s+1} R \Delta^s
\left[ \nabla R \cdot \Phi + R \Div \Phi \right].
\end{aligned}
\end{equation}
So, to define weak solutions to \eqref{RU-drag-reg}, we look for minimal regularity assumptions that 
are induced by energy estimate \eqref{RU-drag-reg-estimation} and which make \eqref{RU-drag-reg1-weak}-\eqref{RU-drag-reg2-weak} meaningful for smooth test-functions. 
For this, we first recall the following lemma -- which is reminiscent of \cite[Lemma~2.1]{BD06b} with a slightly different statement -- 
to estimate negative power of the density which naturally appear in the formulation \eqref{RU-drag-reg}:

\begin{lemma}\label{lem:lowerbound}
For $n \in \N^*$ and $\Omega= \T^d$ or $\Omega = \R^d$, there holds
\[
\| \nabla^n (f^{-1}) \|_{L^2(\Omega)}
\lesssim \left( 1+\| f^{-1}\|_{L^4(\Omega)} + \| f^{-1} \|_{L^{2(n+1)}(\Omega) } \right)^{n+1}  \left( 1+ \| f \|_{H^{\sigma} (\Omega)}  \right)^{n}
\]
with $\sigma > n + d/2$.
\end{lemma}
\begin{proof}
Recall the embedding $H^{d/2+0}(\Omega) \hookrightarrow L^\infty (\Omega)$.
We compute
\[
|\nabla^n (f^{-1})|^2 
\lesssim 
\sum_{j=1}^{n} \sum_{i_1 + \cdots + i_j = n} \frac{|\nabla^{i_1} f|^2
  \cdots |\nabla^{i_j} f|^2}{f^{2(j+1)}},
\]
hence, for any $j \geq 1,$ we have:
\[
\begin{aligned}
\int \frac{|\nabla^{i_1} f|^2 \cdots |\nabla^{i_j} f|^2}{f^{2(j+1)}}\, \dd x 
&\lesssim \| \nabla^{i_1} f \|_{L^\infty(\Omega)}^2 \cdots \|
\nabla^{i_j} f \|_{L^\infty(\Omega)}^2 \int f^{-2(j+1)} \,  \dd x \\ 
&\lesssim  
\| f \|_{H^{\sigma}(\Omega}^{2j} \,
\| f^{-1} \|_{L^{2(j+1)}}^{2(j+1)} \\
&\lesssim  
\left( 1+ \| f \|_{H^{\sigma}(\Omega)} \right)^{2n}
\left(1+ \| f^{-1} \|_{L^{2(j+1)}(\Omega)} \right)^{2(n+1)} \\
&\lesssim  
\left( 1+ \| f \|_{H^{\sigma}(\Omega)} \right)^{2n}
\left(1+ \| f^{-1} \|_{L^{4}(\Omega)} + \| f^{-1} \|_{L^{2(n+1)}(\Omega)} \right)^{2(n+1)},
\end{aligned}
\]
which completes the proof.
\end{proof}
Since $\Ereg$ enables to control the $H^{2s+1}$-norm of $R$ together with the mean of $R^{-\alpha}$, we may infer that, for $\alpha >4$ and 
$s > d$, the energy estimate \eqref{RU-drag-reg-estimation} implies that $1/R$ is continuous. We also recall that 
the Laplace equation on the torus enjoys classical elliptic estimate
so that the dissipation $\Dreg$ (note that $r_0 , \delta_2>0$) yields $U \in L^2_{\mathrm{loc}}(\mathbb R^+;H^2(\mathbb T^d_{\ell})).$ 
Introducing the regularity expected for $R$ and $U$ into the continuity equation \eqref{RU-drag-reg1} entails that $\partial_t R \in L^2_{\mathrm{loc}}(\mathbb R^+;H^1(\mathbb T^d_{\ell})).$ Then, our definition of weak solution to \eqref{RU-drag-reg} reads as follows: 
\begin{definition}\label{def:weak-sol}
Given $(R_0,U_0) \in L^{1}(\mathbb T^d_{\ell}) \times L^{2}(\mathbb T^d_{\ell}),$ we say that $(R,U)$ is a global weak solution to \eqref{RU-drag-reg} associated to the initial data $(R_0,U_0) $ if we have:
\begin{enumerate}[label={\normalfont(\roman*)}]
\item $(R,U)$ satisfies
\begin{equation}\label{estim-def}
\begin{aligned}
& R \in H^{1}_{\mathrm{loc}}(\mathbb R^+ ; H^{1}(\mathbb T^d_{\ell}))
\cap  C(\mathbb R^+ ;H^{2s}(\mathbb T^d_{\ell})) \cap
L^2_{\mathrm{loc}}(\mathbb R^+ ; H^{2s+2}(\mathbb T^d_{\ell}))  \text{
  with } 1/R \in C(\mathbb R^+ \times \mathbb T^d_{\ell}),  \\
& U \in L^{\infty}_{\mathrm{loc}}(\mathbb R^+ ; L^2(\mathbb T^d_{\ell})) \cap L^2_{\mathrm{loc}}(\mathbb R^+ ; H^2(\mathbb T^d_{\ell})).
\end{aligned}
\end{equation}

\item Equation \eqref{RU-drag-reg1-weak} holds true for any $\Psi 
 \in \DD ([0,\infty) \times \T^d_\ell )$.
\item Equation \eqref{RU-drag-reg2-weak} holds true for any 
$\Phi \in \DD ([0,\infty) \times \T^d_\ell )^d.$
\end{enumerate}
\end{definition}

\begin{remark}
Thanks to the above remarks, the regularity statement (i) is
sufficient to obtain that all the terms in
\eqref{RU-drag-reg1-weak}-\eqref{RU-drag-reg2-weak} are well-defined. 
\end{remark}

In this section, we restrict to initial data with smooth and strictly positive density. This means that we shall assume that
$(R_0,U_0)$ satisfy:
\begin{equation}\label{RUinitial-theta2}
 R_0 \in \mathcal D(\mathbb T^d_{\ell}), \quad U_0 \in L^2(\mathbb T^d_{\ell}), \quad \inf_{y \in \T^d_\ell} R_0(y) \ge \theta >0 .
\end{equation}

The first main result of this section is the following proposition:

\begin{proposition}\label{prop:sol-approche}
Given initial data $(R_0,U_0)$ satisfying \eqref{RUinitial-theta2}, there exists a global solution $(R,U)$ to \eqref{RU-drag-reg} associated to $(R_0,U_0)$ on the torus $\T^d_\ell$, which satisfies moreover the 
conservation of mass~\eqref{mass_conservation} 
and the energy estimate
 \begin{equation} \label{RU-drag-reg-estimation.1}
\Ereg(R,U)(T) +  \int_0^T \Dreg(R,U)(s) \dd s
 \leq C_0 (\Ereg(R_0,U_0))   ,  \quad \text{for a.e.\ } T \geq 0,
 \end{equation} 
 for some constant $C_0>0$ depending on $\Ereg(R_0,U_0)$.

\end{proposition}

\begin{remark} \label{rem:regularite-weaksol-1}
We note that the energy estimate \eqref{RU-drag-reg-estimation.1}
together with \eqref{mass_conservation} entail that the solution 
we construct enjoys the following regularity properties, with norms
corresponding to these spaces bounded with respect to $\Ereg(R_0,U_0)$ only: 
\[
\begin{aligned}
& R (1+|y|^2 + |\log R|) \in L^\infty_{\mathrm{loc}}(\mathbb R^+; L^1(\mathbb T^d_{\ell})), && 
\sqrt{R} U  \in L^\infty_{\mathrm{loc}}(\mathbb R^+ ; L^2(\mathbb T^d_{\ell})), \\
& \sqrt{\nu} \, \sqrt{R} \D U  \in L^2_{\mathrm{loc}}(\mathbb R^+ ;  L^2(\mathbb T^d_{\ell})), &&
\eps \nabla \sqrt{R} \in L^\infty_{\mathrm{loc}}(\mathbb R^+ ; L^2(\mathbb T^d_{\ell})),\\  
& \sqrt{r_0}\, U  \in L^2_{\mathrm{loc}}(\mathbb R^+ ;L^2(\mathbb T^d_{\ell})), &&
\sqrt{r_1}\, R^{\frac{1}{4}} U \in  L^4_{\mathrm{loc}}(\mathbb R^+ ; L^4(\mathbb T^d_{\ell})) , \\
& \sqrt{\delta_2} \, \Delta U  \in L^2_{\mathrm{loc}}(\mathbb R^+ ; L^2(\mathbb T^d_{\ell})), &&
\sqrt{\eta_2} \, R  \in L^\infty_{\mathrm{loc}}(\mathbb R^+ ; H^{2s+1}(\mathbb T^d_{\ell})), \\
& \eta_1^{\frac{1}{\alpha}}\, R^{-1} \in L^\infty_{\mathrm{loc}}(\mathbb R^+ ; L^{\alpha}(\mathbb T^d_{\ell})), &&
\sqrt{\delta_1 \eta_1}\, \nabla R^{-\frac{\alpha}{2}} \in L^2_{\mathrm{loc}}(\mathbb R^+ ; L^2(\mathbb T^d_{\ell})), \\
& \sqrt{\nu \eps^2 }\, \nabla^2 \sqrt{R} \in L^2_{\mathrm{loc}}(\mathbb R^+ ; L^2(\mathbb T^d_{\ell})), &&
(\nu \eps^2 )^{\frac{1}{4}}\, \nabla R^{\frac{1}{4}} \in L^4_{\mathrm{loc}}(\mathbb R^+ ; L^4(\mathbb T^d_{\ell})),\\
&\sqrt{\delta_1\eta_2}\, \Delta^{s+1} R  \in L^2_{\mathrm{loc}}(\mathbb R^+ ; L^2(\mathbb T^d_{\ell})).
\end{aligned}
\]
We refer to \eqref{eq:equivJungel} for the regularity claim on the before-last line. Also, combining these bounds with Lemma \ref{lem:lowerbound}, we obtain that, for arbitrary $T>0, $ there exists a $C(\Ereg(R_0,U_0),\eta_1,\eta_2,\theta,T)>0$ so that
\begin{equation}\label{estim-lowerbdd}
\| 1/R \|_{L^\infty ((0,T) \times \T^d_\ell)} \le C(\Ereg(R_0,U_0),\eta_1, \eta_2, \theta,T) .
\end{equation}
\end{remark}

\medskip

The proof of Proposition \ref{prop:sol-approche} is the content of the
next subsection. Then in the last subsection, we focus on a further estimate satisfied by the weak solutions that we construct.

\subsection{Proof of Proposition \ref{prop:sol-approche}.}
\label{sec:regularised} 

The plan of the proof follows closely the method of \cite{VasseurYu}.
In the whole section $(R_0,U_0)$ is a fixed initial data satisfying \eqref{RUinitial-theta2}.

\subsubsection{Faedo-Galerkin approximation} Let $X_N = \mathrm{span}
\{ e_1 , \dots , e_N \}$ be the finite-dimensional space corresponding
to the projection in $L^2(\mathbb T^d_{\ell})$ onto the  first $N$
Fourier modes. We consider the system whose unknowns are  
\[
(R_N,U_N) \in C(\mathbb R^+;H^{2s+1}(\mathbb T^d_{\ell})) \times  C(\mathbb R^+;X_N),
\]
and composed by \eqref{RU-drag-reg1} and the following weak formulation of \eqref{RU-drag-reg2}: for any $t \in (0,T)$ and any vector field $\phi \in (X_N)^d$,
\begin{equation}\label{FG}
\begin{aligned}
& \dfrac{\textrm{d}}{\textrm{d}t}\int_{\T^d_\ell} R_NU_N \cdot \phi  
- \frac{1}{\tau^2}  \int_{\T^d_\ell} R_N U_N \otimes U_N : \nabla \phi 
+  \int_{\T^d_\ell} R_N (2y \cdot \phi - \Div \phi)  \\
&\qquad 
+  \frac{r_0}{\tau^2} \int_{\T^d_\ell}   U_N \cdot \phi \, d y
+  \frac{r_1}{\tau^2} \int_{\T^d_\ell} R_N|U_N|^2  U_N \cdot \phi 
+  \frac{\delta_1}{\tau^2}  \int_{\T^d_\ell}  ([\nabla R_N \cdot \nabla] U_N ) \cdot \phi  \\
&\qquad 
+ \frac{\eps^2}{2\tau^2}  \int_{\T^d_\ell}   \left[ 2 \Delta \sqrt R_N \nabla \sqrt R_N \phi + \Delta \sqrt R_N \sqrt R_N \Div \phi \right] 
+ \frac{\nu}{\tau^2}  \int_{\T^d_\ell}  R_N \D U_N : \nabla \phi  \\
&\qquad 
+   \frac{\nu\dot \tau}{\tau} \int_{\T^d_\ell}   R_N \Div \phi 
+  \frac{\delta_2}{\tau^2} \int_{\T^d_\ell}   \Delta U_N \cdot \Delta \phi 
+ \eta_1  \int_{\T^d_\ell}  R_N^{-\alpha} \Div \phi  
\\
& \quad - \frac{\eta_2}{\tau^2} \int_{\T^d_\ell}   R_N \nabla \Delta^{2s+1} R_N \cdot \phi    =0,
\end{aligned}
\end{equation}
where we recall that $r_0,r_1, \eps >0$.
We complement the system with initial conditions:
\begin{equation} \label{FG-ini}
\begin{aligned}
& R_N|_{t=0} = R_0, \\
& \left[\int_{\mathbb T^d_{\ell}} R_NU_N \cdot \phi\right]|_{t=0} = \int_{\mathbb T^d_{\ell}} R_0U_0 \cdot \phi ,\qquad \forall \phi \in (X_N)^d.
\end{aligned}
\end{equation}
We have the following existence result for this approximate system:
\begin{proposition}\label{prop:sol-1}
Given $N \in \mathbb N^*,$ there exists a global solution $(R_N,U_N)$ to \eqref{RU-drag-reg1}-\eqref{FG}-\eqref{FG-ini} that satisfies the conservation of mass \eqref{mass_conservation}
and the energy inequality
\begin{equation}\label{eq:EregN}
\sup_{t \in (0,T)} \Ereg(R_N,U_N)
+ \int_0^T \Dreg(R_N,U_N) \, d t \le C(\Ereg(R_N,U_N)|_{t=0}) ,
\end{equation}
for come constant $C>0$ depending on $\Ereg(R_N,U_N)|_{t=0}$.
\end{proposition}

\begin{proof}
The local existence is obtained following \cite{VasseurYu}
(see also \cite{Jungel}). The novelties with respect to this previous study are: the linearity of the pressure term, the time factors $\tau, \dot{\tau}$ and the new terms
\[
 \int_{\T^d_\ell} R (2y \cdot \phi - \Div \phi), \qquad   \frac{r_0}{\tau^2} \int_{\T^d_\ell}   U \cdot \phi , \qquad  \frac{\nu\dot \tau}{\tau} \int_{\T^d_\ell}   R \Div \phi .
\]
However, these terms are harmless in the fixed-point approach of \cite[Section 2]{VasseurYu}, for instance.

\medskip

The global existence is then a consequence of the energy estimate that we obtain as follows.
Conservation of mass follows by integrating \eqref{RU-drag-reg1}.
We may then take $\phi = U_N(t)/\tau^2(t)$ in \eqref{FG} 
since it corresponds to writing the $N$ equations obtained by setting $\phi = e_j,$ $j=1,\ldots,N$, and combining them with the coefficients defining $U_N$ in this basis. 
This yields  
\begin{equation}\label{dEdt}
\begin{aligned}
\frac{\dd}{\dd t} \Ereg(R_N,U_N) + \Dreg(R_N,U_N) = \frac{2d \delta_1}{\tau^2} \int_{\T^d_\ell} R_N - \frac{\nu \dot \tau}{\tau^3} \int_{\T^d_\ell} R_N \Div U_N.
\end{aligned}
\end{equation}
We deduce the energy inequality by remarking that the right-hand side of \eqref{dEdt} can be bounded by 
\[
\left( \frac{2d \delta_1}{\tau^2} + C \frac{\nu \dot \tau^2}{\tau^2}
\right) \int_{\T^d_\ell} R_N 
+\frac{\nu}{2\tau^4} \int_{\T^d_\ell} R_N |\D U_N|^2
\le C \frac{(1+\dot \tau^2)}{\tau^2} \int_{\T^d_\ell} R_N  + \frac12 \Dreg(R_N,U_N),
\]
using the conservation of mass together with 
\[
\int_0^\infty \frac{1+\dot \tau^2(t)}{\tau^2(t)} \, \dd t < \infty,
\]
and recalling that $\Ereg$ is nonnegative.
\end{proof}

\subsubsection{Convergence of the approximate solutions}
\label{sec:convergenceapproximate}

We split the proof into three steps: defining limits to the sequence 
of approximate solutions $(R_N,U_N),$ improving the sense in which this sequence converges, passing to the limit in the weak formulation \eqref{FG}. In all the convergences mentioned in the proof, we have to extract subsequences that we do not relabel for conciseness.\\

So,  let $\{ (R_N,U_N) \}_N$ be the sequence of approximate solutions to \eqref{RU-drag-reg1}-\eqref{FG}-\eqref{FG-ini} given by Proposition~\ref{prop:sol-1}. We note that we have initially 
 $R_N(0,\cdot) = R_{0}$ and $R_{N}U_N(0,\cdot) = \mathbb P_{N}[R_0 U_0]$
 where $\mathbb P_N$ stands for the ($L^2(\mathbb
 T^d_{\ell})$)-projection onto $X_N.$ In particular, since by
 assumption $R_0U_0 \in L^2(\mathbb T^{d}_{\ell}),$ we have 
 \begin{equation} \label{eq:controluniform1}
 \Ereg(R_N,U_N)|_{t=0} \leq \Ereg(R_0,U_0).
 \end{equation}

\medskip\noindent
\textit{Step 1.}
From \eqref{eq:controluniform1} and the energy inequality derived in Proposition~\ref{prop:sol-1}, we infer that 
\[
\sup_{t \ge 0} \Ereg(R_N,U_N) + \int_0^{\infty} \Dreg(R_N,U_N)  \leq C( \Ereg(R_0,U_0) ), \quad \forall \, N.  
\]
We obtain then uniform bounds 
on $(R_N,U_N)$ in a series of spaces similar to the ones in
Remark~\ref{rem:regularite-weaksol-1}. We first  extract from this list
that we have uniform bounds with respect to $N$ for:
\[
\begin{aligned}
& \frac{1}{\tau}  \sqrt{\eta_2} R_N \text{ in $L^{\infty}(\mathbb R^+; H^{2s+1}(\mathbb T^{d}_{\ell}))$}, &
& \left(\dfrac{\eta_1}{\alpha+1}\right)^{\frac 1\alpha}\dfrac{1}{R_N} \text{ in $L^{\infty}(\mathbb R^+; L^{\alpha}(\mathbb T^{d}_{\ell}))$},\, \\    
& \frac{1}{\tau} \sqrt{R_N} U_N \text{ in $L^{\infty}(\mathbb R^+; L^{2}(\mathbb T^d_{\ell}))$}. 
\end{aligned}
\] 
Using the first bound, we can extract a subsequence so that $R_N/\tau$ 
converges to some $R/{\tau}$ in this same space (for the weak-$*$ topology). From the last bound,
we obtain that (up to the extraction of a subsequence) $\sqrt{R_N} U_N / \tau $ converges to some $V/\tau$ in $L^{\infty}(\mathbb R^+;L^2(\mathbb T^{d}_{\ell}))-w*.$ 
Restricting to any time interval $(0,T)$ with $T<\infty,$ the second
bound with the first one and Lemma \ref{lem:lowerbound} imply that
$R_N$ is uniformly bounded from below on $(0,T)$ by a constant
$C(\Ereg(R_0,U_0),\eta_1,\eta_2,\theta,T).$ Hence, we have also  
\begin{equation}\label{Rminore}
R \geq C(\Ereg(R_0,U_0),\eta_1,\eta_2,\theta,T) \quad\text{in} \quad (0,T),
\end{equation}
and we may set $U = V/\sqrt{R}.$  
We focus now on the restriction of these limits on  $(0,T).$ 

\medskip\noindent
\textit{Step 2.}
On $(0,T),$ we establish convergences of $R_N$ and $U_N$ in
a stronger sense.

\medskip

To this end, we now extract from the list given by  Remark~\ref{rem:regularite-weaksol-1} uniform bounds for
\[
\begin{aligned}
& \text{  $R_N$  in $L^{\infty}(0,T;H^{2s+1}(\mathbb T^d_{\ell})) \cap L^2(0,T;H^{2s+2}(\mathbb T^d_{\ell})),$}\\
&\text{ $1/R_N$  in $L^{\infty}((0,T) \times \mathbb T^{d}_{\ell}),$}\\
& \text{ $U_N$ in $L^2(0,T;H^2(\mathbb T^d_{\ell})).$}
\end{aligned}
\]
The continuity equation \eqref{RU-drag-reg1} satisfied by $R_N$ implies then that $\partial_t R_N$ is bounded in $L^2(0,T;H^1(\mathbb T^{d}_{\ell})).$ 
Combining classical weak-convergence results and Ascoli-Arzel\`a type arguments entails that:
\begin{equation} \label{eq:convergence-reg-R}
\begin{aligned}
& R_N \to R  \text{ in $C([0,T]; H^{2s}(\mathbb T^d_{\ell})),$} \\
& R_N \rightharpoonup  R  \text{ in $L^2(0,T;H^{2s+2}(\mathbb T^d_{\ell}))-w,$}\\
& R_N \rightharpoonup R \text{ in $H^1(0,T;H^{1}(\mathbb T^d_{\ell}))-w.$}
\end{aligned}
\end{equation}
Given the bound by below on $R_N$ \eqref{Rminore}, we also have that 
$1/R_N$ converges to $1/R$ in $C([0,T] \times \mathbb T^d_{\ell}).$

\smallbreak

Next,  given the uniform bounds for $U_N$ and $R_N,$ and since  $(e_k)_{k\in\mathbb{N}}$
is orthogonal for the $H^2$-scalar product, we have that $R_NU_N$
and $\mathbb P_N[R_NU_N]$ are uniformly bounded  in
$L^2(0,T;H^2(\mathbb T^d_{\ell}))$ too.  
On the other hand, the weak formulation satisfied by the approximation $(R_N,U_N)$
reads:
\[
\begin{aligned}
\partial_t( \mathbb{P}_N[R_N U_N]) 
&= \mathbb P_N\Biggl[  -\frac{1}{\tau^2} \Div (R_N U_N \otimes U_N)
-2 y R_N  - \nabla R_N + \eta_1 \nabla R_N^{-\alpha} +\frac{r_0}{\tau^2} U_N \\
&\quad
+ \frac{r_1}{\tau^2} R_N |U_N|^2 U_N + \frac{\delta_1}{\tau^2}(\nabla R_N \cdot \nabla) U_N
+\frac{\eps^2}{2\tau^2} R_N \nabla \( \frac{\Delta  \sqrt{R_N}}{\sqrt{R_N}}\) \\  
&\quad
+\frac{\nu}{\tau^2} \Div (R_N \D U_N) + \frac{\nu \dot \tau}{\tau} \nabla R_N 
+\frac{\delta_2}{\tau^2} \Delta^2 U_N  + \frac{\eta_2}{\tau^2} R_N \nabla \Delta^{2s+1} R_N \Biggr]\\
& := \mathbb P_N [F_N]
\end{aligned}
\]
Again we note here that $\mathbb P_N$ is orthogonal with respect to
the $H^s$-scalar product, so that
\[
\|\mathbb P_N F_N\|_{H^{-s}(\mathbb T^d_{\ell})} \leq
\|F_N\|_{H^{-s}(\mathbb T^d_{\ell})}  ,
\quad \forall \, s \in \mathbb N.
\]
For $s$ sufficiently large, we may then combine the various uniform estimates satisfied by $(R_N,U_N)$ on 
$(0,T)$ to infer that $\partial_t (\mathbb P_N[R_NU_N])$ is uniformly bounded in {$L^2(0,T;H^{-(2s+2)}(\mathbb T^d_{\ell})).$}  To prove this, the main terms to be discussed are $\Div (R_NU_N \otimes U_N)$ and 
$R_N |U_N|^2 U_N$ which can be handled (since $d\le 3$) via the embedding $H^2(\mathbb T^d_{\ell}) \subset L^{\infty}(\mathbb T^d_{\ell}).$  
To summarize, we know that $\mathbb P_N[R_NU_N]$ is bounded in $L^2(0,T;H^2(\mathbb T^d_{\ell}))$ and 
$\partial_t (\mathbb P_N[R_NU_N])$ is bounded in {$L^2(0,T;H^{-(2s+2)}(\mathbb T^d_{\ell})).$}
Aubin-Lions like arguments imply then that $\mathbb P_N[R_NU_N]$ converges in $L^2(0,T;H^1(\mathbb T^{d}_{\ell}))$. 
Due to the compactness of the embedding $H^{2}(\mathbb T^d_{\ell})
\subset H^{1}(\mathbb T^d_{\ell})$ again,  there exists a sequence
$(\varepsilon_N)_{N}$ converging to $0$ so that 
 \[
 \|\mathbb P_N[R_NU_N] - R_NU_N\|_{L^2(0,T;H^1(\mathbb T^d_{\ell}))}
 \le \varepsilon_N \|R_NU_N\|_{L^2(0,T;H^2(\mathbb T^d_{\ell}))}.
 \]
 Consequently, $(\mathbb P_N[R_NU_N])_N$ and $(R_NU_N)_N$ both converge
  to $RU$ in $L^2(0,T;H^1(\mathbb T^d_{\ell})).$  Moreover, since $(1/R_N)_{N\in \mathbb N}$ is uniformly 
 bounded and $R_N$ converges to $R$ in a sufficiently regular space,
 this also implies  that 
\begin{equation} \label{eq:convergence-reg-U}
U_N  \to U \text{ in $L^2(0,T;H^1(\mathbb T^d_{\ell}))$}.
\end{equation}
To end up this part on the convergence of $U_N$, we note that the uniform estimates satisfied by $(R_N,U_N)$ also entail that $U_N$ is bounded in 
$L^{\infty}(0,T;L^2(\mathbb T^d_{\ell})) \cap L^2(0,T;H^2(\mathbb T^d_{\ell}))$ so that the limit $U$ lies in these spaces. 
 
 \medskip\noindent
\textit{Step 3.}
Given the time-regularity of approximate solutions, $R_N$ and $R_NU_N$ satisfy
\eqref{RU-drag-reg1-weak} for arbitrary $\Psi \in \DD ([0,\infty)
\times \T^d_\ell )$,  and \eqref{RU-drag-reg2-weak} for arbitrary 
$\Phi \in \DD ([0,\infty) ; X_N)$,   respectively.
The two sets of convergence results \eqref{eq:convergence-reg-R}  and \eqref{eq:convergence-reg-U} are then sufficient to pass to the limit in these weak formulations. Again, the main difficulty might be here to pass to the limit in $R_N |U_N|^2 U_N$.
However, we note that $R_N$ converges in the set of continuous
functions while $U_N$  
is bounded in $L^{\infty}_{\rm loc}((0,\infty);L^2(\mathbb T^d_{\ell}))$
and converges in $L^2_{\rm loc}((0,\infty);H^1(\mathbb T^d_{\ell}))$ so
that, by interpolation, it converges in
$L^{4}_{\rm loc}((0,\infty);L^3(\mathbb T^d_{\ell})).$  
At this point, $(R,U)$ satisfies
\eqref{RU-drag-reg1-weak} for arbitrary $\Psi \in \DD ([0,\infty) \times \T^d_\ell )$ and \eqref{RU-drag-reg2-weak} for arbitrary 
$\Phi \in \DD ([0,\infty) ; \bigcup_{N} X_N).$ We note then
that for arbitrary $\Phi \in\DD ([0,\infty) \times \mathbb T^d_{\ell})
,$  $\partial_t \mathbb P_N[\Phi]$ and $\mathbb P_N[\Phi]$
converge to $\partial_t \Phi$  in $C([0,\infty); L^2(\mathbb
T^{d}_{\ell}))$ and $\Phi$ in $L^2(0,\infty;H^{2s+2}(\mathbb
T^d_{\ell}))$, respectively. This is sufficient to extend  
\eqref{RU-drag-reg2-weak} to arbitrary $\Phi \in \DD ([0,\infty) \times \T^d_{\ell}).$ 

As for energy estimate, we note that $(R_N,U_N)$ satisfies
\eqref{eq:EregN} for arbitrary $N$ and the initial data verifies \eqref{eq:controluniform1}. Since $\Ereg(R_N,U_N)$ 
is continuous with respect to topologies for which $R_N,U_N$ converge strongly,
while $\Dreg(R_N,U_N)$ is continuous with respect to topologies for which
$R_N,U_N$ converge weakly, we obtain that $(R,U)$ satisfies
\eqref{RU-drag-reg-estimation.1} in the limit $N \to \infty.$ 
 This concludes the proof of Proposition \ref{prop:sol-approche}.

\begin{remark} \label{rem:timeregularity}
With arguments similar to the  ones in {\em Step 3}
of the above proof, we can extend  the weak form \eqref{RU-drag-reg2-weak} of the momentum equation to any test-function  
$\Phi \in (L^2(0,T ; H^{2s+1}(\mathbb T^d_{\ell}))^{d}$
having compact support and such that $\partial_t \Phi \in (L^2(0,T;L^2(\mathbb T^d_{\ell})))^d.$

\end{remark} 
 
\subsection{Further properties of weak solutions to the regularized problem}

Along with the energy estimate \eqref{RU-drag-reg-estimation.1}, we only
showed that we had a list of regularity properties satisfied by our
weak solutions $(R,U).$ Nevertheless, most of these estimates rely on
the regularization parameters $\eta_1,\eta_2,r_0,r_1,$ etc. In order to  
let these parameters vanish, we need other estimates on these solutions.
This is the motivation of the following lemma:

\begin{lemma}[BD-entropy]
Assume the initial data satisfies \eqref{RUinitial-theta2}.  
Then there exist constants $C_1, C_2, C_3$ with dependencies
mentioned in parentheses, such that, for arbitrary $T>0,$ the global
solution $(R,U)$ to \eqref{RU-drag-reg} constructed in Proposition~\ref{prop:sol-approche}
satisfies 
\begin{equation}\label{estim-entropy}
\begin{aligned}
&\sup_{t \in (0,T)}\EBDreg^+(R,U)(t)
+ \int_0^T \DBDreg(R,U) (t)\, \dd t \\
&\qquad\qquad 
\le C_1(\Ereginit, {\mathcal E}^+_{{\rm BD,reg}\mid t=0} ) 
+  (\delta_1 + \delta_2) C_2 \left(r_0 , r_1, \eta_1,\eta_2, 
  \Ereginit, T\right)
  +C_3(r_0),
\end{aligned}
\end{equation}
where $\EBDreg^+$ is the positive part of the BD-entropy defined by
\[
\begin{aligned}
\EBDreg^+(R,U) 
&= \frac{1}{2\tau^2} \int_{\T^d_\ell} 
\left( R|U + \nu \nabla \log R|^2 + \eps^2 |\nabla \sqrt{R}|^2 - 2 r_0 (\log R)\mathbf 1_{R\leq 1} \right)\\
&\quad
+ \int_{\T^d_\ell} \left( R|y|^2 + R \log R + \frac{\eta_1}{\alpha+1} R^{-\alpha} \right)
+ \frac{\eta_2}{2\tau^2} \int_{\T^d_\ell} |\nabla \Delta^s R|^2,
\end{aligned}
\]
and its associated nonnegative dissipation is given by
\[
\begin{aligned}
\DBDreg(R,U) 
&= \frac{\dot \tau}{\tau^3} \int_{\T^d_\ell} \left( R|U|^2 + \eps^2 |\nabla \sqrt{R}|^2 +\eta_2 |\nabla \Delta^s R_N|^2\right)
+\frac{2r_0 \nu \dot \tau}{\tau^3} \int_{\T^d_\ell} |\log R|  \, {\mathbf 1}_{R<1}\\
&\quad
+ \left( \frac{ \delta_1 \nu^2 }{\tau^4}  + \frac{\nu \eps^2}{\tau^4} 
+ \frac{\delta_1 \eps^2}{2\tau^4} \right) \int_{\T^d_\ell} R|\nabla^2  \log R|^2
+\left( \frac{4 \nu}{\tau^2} + \frac{4\delta_1}{\tau^2} \right) \int_{\T^d_\ell} |\nabla \sqrt{R}|^2 \\
&\quad
+\left( \frac{\eta_1 \nu \alpha}{4\tau^2} + \frac{4\delta_1 \eta_1}{10 \tau^2}   \right)\int_{\T^d_\ell} |\nabla R^{-\frac{\alpha}{2}}|^2 
+ \frac{\nu}{\tau^4} \int_{\T^d_\ell} R |\A U |^2 
+ \frac{(\eta_2 \nu + \delta_1 \eta_2)}{\tau^4} \int_{\T^d_\ell} |\Delta^{s+1}R|^2 \\
&\quad
+ \frac{\delta_2}{\tau^4}   \int_{\T^d_\ell} |\Delta U|^2
+ \frac{r_0}{\tau^4} \int_{\T^d_\ell} |U|^2 
+ \frac{r_1}{\tau^4} \int_{\T^d_\ell} R |U|^4.
\end{aligned}
\]
\end{lemma}

\begin{remark} \label{rem:BDentropie}
Below, we see the positive BD-entropy as the positive part of the
complete BD-entropy:
\[
\begin{aligned}
\EBDreg(R,U) 
&= \frac{1}{2\tau^2} \int_{\T^d_\ell} 
\left( R|U + \nu \nabla \log R|^2 + \eps^2 |\nabla \sqrt{R}|^2 - 2 r_0
  \log R \right)\\ 
&\quad
+ \int_{\T^d_\ell} \left( R|y|^2 + R\log R+ \frac{\eta_1}{\alpha+1} R^{-\alpha} \right)
+ \frac{\eta_2}{2\tau^2} \int_{\T^d_\ell} |\nabla \Delta^s R|^2, 
\end{aligned}
\]
and we note that we have then
\begin{align*}
  \EBDreg^+ = \EBDreg-\EBDreg^-,\quad \EBDreg^-= -\frac{r_0}{\tau^2} \int_{\T^d_\ell}  
 \log R \, {\mathbf 1}_{R>1}. 
\end{align*}

\end{remark}

\begin{proof}
We consider in this proof $(R,U)$ a weak solution to \eqref{RU-drag-reg} constructed 
in Proposition~\ref{prop:sol-approche}. 
We have
\[
\begin{aligned}
& \nabla R \in H^{1}_{\mathrm{loc}}(\mathbb R^+;L^2(\mathbb T^d_{\ell})) \cap L^{\infty}_{\mathrm{loc}}(\mathbb R^+;H^{2s-1}(\mathbb T^d_{\ell})) \cap L^2_{\mathrm{loc}}(\mathbb R^+ ; H^{2s+1}(\mathbb T^d_{\ell})),
\\
& 1/R \in H^{1}_{\mathrm{loc}}(\mathbb R^+;L^2(\mathbb T^d_{\ell}))
\cap L^{\infty}_{\mathrm{loc}}(\mathbb R^+ ; H^{2s}(\mathbb T^d_{\ell}))
\cap L^2_{\mathrm{loc}}(\mathbb R^+ ; H^{2s+2}(\mathbb T^d_{\ell})).
\end{aligned}
\]
For $s$ sufficiently large, we obtain  that
$\Phi = (\nu \nabla \log R) /\tau^2 $  satisfies:
\[ 
\Phi \in (L^2_{\mathrm{loc}}(\mathbb R^+; H^{2s+1}(\mathbb T^d_{\ell}))^{d},
\quad 
\partial_t \Phi \in L^2_{\mathrm{loc}}(\mathbb R^+;L^2(\mathbb T^d_{\ell})))^d.
\]
Hence, for arbitrary $\chi \in \mathcal D(0,\infty),$ we can take $\Phi = (\nu \nabla \log R)\chi /\tau^2 $ 
as a test function in the weak formulation of the momentum equation \eqref{RU-drag-reg2-weak}.   
Combining with a standard regularity estimate for \eqref{RU-drag-reg1},
we obtain that, in $\DD'((0,T)),$ there holds:
\begin{equation}\label{identite1}
\begin{aligned} 
& \dfrac{\textrm{d}}{\textrm{d}t} \left[\dfrac{\nu}{\tau^2} \int_{\mathbb T^d_{\ell}} RU \cdot \nabla \log R  \right]
+ \dfrac{2 \nu \dot{\tau}}{\tau^3} \int_{\mathbb T^d_{\ell}} RU \cdot \nabla \log R   \\
& \quad 
+ \dfrac{\eps^2\nu }{\tau^4} \int_{\mathbb T^d_{\ell}} R |\nabla^2\log(R)|^2  
+ \left( \dfrac{\nu}{\tau^2} - \dfrac{\nu^2 \dot{\tau}}{\tau^3} \right) \int_{\mathbb T^d_{\ell}} 4 |\nabla \sqrt{R}|^2  \\
& \quad 
+ \dfrac{4 \eta_1 \nu }{\alpha}  \int_{\mathbb T^d_{\ell}} \left| \nabla \sqrt{R^{-\alpha}}\right|^2 
+ \dfrac{\eta_2 \nu}{\tau^4}   \int_{\mathbb T^d_{\ell}} |\Delta^{s+1} R|^2 \\
& \qquad 
  = \dfrac{2d\nu}{\tau^2} \int_{\mathbb T^d_{\ell}} R 
- \dfrac{r_0 \nu}{\tau^4} \int_{\mathbb T^d_{\ell}} U \cdot \nabla \log R - \dfrac{r_1 \nu}{\tau^4} \int_{\mathbb T^d_{\ell}} |U|^2 U \cdot \nabla R \\
&  \qquad \qquad 
-  \dfrac{\nu^2}{\tau^4} \int_{\mathbb T^d_{\ell}} R \mathbb D U : \nabla^2 \log R  \\
& \qquad \qquad
- \dfrac{\delta_1 \nu}{ \tau^4}\int_{\mathbb T^{d}_{\ell}}\nabla U : \nabla R \otimes \nabla \log R  
- \dfrac{\delta_2\nu}{\tau^4}  \int_{\mathbb T^d_{\ell}} \Delta U \cdot \nabla \Delta \log R  \\
& \qquad \qquad 
- \dfrac{\delta_1 \nu}{ \tau^4}\int_{\mathbb T^{d}_{\ell}} \dfrac{\Delta R}{R} \Div (RU) 
+ \dfrac{\nu}{\tau^4} \int_{\mathbb T^d_{\ell}} R \nabla U:\nabla^{\top} U   .
\end{aligned}  
\end{equation}
The proof of this identity is mostly technical. More details are provided in Appendix \ref{sec:identite1}.
On the other hand, differentiating the continuity equation \eqref{RU-drag-reg1} we obtain:
\[
\partial_t ( R \nabla \log R  ) +  \dfrac{1}{\tau^2} \Div (R \nabla \log R \otimes U ) +  \dfrac{1}{\tau^2}  \Div(R \nabla^{\top} U)  =  \dfrac{\delta_1}{\tau^2} \Delta \nabla R.
\]
This identity holds in $L^2_{\mathrm{loc}}(\mathbb R^+ ; L^2(\mathbb T^d_{\ell}))$ so, we can multiply it with a truncation of $\nabla \log R/\tau^2.$
This leads to the energy estimate:
\begin{equation} \label{identite2.0}
\begin{aligned}
& \frac{\dd}{\dd t} \left[ \frac{1}{2 \tau^2} R |\nabla \log R|^2  \right]  + \dfrac{\dot{\tau}}{\tau^3} \int_{\mathbb T^d_{\ell}} R |\nabla \log R|^2 
+\dfrac{\delta_1}{2 \tau^4}  \int_{\mathbb T^d_{\ell}}  \Delta R |\nabla \log R|^2  \\
& \qquad \quad =  \dfrac{1}{\tau^4} \int_{\mathbb T^d_{\ell}} R \nabla U : \nabla^{2} \log R +  \dfrac{\delta_1}{\tau^4} \int_{\mathbb T^d_{\ell}} \Delta \nabla R \cdot \nabla \log R.
\end{aligned}
\end{equation}
In this last identity, we note that:
\begin{align*}
\int_{\mathbb T^d_{\ell}} \Delta \nabla R \cdot \nabla \log R&  = - \int_{\mathbb T^d_{\ell}} \nabla^2 R : \nabla^2 \log R = -\int_{\mathbb T^d_\ell} \nabla ( R \nabla \log R) : \nabla \log R, \\
& = - \dfrac{1}{2}\int_{\mathbb T^d_{\ell}} \nabla R \cdot \nabla |\nabla \log R|^2  - \int_{\mathbb T^d_\ell} R |\nabla^2 \log R|^2 \\
& = \dfrac{1}{2} \int_{\mathbb T^d_{\ell}} \Delta R  |\nabla \log R|^2 -  \int_{\mathbb T^d_{\ell}} R |\nabla^2 \log R|^2 .
\end{align*}
Consequently, we rewrite the previous energy identity \eqref{identite2.0} as:
\begin{equation} \label{identite2}
\begin{aligned}
& \frac{\dd}{\dd t} \left[ \frac{1}{2 \tau^2} R |\nabla \log R|^2  \right]  + \dfrac{\dot{\tau}}{\tau^3} \int_{\mathbb T^d_{\ell}} 4 |\nabla \sqrt{R}|^2 
+  \dfrac{\delta_1}{\tau^4} \int_{\mathbb T^d_\ell} R |\nabla^2 \log R|^2   \\
& \qquad \quad =  \dfrac{1}{\tau^4} \int_{\mathbb T^d_{\ell}} R \nabla U : \nabla^{2} \log R .
\end{aligned}
\end{equation}
At this point, we combine \eqref{identite1}$+ \nu^2$\eqref{identite2}, which yields
\[
\begin{aligned}
&\frac{\dd}{\dd t} \left\{ \frac{1}{\tau^2} \int_{\mathbb T^d_{\ell}}  \left( \nu RU \cdot \nabla \log R 
+ \frac{\nu^2}{2} R |\nabla \log R|^2\right)   \right\} + \frac{2\nu \dot \tau}{\tau^3}\int_{\mathbb T^d_{\ell}} RU \cdot \nabla \log R  \\
&\qquad
+  \frac{4 \nu}{\tau^2} \int_{\mathbb T^d_{\ell}} |\nabla \sqrt{R}|^2  + 
\left( \frac{ \delta_1 \nu^2 }{\tau^4} + \frac{\eps^2
    \nu}{\tau^4}\right)
\int_{\mathbb T^d_{\ell}} R |\nabla^2 \log{R}|^2
\\
&\qquad
+ \frac{4 \eta_1 \nu }{\alpha\tau^2}\int_{\mathbb T^d_{\ell}} |\nabla R^{-\frac{\alpha}{2}}|^2 
+ \frac{\eta_2 \nu}{\tau^4} \int_{\mathbb T^d_{\ell}} |\Delta^{s+1}R|^2 \\
&\qquad
=  \dfrac{2d\nu}{\tau^2} \int_{\mathbb T^d_{\ell}} R 
- \dfrac{r_0 \nu}{\tau^4} \int_{\mathbb T^d_{\ell}} U \cdot \nabla \log R - \dfrac{r_1 \nu}{\tau^4} \int_{\mathbb T^d_{\ell}} |U|^2 U \cdot \nabla R \\
&  \qquad \qquad 
-  \dfrac{\nu^2}{\tau^4} \int_{\mathbb T^d_{\ell}} R \mathbb D U : \nabla^2 \log R + \dfrac{\nu^2}{\tau^4} \int_{\mathbb T^d_{\ell}} R \nabla U \cdot \nabla^2 \log{R} \\
& \qquad \qquad
- \dfrac{\delta_1 \nu}{ \tau^4}\int_{\mathbb T^{d}_{\ell}}\nabla U : \nabla R \otimes \nabla \log R  
- \dfrac{\delta_2\nu}{\tau^4}  \int_{\mathbb T^d_{\ell}} \Delta U \cdot \nabla \Delta \log R  \\
& \qquad \qquad 
- \dfrac{\delta_1 \nu}{ \tau^4}\int_{\mathbb T^{d}_{\ell}} \dfrac{\Delta R}{R} \Div (RU) 
+ \dfrac{\nu}{\tau^4} \int_{\mathbb T^d_{\ell}} R \nabla U:\nabla^{\top} U   .
\end{aligned}
\]
Introducing $\mathbb A U = \frac12 (\nabla U - \nabla^{\top} U)$ the skew-symmetric part of $\nabla U,$ the second line of the right-hand side also reads
\begin{align*}
-  \dfrac{\nu^2}{\tau^4} \int_{\mathbb T^d_{\ell}} R \mathbb D U : \nabla^2 \log R + \dfrac{\nu^2}{\tau^4} \int_{\mathbb T^d_{\ell}} R \nabla U \cdot \nabla^2 \log{R} 
 =  \dfrac{\nu^2}{\tau^4} \int_{\mathbb T^d_{\ell}} R \mathbb A U : \nabla^2 \log R 
= 0,
\end{align*}
since skew-symmetric and symmetric matrices are orthogonal for the matrix contraction.
Remark also that from the continuity equation \eqref{RU-drag-reg1} we get
\[
\partial_t (\log R) + \frac{1}{\tau^2} \nabla \log R \cdot U + \frac{1}{\tau^2} \Div U = \frac{\delta_1}{\tau^2} \frac{\Delta R}{R},
\]
whence
\[
-\frac{r_0 \nu}{\tau^4} \int_{\mathbb T^d_{\ell}} U \cdot \nabla \log R = \frac{\dd}{\dd t} \left[ \frac{r_0 \nu}{\tau^2} \int_{\mathbb T^d_{\ell}} \log R \right] + \dfrac{2 r_0 \nu \dot{\tau}}{\tau^3}
\int_{\mathbb T^d_{\ell}} \log R 
- \frac{r_0 \nu \delta_1 }{\tau^4} \int_{\mathbb T^d_{\ell}} \frac{\Delta R}{R}.
\]
We finally obtain the identity:
\begin{equation} \label{identite3}
\begin{aligned}
&\frac{\dd}{\dd t} \left\{ \frac{1}{\tau^2} \int_{\mathbb T^d_{\ell}}  \left( \nu RU \cdot \nabla \log R 
+ \frac{\nu^2}{2} R |\nabla \log R|^2 - 2 r_0 \nu \log R \right)   \right\}  \\
& \qquad + \frac{2\nu \dot \tau}{\tau^3}\int_{\mathbb T^d_{\ell}}\left( RU \cdot \nabla \log R - r_0 \log R\right)  \\
&\qquad
+  \frac{4 \nu}{\tau^2} \int_{\mathbb T^d_{\ell}} |\nabla \sqrt{R}|^2  + 
\left( \frac{ \delta_1 \nu^2 }{\tau^4} + \frac{\eps^2 \nu}{\tau^4}\right) \int_{\mathbb T^d_{\ell}} R |\nabla^2 \log{R}|^2
\\
&\qquad
+ \frac{4 \eta_1 \nu }{\alpha\tau^2}\int_{\mathbb T^d_{\ell}} |\nabla R^{-\frac{\alpha}{2}}|^2 
+ \frac{\eta_2 \nu}{\tau^4} \int_{\mathbb T^d_{\ell}} |\Delta^{s+1}R|^2 \\
&\qquad
=  \dfrac{2d\nu}{\tau^2} \int_{\mathbb T^d_{\ell}} R 
- \dfrac{r_0 \nu \delta_1}{\tau^4} \int_{\mathbb T^d_{\ell}} \dfrac{\Delta R}{R}- \dfrac{r_1 \nu}{\tau^4} \int_{\mathbb T^d_{\ell}}  |U|^2 U \cdot \nabla R \\
& \qquad \qquad
- \dfrac{\delta_1 \nu}{ \tau^4}\int_{\mathbb T^{d}_{\ell}}\nabla U : \nabla R \otimes \nabla \log R  
- \dfrac{\delta_2\nu}{\tau^4}  \int_{\mathbb T^d_{\ell}} \Delta U \cdot \nabla \Delta \log R  \\
& \qquad \qquad 
- \dfrac{\delta_1 \nu}{ \tau^4}\int_{\mathbb T^{d}_{\ell}} \dfrac{\Delta R}{R} \Div (RU) 
+ \dfrac{\nu}{\tau^4} \int_{\mathbb T^d_{\ell}} R \nabla U:\nabla^{\top} U   .
\end{aligned}
\end{equation}
We now integrate this identity with respect to time and combine with \eqref{RU-drag-reg-estimation}, observing that 
\[
\int_{\mathbb T^d_{\ell}} R |\mathbb D U|^2 - \int_{\mathbb T^d_{\ell}} R \nabla U:\nabla^{\top} U  = \int_{\mathbb T^d_{\ell}}R |\mathbb A U|^2.
\]
Thus, we obtain (with the notations of  Remark \ref{rem:BDentropie}) that,
for almost all $T \geq 0$,
\[
\begin{aligned}
& \EBDreg(R,U)(T) +\int_0^T \frac{\dot \tau}{\tau^3} \int_{\mathbb T^d_{\ell}} \left( R|U|^2 + \eps^2 |\nabla \sqrt{R}|^2 +\eta_2 |\nabla \Delta^s R|^2\right)\\
&\qquad
+ 2r_0 \nu  \int_0^T \frac{\dot \tau}{\tau^3} \int_{\mathbb T^d_{\ell}} |\log R|  \, {\mathbf 1}_{R<1}
+ \left( \delta_1 \nu^2 + \nu \eps^2  
+ \frac{\delta_1 \eps^2}{2} \right) \int_0^T \frac{1}{\tau^4} \int_{\mathbb T^d_{\ell}} R|\nabla^2  \log R|^2 \\
& \qquad 
+\left( {\nu} + {\delta_1} \right) \int_0^T\frac{4}{\tau^2}\int_{\mathbb T^d_{\ell}} |\nabla \sqrt{R}|^2
+ ( \nu + \delta_1 ) \int_0^T \dfrac{4 \eta_1 }{\tau^2 \alpha}\int_{\mathbb T^d_{\ell}} |\nabla R^{-\frac{\alpha}{2}}|^2 \\
&\qquad
+ \int_0^T \frac{\nu}{\tau^4} \int_{\mathbb T^d_{\ell}} R |\A U |^2 
+ \int_0^T \frac{(\eta_2 \nu + \delta_1 \eta_2)}{\tau^4}  \int_{\mathbb T^d_{\ell}} |\Delta^{s+1}R|^2 \\
&\qquad
+ \int_0^T \frac{\delta_2}{\tau^4}   \int_{\mathbb T^d_{\ell}} |\Delta U|^2
+ \int_0^T \frac{r_0}{\tau^4} \int_{\mathbb T^d_{\ell}}|U|^2 
+ \int_0^T \frac{r_1}{\tau^4} \int_{\mathbb T^d_{\ell}} R |U|^4
\\
&\quad
\le
-r_1 \nu\int_0^T \dfrac{1}{\tau^4} \int_{\mathbb T^d_{\ell}}  |U|^2 U \cdot \nabla R  
-r_0 \nu \delta_1 \int_0^T \dfrac{1}{\tau^4} \int_{\mathbb T^d_{\ell}} \dfrac{\Delta R}{R} 
+2r_0 \nu \int_0^T \frac{ \dot \tau}{\tau^3} \int_{\mathbb T^d_{\ell}} \log R \, {\mathbf 1}_{R \ge 1} \\
& \qquad \qquad
- \delta_1 \nu\int_0^T\dfrac{1}{ \tau^4}\int_{\mathbb T^{d}_{\ell}}\nabla U : \nabla R \otimes \nabla \log R
-\delta_1 \nu \int_0^T \dfrac{1}{ \tau^4}\int_{\mathbb T^{d}_{\ell}} \dfrac{\Delta R}{R} \Div (RU) \\
& \qquad \qquad
- \delta_2\nu \int_0^T \dfrac{1}{\tau^4}  \int_{\mathbb T^d_{\ell}} \Delta U \cdot \nabla \Delta \log R  
+ 2d (\delta_1+\nu)\int_0^T \frac{1}{\tau^2}\int_{\mathbb T^d_{\ell}} R + \nu  \int_0^T \frac{\dot \tau}{\tau^3} \int_{\mathbb T^d_{\ell}} R \Div U  \\
& \qquad \qquad + \EBDreg(R_0,U_0).
\end{aligned}
\]
We denote by $I_1,\ldots,I_8$ the integrals on the right-hand side of this inequality so that we have
$$
\EBDreg(R,U)(T) + \int_0^T \DBDreg(R,U)(t) \, \dd t \le \EBDreg(R_0,U_0) + \sum_{k=1}^8 I_k,
$$ 
and we estimate each of them separately. 
In the sequel, we denote by $K$ and $C$  constants (that may change from line to line). 
The constant $K$ depends only on the parameters of the target system
(namely $\nu,\varepsilon$) and the
initial energy $\Ereg(R_0,U_0)$, while the constant $C$ may depend also on
$T,$ the parameters $\eps, \nu, r_0,r_1,\eta_1,\eta_2, $ and the
initial energy $\Ereg(R_0,U_0).$ But none of them depends on
$(\delta_1,\delta_2).$  We remark that the functions $\frac{1}{\tau^2}$, $\frac{\dot\tau^2}{\tau^2}$, $\frac{1}{\tau^3}$ and $\frac{\dot \tau}{\tau^3}$ are integrable in time over $\R_+$, which we shall use below.

\medskip

For the term $I_1$, integrating by parts, applying Young inequality -- and referring again to \eqref{RU-drag-reg-estimation.1} -- yields:
\[
\begin{aligned}
|I_1| 
&\le  r_1 \nu \int_0^T  \frac{1}{\tau^4} \int_{\mathbb T^d_{\ell}} R |U|^2 |\nabla U|, \\
& 
\le K \left[ \int_0^T \frac{r_1}{\tau^4} \int_{\mathbb T^d_{\ell}} R|U|^4 
+ \int_0^T \frac{\nu}{\tau^4}\int_{\mathbb T^d_{\ell}} {R}| \D U |^2 \right]
+ \frac{1}{2}\int_0^{T} \frac{\nu}{\tau^4} \int_{\mathbb T^d_{\ell}}  {R}| \A U |^2,  \\
&\le \dfrac{1}{2}\int_0^{T} \frac{\nu}{\tau^4} \int_{\mathbb T^d_{\ell}}  {R} |\A U |^2  + K ,
\end{aligned}
\]
and we observe that the first term can be absorbed by the dissipation $\DBDreg$.

For the term $I_2$, since $\alpha >2$ and $s > 2,$ there holds thanks to \eqref{RU-drag-reg-estimation.1}:
\begin{align*}
|I_2| 
& \le {r_0 \nu \delta_1 }\int_0^T \frac{1}{\tau^4} \| \Delta R \|_{L^2} \, \| R^{-1} \|_{L^{2}}
\le \delta_1 K \sup_{(0,T)} \|\Delta R/\tau\|_{L^2}  \sup_{(0,T)} \|R^{-\alpha}\|^{1/2}_{L^1} \int_0^{T} \dfrac{1}{\tau^3}
\le \delta_1  C .
\end{align*}

For the term $I_3$, we have:
\[
I_3 
\le2r_0 \nu\int_0^T \frac{ \dot \tau}{\tau^3} \int_{\mathbb T^d_{\ell}} \log R {\mathbf 1}_{R\ge 1} 
\le r_0 K\int_0^T \frac{\dot \tau}{\tau^3} \int_{\mathbb T^d_{\ell}}  R 
\le r_0 K .
\]

For the term $I_4,$ H\"older inequality in space and
Cauchy-Schwarz inequality in time yield
\begin{align*}
|I_4| & = \delta_1 \nu \left|\int_{0}^T \dfrac{1}{\tau^4} \int_{\mathbb T^d_{\ell}} \sqrt{R}\D U : \dfrac{\nabla R \otimes \nabla R}{R^{3/2}} \right| \\
& \le  \delta_1 \sqrt{\nu T} \left[  \int_0^{T} \dfrac{\nu}{\tau^{4}} \int_{\mathbb T^d_{\ell}} {R}|\D U|^2  \right]^{\frac 12}  \sup_{(0,T)} \|\nabla R/\tau\|_{L^{\infty}}^2 
\sup_{(0,T)} \left[ \int_{\mathbb T^d_{\ell}} \frac{1}{R^3}\right]^{\frac 12}.
\end{align*}
Using Sobolev embedding and \eqref{RU-drag-reg-estimation.1}, we
obtain  that, since $s > d/2$:
\[
\sup_{(0,T)} \|\nabla R/\tau\|^2_{L^{\infty}} 
\le K \sup_{(0,T)} \|\nabla \Delta^s R/\tau\|^2_{L^2} \le C,
\] 
and then $|I_4| \leq \delta_1 C.$ 
%
%

For the term $I_5$, we split $I_5 = I_5^{a} + I_5^b$ where:
\[
I_5^a= \delta_1 \nu \int_{0}^T \dfrac{1}{\tau^4} \int_{\mathbb T^d_{\ell}} \dfrac{\Delta R}{\sqrt{R}} \sqrt{R} \Div U,
\quad 
I_5^b =2 \delta_1 \nu \int_{0}^T \dfrac{1}{\tau^4} \int_{\mathbb T^d_{\ell}} \sqrt{R} U \cdot \nabla \sqrt{R}\, \Delta R.  
\]
As previously, we note in these inequalities that thanks to Sobolev
embeddings and \eqref{RU-drag-reg-estimation.1}, there holds:
\[
\sup_{(0,T)} \|\Delta R/\tau\|_{L^{\infty}} + \sup_{(0,T)} \|\nabla R/\tau\|_{L^{\infty}}  + \sup_{(0,T)} \int_{\mathbb T^{d}_{\ell}} \dfrac{1}{R} \leq C.
\]
Consequently, we have the following controls
\[
|I_5^a| \le  \delta_1 \left( \int_{0}^T\frac{\nu}{\tau^4} \int_{\mathbb T^d_{\ell}} R |\D U|^2\right)^{\frac 12} 
\left( \int_0^T \frac{\nu}{\tau^2} \right)^{\frac 12} \sup_{(0,T)} \|\Delta R/\tau\|_{L^{\infty}} \sup_{(0,T)} \left( \int_{\mathbb T^{d}_{\ell}} \dfrac{1}{R}\right)^{\frac 12} \leq \delta_1 C,
\]
and
\[
|I_5^b| \le  \delta_1 \left[ \int_{0}^{T} \dfrac{\nu}{\tau^2} \right] \sup_{(0,T)} \|\sqrt{R}U/\tau\|_{L^2} \sup_{(0,T)} \|\nabla \sqrt{R}/\tau\|_{L^2} \sup_{(0,T)}\|\Delta R/\tau\|_{L^{\infty}} \leq \delta_1 C. 
\]
%

For the term $I_6$ we have:
\[
|I_6| 
\le \delta_2 \int_0^{T}\frac{1}{2\tau^4} \int_{\mathbb T^d_{\ell}} | \Delta U |^2   
+  \delta_2 \nu^2 \int_0^{T}\frac{1}{2\tau^4}  \int_{\mathbb T^d_{\ell}} | \nabla \Delta \log R |^2, 
\]
and we remark that
\[
\nabla \Delta \log R = \frac{\nabla \Delta R}{R} - \frac{\Delta R \nabla R}{R^2}
- 2\frac{\nabla^2 R \nabla R}{R^2}
+2\frac{|\nabla R|^2 \nabla R}{R^3},
\]
so that, using Sobolev embedding and \eqref{RU-drag-reg-estimation.1}
we obtain:
\[
\sup_{(0,T)} \| \nabla \Delta \log R \|_{L^2} \leq 
K \sup_{(0,T)} \left( 1+\| \nabla \Delta^s R \|_{L^2} \right)^3 
\sup_{(0,T)}\left(1+ \int_{\mathbb T^d_{\ell}} \frac{1}{R^3}\right)
\le C,
\]
which implies
\[
|I_6| 
\le \int_0^{T}\frac{\delta_2}{2\tau^4} \int_{\mathbb T^d_{\ell}} | \Delta U |^2   +  \delta_2 C ,
\]
and we observe that the first term can be absorbed by the dissipation $\DBDreg$.

For the last two terms, we have:
\[
I_7+I_8 \le (2d (1+\nu)+\nu) \int_0^{\infty} \frac{1+ \dot \tau^2}{\tau^2}  \int_{\mathbb T^d_{\ell}} R  
+ \int_0^{\infty} \frac{\nu}{\tau^4} \int_{\mathbb T^d_{\ell}} {R} |\D U |^2 , 
\]
where we have used Cauchy-Schwarz and Young inequalities for $I_8$.
Then, thanks to \eqref{RU-drag-reg-estimation.1}, we get
\[
I_7+I_8  \leq K.
\]

Gathering the previous estimates yields
\[
\EBDreg(R,U)(T) +\frac12 \int_0^T \DBDreg(R,U) \, \dd t \le K +r_0K+ (\delta_1 + \delta_2) C + \EBDreg^+(R_0,U_0).
\]
To conclude, we only need to control the negative part of the BD-entropy, which is done by
\begin{align*}
{\mathcal E}^{-}_{\mathrm{BD}}(R,U)(T) & := \frac{r_0}{ \tau^2(T)} \int |\log R(T)| \, {\mathbf 1}_{R(T)\ge1} \le K r_0 \int_{\T^d_\ell} R  \le r_0 K.
\end{align*}
This concludes the proof.
\end{proof}

\section{Global weak solutions to isothermal fluids with drag
  forces}\label{sec:drag}

In this section we construct global weak solutions to the isothermal fluid system with
drag forces, that is system \eqref{RU-drag} with $r_0, r_1 >0$. We consider solutions on the torus $\T^d_\ell$ by passing to the limit in the regularizing parameters $\delta_1,\delta_2,\eta_1 , \eta_2 \to 0$ from solutions to the regularized system \eqref{RU-drag-reg}.
Let $r_0, r_1 >0$, we define the energy and its corresponding dissipation for the system \eqref{RU-drag}:
\begin{align*}
\Edrag(R,U) 
&= \frac{1}{2\tau^2} \int_{\mathbb T^d_{\ell}} \left( R|U|^2 + \eps^2 |\nabla \sqrt{R}|^2 \right)
+ \int_{\mathbb T^d_{\ell}} \left( R|y|^2 + R\log R \right),
\\
\Ddrag(R,U) 
&= \frac{\dot \tau}{\tau^3} \int_{\mathbb T^d_{\ell}}\left( R|U|^2 + \eps^2 |\nabla \sqrt{R}|^2\right)
+ \frac{\nu}{\tau^4} \int_{\mathbb T^d_{\ell}} R |\D U|^2 
+ \frac{r_0}{\tau^4} \int_{\mathbb T^d_{\ell}} |U|^2
+ \frac{r_1}{\tau^4} \int_{\mathbb T^d_{\ell}} R |U|^4,
\end{align*}
as well as the BD-entropy and its corresponding flux
\begin{align*}
\EBDdrag^+ (R,U)
&= \frac{1}{2\tau^2} \int_{\mathbb T^d_{\ell}} 
\left( R|U + \nu \nabla \log R|^2 + \eps^2 |\nabla \sqrt{R}|^2 - 2 r_0 \log R\mathbf{1}_{R<1} \right)\\
& \quad
+ \int_{\mathbb T^d_{\ell}} \( R|y|^2 + R\log R \),
\\
\DBDdrag(R,U) 
&= \frac{\dot \tau}{\tau^3} \int_{\mathbb T^d_{\ell}} \left( R|U|^2 + \eps^2 |\nabla \sqrt{R}|^2 \right)
+\frac{2r_0 \nu \dot \tau}{\tau^3} \int_{\mathbb T^d_{\ell}} |\log R|  \, {\mathbf 1}_{R<1}\\
& \quad  + \frac{\nu \eps^2}{\tau^4} \int_{\mathbb T^d_{\ell}} R|\nabla^2  \log R|^2
+ \frac{4 \nu}{\tau^2}  \int_{\mathbb T^d_{\ell}} |\nabla \sqrt{R}|^2 
+\frac{\nu}{\tau^4} \int_{\mathbb T^d_{\ell}} R |\A U |^2 \\
&\quad
+ \frac{r_0}{\tau^4} \int_{\mathbb T^d_{\ell}} |U|^2 
+ \frac{r_1}{\tau^4} \int_{\mathbb T^d_{\ell}} R |U|^4.
\end{align*}

We note that these quantities correspond to what remains of the
energy and entropy defined in Section~\ref{sec:reg} when the regularizing
parameters $\delta_1,\delta_2$ and $\eta_1,\eta_2$ are sent to $0$.

It is then natural to build-up a definition of global solution to the isothermal system with drag forces \eqref{RU-drag} with $r_0,r_1 >0$ 
based on the only information that $\Edrag$ and $\EBDdrag^+$ are $L^{\infty}(\mathbb R^+)$
while $\Ddrag$ and $\DBDdrag$ are $L^{1}(\mathbb R^+).$ 
For this, it turns out that it is more suitable to interpret the density $R$ as the square of $\sqrt{R}.$ Indeed, combining $\Edrag$ and $\EBDdrag^+$ yields a bound on $ R |\nabla \log(R)|^2 = 4|\nabla  \sqrt{R}|^2.$ Correspondingly, we
write \eqref{RU-drag1} in terms of $\sqrt{R}$: 
\begin{equation} \label{RU-drag1-bis}
\partial_t \sqrt{R} + \dfrac{1}{\tau^2} \Div (\sqrt{R} U) =  \dfrac{1}{\tau^2}  \sqrt{R} \Div U,
\end{equation}
while in \eqref{RU-drag2} we only rewrite the Korteweg term 
applying the identity (see \cite{LacroixVasseur}):
\[
R \nabla \left(\dfrac{\Delta \sqrt{R}}{\sqrt{R}} \right) = \Div \left(
\sqrt{R} \nabla^2 \sqrt{R} - \nabla \sqrt{R} \otimes \nabla \sqrt{R}\right), 
\]
so that we obtain:
\begin{equation} \label{RU-drag2-bis}
\begin{aligned}
& \partial_t (R U) +\frac{1}{\tau^2}\Div ( \sqrt{R} U \otimes \sqrt{R}U) +2 y R  + \nabla R  
+ \frac{r_0}{\tau^2} U + \frac{r_1}{\tau^2} R |U|^2 U   \\
&\qquad \quad
=\frac{\eps^2}{2\tau^2} \Div\( \sqrt{R} \nabla^2 \sqrt{R} - \nabla \sqrt{R} \otimes \nabla \sqrt{R} \)   
+\frac{\nu}{\tau^2} \Div (R \D U) + \frac{\nu \dot \tau}{\tau} \nabla R. 
\end{aligned}
\end{equation}

\medskip
This remark motivates the following definition.
\begin{definition} \label{def_sol_withdrag}
Given positive parameters $r_0,r_1 >0$ and initial data  
$(\sqrt{R_0}, \Lambda_0 = (\sqrt{R} U)_0) \in L^2(\mathbb T^d_{\ell}) \times L^2(\mathbb T^d_{\ell}),$
we call global weak solution to the isothermal system with drag forces \eqref{RU-drag} in $\mathbb T^d_{\ell}$ any pair 
\[
(\sqrt{R},U) \in 
C([0,\infty);H^1(\mathbb T^d_{\ell})-w) \times L^2_{\mathrm{loc}}(\mathbb R^+;L^2( \mathbb T^d_{\ell})),
\] 
satisfying 
\begin{enumerate}
\item[i)] Further regularity properties:
\[\sqrt{R}U \in C([0,\infty); L^2 (\mathbb T^d_{\ell})-w), \quad \nabla^2 \sqrt{R} \in L^2_{\rm loc}(0,\infty;L^2(\mathbb T^d_{\ell})).
\]
\item[ii)] Equations \eqref{RU-drag1-bis} and \eqref{RU-drag2-bis} in the sense of distributions. 
\item[iii)] Initial data $\sqrt{R}|_{t=0} = \sqrt{R_0}$ and $\sqrt{R} (\sqrt{R}U) |_{t=0} = \sqrt{R_0}\Lambda_0.$
\end{enumerate}
\end{definition}

\begin{remark}
We note that, since $\sqrt{R}$ and $\sqrt{R}U$ are continuous with
respect to time, 
we may give sense to the initial conditions required in item iii) of the above definition.
\end{remark} 

\begin{remark}
We observe the difference between the definition of weak solutions for the system without and with drag forces. When the latter are present ($r_0,r_1 >0$), $U$ is well defined as a function, $\nabla U$ as a distribution and $\sqrt R \mathbb D U$ is well defined. However, in the original system without drag forces, $U$ is not well defined and $\sqrt R \mathbb D U$ has to be understood as $\mathbf S_N$. 
\end{remark}

\begin{theorem}\label{theo:RU-drag}
Assume $r_0 , r_1 ,\nu,\eps > 0.$ 
Let $(\sqrt{R_0}, \Lambda_0 = (\sqrt{R} U)_0)$ %
be an initial data satisfying \eqref{RUinitial-theta2} and such that $\Edraginit, \EBDdraginit < + \infty$.
Then there exists a global weak solution $(R,U)$ to the isothermal fluid system with drag forces
  \eqref{RU-drag} in $\mathbb T^d_{\ell}$, in the sense of
  Definition~\ref{def_sol_withdrag}, associated to the initial data $(\sqrt{R_0}, \Lambda_0 )$. Furthermore, there exist constants 
  $C_1$ and $C_2$ (whose dependencies are mentioned in parenthesis)
  such that this solution 
  satisfies the energy inequality 
\[
\sup_{t \ge 0} \Edrag(R,U)
+ \int_0^\infty \Ddrag(R,U) \, \dd t \le C_1(\Edraginit),
\]
and also the BD-entropy inequality
\[
\sup_{t \ge 0} \EBDdrag(R,U)
+ \int_0^\infty \DBDdrag(R,U) \, \dd t \le C_2(\Edraginit,\EBDdraginit ).
\]
\end{theorem}

\begin{proof}[Proof of Theorem \ref{theo:RU-drag} ]
The proof consists of three parts: starting with the regularized system \eqref{RU-drag-reg}, in the first one we pass to the limit in the parameters $\delta_1, \delta_2 \to 0$, which shall give us the existence of global weak solutions to an intermediate system given by \eqref{RU-drag-reg} with $\delta_1=\delta_2 = 0$; then we pass to the limit $\eta_1,\eta_2 \to 0$ to obtain a weak solution to \eqref{RU-drag} on the torus. 
In the whole proof $(\sqrt{R_0}, \Lambda_0 = (\sqrt{R} U)_0)$ is a fixed initial data satisfying \eqref{RUinitial-theta2} and the drag parameters $(r_0,r_1) \in (0,\infty)^2$ are fixed.

\medskip\noindent
\textit{Step 1. Limits $\delta_1, \delta_2 \to 0$.}
In this part, we fix $\eta_1 >0$ and $\eta_2 >0$ and we consider  sequence of parameters $\delta_1,\delta_2$ converging to $0$. To simplify notations we shall denote $\delta = (\delta_1,\delta_2)$ and drop the $\eta_1,\eta_2$ dependencies.
We consider the sequence of global weak solutions $\{ (R_\delta ,U_\delta) \}_\delta$ to the regularized problem~\eqref{RU-drag-reg} associated to $( {R_{0}} ,  U_{0})$, as constructed in Proposition~\ref{prop:sol-approche}.  First, we construct
limits $R$ and $U$ of this sequence as in Step 1 of Section \ref{sec:convergenceapproximate}. 

\medskip

We proceed with improving the sense of the convergence of $\{ (R_\delta ,U_\delta) \}_\delta$ to these limits. For this, we fix an arbitrary finite $T>0.$  Thanks to the energy and BD-entropy inequalities, this sequence verifies uniform estimates in the following spaces:
\begin{equation} \label{eq:unifbound-drag}
\begin{aligned}
& R_\delta (1+|y|^2 + |\log R_\delta|) \text{ in } L^\infty(0,T; L^1(\mathbb T^d_{\ell})), &
 \nabla \sqrt{R_\delta}  \text{ in } L^\infty(0,T; L^2(\mathbb T^d_{\ell})), \\
 & \sqrt{\eta_2} R_{\delta}  \text{ in } L^{\infty}(0,T;H^{2s+1}(\mathbb T^d_{\ell})),& 
 \\
&
\sqrt{R_\delta} U_\delta   \text{ in } L^\infty(0,T;  L^2(\mathbb T^d_{\ell})), 
& 
\sqrt{\nu} \, \sqrt{R_\delta} \nabla U_\delta   \text{ in } L^2(0,T; L^2(\mathbb T^d_{\ell})).
\end{aligned}
\end{equation}
Recalling \eqref{estim-lowerbdd}, this entails that $\{ R_{\delta} \}_{\delta}$ is bounded in $L^{\infty}(0,T;H^{1}(\mathbb T^d_{\ell}))$. 
Writing the weak form \eqref{RU-drag-reg1-weak} with a test function $\Psi  \in \mathcal D((0,T) \times \mathbb T^d_{\ell})$, we obtain that:
\[
\partial_t R_{\delta} = - \sqrt{R_{\delta}} \sqrt{R_{\delta}} \Div (U_{\delta}) -  2\sqrt{R_{\delta}} U_{\delta} \cdot \nabla \sqrt{R_{\delta}} 
+ \dfrac{\delta_1}{\tau^2} \Delta R
\quad \text{ in $\mathcal D'((0,T)\times \mathbb T^d_{\ell}).$}
\]
This implies that $\{ \partial_t R_{\delta} \}_\delta$ is also
bounded in $L^{2}(0,T;L^1(\mathbb T^d_{\ell}))$. Applying again
Ascoli-Arzel\`a arguments yields
$
R_{\delta}  \to R \text{ in $C([0,T];H^{2s}(\mathbb T^d_{\ell}))$}
$
and, moreover, with the uniform bound from below on $R_{\delta}$ in \eqref{estim-lowerbdd}, we get
\[
R^{-1}_\delta \to R^{-1} \quad \text{ in } C([0,T] \times \mathbb T^d_{\ell}).
\]

On the other hand, we note that the above bound \eqref{eq:unifbound-drag}
also entails that $\{ R_{\delta} U_{\delta}\}_{\delta}$ is bounded in $L^{2}(0,T;H^1(\mathbb T^d_{\ell})).$ 
Taking then $\Phi \in \DD((0,T) \times \mathbb T^d_{\ell})$
in \eqref{RU-drag-reg2-weak}, and recalling \eqref{eq_Ktwidentity} which is satisfied by  $R_{\delta}>0,$ we obtain (in $\DD'((0,T) \times \mathbb T^d_{\ell})$):
\begin{align*}
 \partial_t (R _{\delta}U_{\delta}) &=  - \frac{1}{\tau^2}\Div ( \sqrt{R_{\delta}} U_{\delta} \otimes \sqrt{R}_{\delta} U_{\delta}) - 2 y R_{\delta}  - \nabla R_{\delta}  + \eta_1 \nabla R_{\delta}^{-\alpha} \\
&\quad 
 - \frac{r_0}{\tau^2} U_{\delta} -  \frac{r_1}{\tau^2} R_{\delta} |U_{\delta}|^2 U_{\delta} - \frac{\delta_1}{\tau^2}(\nabla R_{\delta} \cdot \nabla) U_{\delta}  \\
&\quad
 + \frac{\eps^2}{2\tau^2}\(\sqrt{R_{\delta}} \nabla^{2}\sqrt{R_{\delta}} - \nabla \sqrt{R_{\delta}} \otimes \nabla \sqrt{R_{\delta}}\)    
+\frac{\nu}{\tau^2} \Div (R_{\delta} \D U_{\delta})  \\
& \quad 
+ \frac{\nu \dot \tau}{\tau} \nabla R_{\delta}
+\frac{\delta_2}{\tau^2} \Delta^2 U_{\delta}  + \frac{\eta_2}{\tau^2} R_{\delta} \nabla \Delta^{2s+1} R _{\delta}.
\end{align*}
Consequently, combining the uniform bounds in \eqref{eq:unifbound-drag} with the uniform bounds in the following spaces (again due to the energy and BD-entropy inequalities):
\begin{equation}\label{estim-RU-delta}
\begin{aligned}
& \sqrt{r_0}\, U_\delta \text{ in } L^2(0,T;L^2(\mathbb T^d_{\ell}),
&   \sqrt{r_1}\, R_\delta^{\frac{1}{4}} U_\delta\text{ in }L^4(0,T; L^4(\mathbb T^d_{\ell})) , 
\\
& \sqrt{\delta_2} \, \Delta U_\delta\text{ in }L^2(0,T; L^2(\mathbb T^d_{\ell}), &  R_\delta  \text{ in } L^2(0,T; H^{2s+2}(\mathbb T^d_{\ell})), 
\\
& \eta_1^{\frac{1}{\alpha}}\, R_\delta^{-1}\text{ in }L^\infty(0,T; L^{\alpha}(\mathbb T^d_{\ell})), 
& \sqrt{\nu \eps^2} \, \nabla^2 \sqrt{R_\delta} \text{ in } L^2(0,T ; L^2(\mathbb T^d_{\ell})),
\end{aligned}
\end{equation}
we conclude that $\{ \partial_t (R_{\delta} U_{\delta}) \}_\delta$ is bounded 
in $L^{2}(0,T;H^{-(2s+1)}(\mathbb T^d_{\ell})).$
This entails that $R_{\delta} U_{\delta} \to RU$ in $L^2(0,T; L^2(\mathbb T^d_{\ell})).$

%
Thanks to the previous estimates and Aubin-Lions/Ascoli-Arzel\`a arguments, we obtain the following convergences:
\begin{equation}\label{compacite-delta}
\begin{aligned}
R_\delta \to R  & \text{ in } L^2(0,T;H^{2s+2}(\mathbb T^d_{\ell})-w)\text{ and }C([0,T];H^{2s}(\mathbb T^d_{\ell})), \\
R_\delta U_\delta \to RU & \text{ in } L^2(0,T;L^p(\mathbb T^d_{\ell}))\,, \quad \forall \, p<6\\
U_\delta \to U & \text{ in } L^2(0,T;L^{2}(\mathbb T^d_{\ell})) ,
\\
\sqrt{R_\delta} U_\delta \to \sqrt{R} U &  \text{ in } L^p(0,T;L^2(\mathbb T^d_{\ell})) \text{ ($\forall \, p < \infty$) and } C([0,T] ; L^2(\mathbb T^d_{\ell})-w), 
\\
R_\delta^{\frac14} U_\delta \to R^{\frac14} U  & \text{ in } L^p(0,T;L^p(\mathbb T^d_{\ell})), \quad \forall \, p <4. \\
\end{aligned}
\end{equation}
The above list of convergences shows that we can pass to the limit in
the initial condition. It also readily implies that:
\begin{align}
R_\delta U_\delta \otimes U_{\delta} \to RU \otimes U & \text{ in } L^1(0,T;L^1(\mathbb T^d_{\ell})), \label{compacite-RU2-delta} \\
R_\delta |U_\delta|^2 U_\delta  \to R|U|^2 U  & \text{ in } L^1(0,T;L^1(\mathbb T^d_{\ell})), \label{compacite-RU3-delta}\\ 
\sqrt{R_\delta}  U_\delta \to \sqrt{R} U  & \text{ in } L^2(0,T;L^2(\mathbb T^d_{\ell})). 
\label{compacite-sqrtRU-delta}
\end{align}

We can now pass to the limit in the equations \eqref{RU-drag-reg1-weak}-\eqref{RU-drag-reg2-weak} when $\delta \to 0$, by remarking that, using the above estimates, we have
\[
\delta_1 \int_0^T \!\!\! \int \frac{1}{\tau^2} R_\delta \Delta \Psi \to 0,
\quad
\delta_1 \int_0^T \!\!\! \int \frac{1}{\tau^2} \nabla U_\delta  : \nabla R_\delta \otimes \Phi \to 0,
\quad
\delta_2 \int_0^T \!\!\! \int \frac{1}{\tau^2} \Delta U_\delta  \Delta \Phi \to 0,
\]
where $\Psi$ and $\Phi$ are smooth test functions with compact support in $(0,T) \times \mathbb T^d_{\ell}$.
We have hence constructed $(R,U)$ which is a global weak solution to the intermediate system corresponding to \eqref{RU-drag-reg} with $\delta_1=\delta_2=0$, and, passing to the limit $\delta \to 0$ in the energy~\eqref{RU-drag-reg-estimation.1} and BD-entropy~\eqref{estim-entropy} inequalities, the solution $(R,U)$ satisfies moreover the energy inequality~\eqref{RU-drag-reg-estimation.1} with $\delta_1=\delta_2=0$ as well as the BD-entropy inequality~\eqref{estim-entropy} with $\delta_1=\delta_2=0$.

\medskip

Before going further, we remark that the continuity equation \eqref{RU-drag-reg1} holds almost everywhere. Since $R >0$ on any compact interval of time, this entails that $\sqrt{R}$
satisfies \eqref{RU-drag1-bis} in $\DD'((0,\infty) \times \mathbb T^d_{\ell})$.

\medskip\noindent
\textit{Step 2. Limits $\eta_1 , \eta_2 \to 0$.}
With similar conventions as in the previous step, we introduce now $\eta = (\eta_1,\eta_2)$ and we consider $\{ (R_\eta, U_\eta) \}_\eta$ the sequence of global weak solutions associated with initial data $( \sqrt{R_{0}} , \Lambda_{0}  )$ constructed in the Step 1. Thanks to the energy and BD-entropy inequalities, 
we obtain again the following uniform bounds:
\begin{equation} \label{eq:unifbound-dragbis}
\begin{aligned}
& R_{\eta} (1+|y|^2 + |\log R_\eta|) \text{ in } L^\infty(0,T; L^1(\mathbb T^d_{\ell})), &
 \nabla \sqrt{R_\eta}  \text{ in } L^\infty(0,T; L^2(\mathbb T^d_{\ell})), \\
&
\sqrt{R_\eta} U_{\eta} \text{ in } L^\infty(0,T;  L^2(\mathbb T^d_{\ell})), 
& 
 \sqrt{R_\eta} \nabla U_\eta \text{ in } L^2(0,T; L^2(\mathbb T^d_{\ell})).
\end{aligned}
\end{equation}
Introducing this bound in \eqref{RU-drag1-bis} -- so that we prove
$\partial_t \sqrt{R_{\eta}}$ is bounded in $L^{2}(0,T;H^{-1}(\mathbb T^{d}_{\ell}))$ --  and remarking that 
$\sqrt{R_{\eta}}$ is bounded in $L^{\infty}(0,T;H^{1}(\mathbb T^{d}_{\ell})),$ Aubin-Lions argument entails that 
\[
\sqrt{R_{\eta}} \to  \sqrt{R} \text{ in } C([0,T];L^2(\mathbb T^d_{\ell})) \text{ and } L^{2}(0,T;L^{2}(\mathbb T^{d}_{\ell})).
\]
Furthermore, thanks to the energy and BD-entropy inequalities, we have the uniform bounds:
\begin{equation}\label{estim-RU-eta}
\begin{aligned}
& \sqrt{r_0}\, U_\eta  \text{ in }  L^2(0,T; L^2(\mathbb T^d_{\ell})), \\
& \sqrt{r_1}\, R_\eta^{\frac{1}{4}} U_\eta \text{ in }   L^4(0,T; L^4(\mathbb T^d_{\ell})), \quad
r_0 \log \left(\frac{1}{R_\eta}\right)_+ \text{ in }   L^\infty(0,T; L^1(\mathbb T^d_{\ell})), \\
&  \eps \, \nabla^2 \sqrt{R_\eta} \text{ in }  L^2(0,T; L^2(\mathbb T^d_{\ell})), \quad
\sqrt{\eps} \, \nabla R_\eta^{\frac{1}{4}} \text{ in }  L^4(0,T; L^4(\mathbb T^d_{\ell})).
\end{aligned}
\end{equation}
From these bounds, and arguing similarly as in Step 1, we get the convergences
\begin{equation}\label{compacite-eta}
\begin{aligned}
U_\eta \to U & \text{ in } L^2(0,T;L^{2}(\mathbb T^d_{\ell}))-w, \\
\sqrt{R_\eta} U_\eta \to \sqrt{R} U &  \text{ in } C([0,T] ; L^2(\mathbb T^d_{\ell})-w), \\
R_\eta^{\frac14} U_\eta \to R^{\frac14} U  & \text{ in } L^4(0,T;L^4(\mathbb T^d_{\ell}))-w, \\
R_\eta U_\eta \to RU & \text{ in } L^2(0,T;L^2(\mathbb T^d_{\ell})) .
\end{aligned}
\end{equation}
Furthermore, we remark that we have
\[
R_\eta |U_\eta|^2 U_\eta  \to R|U|^2 U   \text{ a.e.}\\
\]
so that we can apply the uniform bound on $\{ R_{\eta}^{1/4}U_{\eta}\}_{\eta}$ to reproduce the arguments of \cite[Lemma 2.3]{VasseurYu} to yield:
\[
R_\eta U_\eta \otimes U_{\eta} \to RU \otimes U  \text{ in } L^1(0,T;L^1(\mathbb T^d_{\ell})) .
\]
With these convergences at-hand, we can already pass to the limit in the weak formulation of the continuity equation~\eqref{RU-drag-reg1-weak}.
For the weak formulation~\eqref{RU-drag-reg2-weak}, we only need to prove the convergence to zero of the cold
pressure term $\eta_1 \nabla R_\eta^{-\alpha}$ and the regularization
term $\frac{\eta_2}{\tau^2} R_\eta \nabla \Delta^{2s+1} R_\eta$, since the other terms can be treated with the above convergences. 

\smallbreak

We recall that we have the estimates
\begin{equation}\label{estim-RU-eta-bis}
\begin{aligned}
\sqrt{\eta_2} \, R_\eta  \in L^\infty(0,T; H^{2s+1}(\mathbb T^d_{\ell})), \quad
\sqrt{\eta_2} \, \Delta^{s+1} R_\eta  \in L^2(0,T; L^2(\mathbb T^d_{\ell})), \\
\eta_1^{\frac{1}{\alpha}}\, R_\eta^{-1} \in L^\infty(0,T; L^{\alpha}(\mathbb T^d_{\ell})), \quad
\sqrt{\eta_1}\, \nabla R_\eta^{-\frac{\alpha}{2}} \in L^2(0,T; L^2(\mathbb T^d_{\ell})).
\end{aligned}
\end{equation}
On the one hand, from \eqref{estim-RU-eta-bis} and Fatou's lemma we obtain
$$
\int  \log \left(\frac{1}{R}\right)_+  \dd y =
\int \liminf_{\eta \to 0} \log \left(\frac{1}{R_\eta}\right)_+  \dd y < + \infty,
$$
which implies that $\mathrm{meas}(\{ y \in \T^d_\ell \mid R(t,y) = 0   \}) = 0$ for a.e.\ $t\in (0,T)$. Since we already know that $R_\eta \to R$ a.e.\ in $(t,y)$, we deduce
$$
\eta_1 R_\eta^{-\alpha} \to 0  \text{ a.e.\ in } (t,y) \text{ when } \eta_1 \to 0.
$$
We now claim that the uniform estimate $\eta_1 R_\eta^{-\alpha} \in L^{\frac53} ((0,T) \times \T^d_\ell)$ holds, from which we deduce the convergence
\[
\eta_1 R_\eta^{-\alpha} \to 0 \text{ in } L^1(0,T ; L^1 (\T^d_\ell)) \text{ when } \eta_1 \to 0.
\]
Let us prove this claim:
since $\sqrt{\eta_1}\, \nabla R_\eta^{-\frac{\alpha}{2}} \in L^2(0,T ; L^2 (\T^d_\ell))$ and $\sqrt{\eta_1}\, R_\eta^{-\frac{\alpha}{2}} \in L^\infty (0,T ; L^2 (\T^d_\ell))$, we get $\sqrt{\eta_1}\, R_\eta^{-\frac{\alpha}{2}} \in L^2(0,T ; H^1(\T^d_\ell) ) \hookrightarrow L^2(0,T ; L^6 (\T^d_\ell))$, 
 whence ${\eta_1} R_\eta^{-\alpha} \in L^1 (0,T ; L^3 (\T^d_\ell))$. We finally obtain the claim by using the interpolation inequality 
\[
\| f \|_{L^{\frac53} ((0,T) \times \T^d_\ell)} \le \| f \|_{L^\infty (0,T ; L^1 (\T^d_\ell))}^{\frac25} \, \| f \|_{L^1 (0,T ; L^3 (\T^d_\ell))}^{\frac35}.
\]

On the other hand, we now want to show that, for any test function $\Phi \in \mathcal D([0,T) \times \T^d_\ell)^d$,
\begin{equation}
\eta_2 \int_0^T \!\!\! \int \frac{1}{\tau^2} \Delta^{s+1} R_\eta \Delta^s \left[ \nabla R_\eta \cdot \Phi + R_\eta \Div \Phi \right]
\to 0 \text{ as } \eta_2 \to 0,
\end{equation}
and we only concentrate in the sequel on the most difficult term, that is corresponding to the $\Delta^s(\nabla R_\eta )\cdot \Phi$ term, the other ones being treated similarly. Recall that $R_\eta \in L^\infty (0,T ; L^1 \cap L^3 (\T^d_\ell))$ uniformly in $\eta$ thanks to \eqref{eq:unifbound-dragbis}, and also the interpolation inequality
\[
\| f \|_{\dot H^{2s+1}(\T^d_\ell) } \lesssim \| f \|_{\dot H^{2s+2}(\T^d_\ell) }^{\frac{2s+1}{2s+2}} \, \| f \|_{L^{2}(\T^d_\ell) }^{\frac{1}{2s+2}}.
\]
Therefore, denoting $0 < a = \frac{2s+1}{2s+2} < 1$, we have
\[
\begin{aligned}
&\left|  \eta_2 \int_0^T \!\!\! \int \frac{1}{\tau^2} \Delta^{s+1} R_\eta \Delta^s(\nabla R_\eta )\cdot \Phi \right| \\
&\qquad
\le C_{\Phi}\,  \eta_2   \| \nabla^{2s+2} R_\eta\|_{L^2(0,T ; L^2 (\T^d_\ell))} \,  \| \nabla^{2s+1} R_\eta \|_{L^2(0,T ; L^2 (\T^d_\ell))}  
\\
&\qquad
\le C_{\Phi}\,  \eta_2^{\frac12-\frac{a}{2}}  \left( \sqrt{\eta_2} \| \nabla^{2s+2} R_\eta\|_{L^2(0,T ; L^2 (\T^d_\ell))} \right)^{1 +  \frac{2s+1}{(2s+2)}} \,  \| \nabla^{2s+1} R_\eta \|_{L^2(0,T ; L^2 (\T^d_\ell))}^{\frac{1}{2s+2}}
\to 0 \text{ as } \eta_2 \to 0.
\end{aligned}
\]
This ends the proof that $(\sqrt{R},U)$ satisfies \eqref{RU-drag2-bis}.
\end{proof}

At this stage we have constructed a global weak solution $(\sqrt{R},U)$ to the isothermal fluid system \eqref{RU-drag} with drag forces ($r_0 , r_1 >0$) on the torus $\T^d_\ell$, in the sense of Definition~\ref{def_sol_withdrag}, 
for smooth initial data satisfying \eqref{RUinitial-theta2}. 
Furthermore this solution verifies the energy and BD-entropy inequalities of the statement of the theorem, which are obtained straightforwardly in the limit $\eta \to 0$ from the associated inequalities for $(R_\eta, U_\eta)$.

%

\section{Global weak solutions in the whole space $\R^d$}\label{sec:tore} 

The 
next steps consist in passing to the limit $r_0,r_1\to 0,\ell \to
\infty$, and possibly $\eps \to 0$. To do so, we adapt the
approach of 
\cite{LacroixVasseur}, based on a suitable notion of renormalized
solution. We emphasize the main steps of the proof and the technical
modifications, and refer to \cite{LacroixVasseur} for other details. 

\subsection{Outline of the proof}
The method introduced in \cite{LacroixVasseur} is based on the
introduction of a new family 
of solutions to the Navier-Stokes system: the renormalized weak solutions.  In our framework these solutions are defined as follows:
\begin{definition}[Renormalized weak solution]\label{def:renorm}
Let $\Omega = \mathbb T^d_{\ell}$ or $\Omega = \mathbb R^d.$ Let $r_0,
r_1 \ge 0$,  $ \eps \ge 0$ and $\nu >0$.  Let $(\sqrt{R_0}, \Lambda_0 =
(\sqrt{R} U)_0 ) \in H^1\cap \F(H^1)(\Omega) \times L^2(\Omega)$
verify
\[
\sqrt{R_0} \ge 0 \text{ a.e.  on }  \Omega,  \quad   (\sqrt{R}U)_0 =
0 \text{ a.e. on } \{\sqrt{R_0} = 0 \}.
\]
We say that $(R,U)$ is a global renormalized weak solution to \eqref{RU-drag} in $\Omega$, and associated to the initial data  $(\sqrt{R_0}, \Lambda_0)$, if
there exists a collection $(\sqrt R,\sqrt R U,{\mathbf S}_K,{\mathbf
  T}_N)$ satisfying
\begin{itemize}
\item[i)] The following regularities: 
\begin{align*}
&\(\<y\> +|U|\) \sqrt{R} \in L^{\infty}_{\rm loc}\(0,\infty; L^2 (\Omega)\),\quad 
 \nabla \sqrt R \in L^{\infty}_{\rm loc}\(0,\infty; L^2 (\Omega)\),\\ 
& \eps \nabla^2 \sqrt{R} \in L^2_{\rm loc}(0,\infty;L^2(\Omega)) ,\quad
  \mathbf T_N \in L^2_{\rm loc}(0,\infty; L^2(\Omega)) ,\\
& \sqrt{\eps} \nabla R^{1/4}\in L^4_{\rm loc}(0,\infty;L^4(\Omega)),\quad
  r_1^{1/4}R^{1/4}U \in L^4_{\rm loc}(0,\infty;L^4(\Omega),\\
& r_0^{1/2} U\in L^2_{\rm loc}(0,\infty;L^2(\Omega)),\quad r_0\log
  R\in L^\infty_{\rm loc}(0,\infty;L^1(\Omega)),
\end{align*}
with the compatibility conditions
\[
\sqrt{R} \ge 0 \text{ a.e.  on } (0,\infty)\times \Omega,  \quad   \sqrt{R}U=
0 \text{ a.e. on } \{\sqrt{R} = 0 \}.
\]
\item[ii)] For any function $\varphi\in W^{2,\infty}(\R^d)$, there
  exist 
two measures $f_\varphi,g_\varphi \in{\mathcal M}((0,\infty)\times \Omega)$ 
    with 
  \begin{equation*}
    \| f_\varphi\|_{{\mathcal M}((0,\infty)\times \Omega) }+
\| g_\varphi\|_{{\mathcal M}((0,\infty)\times \Omega) }
\le C\| \nabla^2 \varphi\|_{L^\infty(\R^d)},
  \end{equation*}
where the constant $C$ depends only on the solution $(\sqrt R,\sqrt R
U)$, such that  in $\mathcal D'((0,\infty)\times \mathbb R^d)$,
\begin{subequations}\label{eq:renorm}
    \begin{align}[left = \empheqlbrace\,] \label{eq_renorm_continuity}
  &&& \partial_t \sqrt{R}+\frac{1}{\tau^2}\Div (\sqrt{R} U )=
  \frac{1}{2\tau^2}{\rm Trace}(\mathbf T_N),\\
     \label{eq_renorm_momentum} 
   &&& \partial_t ({R}\varphi(U)) +\frac{1}{\tau^2}\Div ( R \varphi(U) \otimes U) \\
    &&& 
      \quad +2 y R\varphi'(U) +\varphi'(U)\nabla
      R+\frac{r_0}{\tau^2}U\varphi'(U)  
+ \frac{r_1}{\tau^2}R|U|^2U\varphi'(U) \notag \\
     &&&\quad =\Div \left(\dfrac{\nu}{\tau^2} \sqrt R \varphi'(U)\mathbf S_N
        + \dfrac{\eps^2}{2\tau^2} \varphi'(U)\mathbf S_K\right) \notag
        +      \dfrac{\nu \dot{\tau}}{\tau} \varphi'(U)\nabla R+f_\varphi , \notag
    \end{align}
\end{subequations}
with $\mathbf S_N$ the symmetric part of $\mathbf T_N$ and the compatibility conditions:
\begin{align*} 
& \sqrt{R}\varphi'_i(U)[\mathbf T_{N}]_{jk}
  = \partial_j(R\varphi'_i(U) U_k) - 2 \sqrt{R}U_k \partial_j
  \sqrt{R}+g_\varphi\,, \quad \forall i,j,k\in \{1,\cdots,d\},\\[6pt]  
& \mathbf S_K 
=\sqrt{R}\nabla^2 \sqrt{R} -  \nabla \sqrt{R} \otimes \nabla \sqrt{R}
  \,. 
\end{align*}
\item[iii)] For any $\psi\in C^\infty(\Omega)$, 
  \begin{align*}
    &\lim_{t\to 0}\int_{\Omega} \sqrt{R}(t,y)\psi(y) \, \dd y=
    \int_{\Omega}   \sqrt{R_0} (y)\psi(y) \, \dd y, \\
& \lim_{t\to 0}\int_{\Omega} \sqrt{R}(t,y) (\sqrt{R} U)(t,y) \psi(y) \, \dd y =
    \int_{\Omega}  \Lambda_0 (y) \psi(y) \, \dd y.
  \end{align*}
\end{itemize}
\end{definition}

Recall the definition of global weak solutions for
\eqref{RU-drag} on the torus in Definition~\ref{def_sol_withdrag} for
the case $r_0,r_1 >0$, or in Definition~\ref{def:weak} for solutions in
$\mathbb R^d$ with $r_0=r_1=0$.  
The main interest of the notion of renormalized solutions lies in the
fact that it is easier to 
construct solutions to \eqref{eq:renorm}. More precisely, it is easier to prove the weak stability of renormalized solutions, and to prove the following
properties:
\begin{itemize}

\item For $r_0,r_1\ge 0$, any renormalized weak solution is also a
  weak solution, 
\item In the case $r_0,r_1,\eps>0$, the two notions are equivalent: any
  weak solution is a renormalized solution. 
\end{itemize}

The proof of existence of weak solution to the quantum Navier Stokes
system  then reduces to three steps:
\begin{itemize}
\item Proving that the weak solutions with drag forces that we constructed previously are indeed 
renormalized solutions.
\item Proving compactness of renormalized solutions in terms of the
  parameters $r_0,r_1,\eps$ and $\ell$.
\item Proving that renormalized solutions in the whole space provide weak solutions in $\mathbb R^d.$
\end{itemize}

\subsection{Proof of the main theorem}

Consider initial data $(\sqrt{R_0},\Lambda_0 = (\sqrt{R}U)_0) \in
H^1\cap \F(H^1)(\mathbb R^d) \times L^2(\mathbb R^d)$ 
as in the assumption of Theorem~\ref{theo:main}. We first construct a sequence of initial data
\[
\sqrt{R_{0,\ell}},\Lambda_{0,\ell}  \in H^1(\mathbb T^d_{\ell}) \times L^2(\mathbb T^d_{\ell}) ,
\quad \forall \, \ell \in \mathbb N^*,
\] 
which enter the framework of Theorem  \ref{theo:RU-drag}. This shall yield an associated sequence $\{ (\sqrt{R}_{\ell},U_{\ell})\}_{\ell \in \mathbb
  N^*}$ of weak solutions to the isothermal system \eqref{RU-drag}
with drag forces ($r_0,r_1>0$) on the torus $\mathbb T^d_\ell$. We design our sequence of truncated initial data so that, for
well-chosen drag parameters, the  energy and BD-entropy estimates of
Theorem~\ref{theo:RU-drag} yield uniform bounds for these solutions. 

\smallbreak 

So, we consider a plateau function $\chi\in C_c^\infty(\R^d)$ and smoothing kernel 
$\zeta \in C^{\infty}_c(\mathbb R^d)$ such that
\begin{align*}
&   {\mathbf 1}_{|y|\le 1/2}\le \chi \le {\mathbf 1}_{|y|<1}, \\
&  \operatorname{supp}(\zeta) \subset B(0,1),  \qquad \int_{\mathbb R^d}
                                                                   \zeta(y){\rm d}y = 1, 
\end{align*}
and, for $\ell,\iota>0$, we set
\begin{equation*}
  \chi_\ell(y)=\chi\(\frac{y}{\ell}\) , \quad \zeta_{\iota}(y) = \dfrac{1}{\iota^d} \zeta\left( \frac{y}{\iota}\right).
\end{equation*}
Given $\ell \in \mathbb N^*,$ $\iota >0$ and $\theta >0$ we define now 
$S^{0}_{\ell,\theta,\iota}$ and $\Lambda_{0,\ell}$ as 
\begin{equation*}
 S^{0}_{\ell,\theta,\iota} (y) =  \(\sqrt{R_{0}} (y) \chi_\ell(y) + \theta\)\ast \zeta_{\iota},\quad
   \Lambda_{0,\ell} (y) = \Lambda_{0} (y),\quad \text{for}\quad y\in [-\ell,\ell]^d. 
\end{equation*}
Since $\chi_\ell$ is zero on the boundary of the box, the above
formula for $S^{0}_{\ell,\theta,\iota}$ defines an initial data that is
smooth, strictly positive, and periodic. The above candidate 
$(S^{0}_{\ell,\theta,\iota},\Lambda_{0,\ell})$ satisfies then the assumptions
of Theorem \ref{theo:RU-drag} whichever the value of $\theta,\iota >0.$
The main property of this construction is the following proposition. 

\begin{proposition}
There exist sequences $(\theta_\ell)_{\ell \in \mathbb N^*}$
and $(\iota_{\ell})_{\ell \in \mathbb N^*}$ such that, denoting 
\[
\sqrt{R_{0,\ell}} := S^{0}_{\ell,\theta_{\ell},\iota_{\ell}} ,\quad
\forall \, \ell \in \mathbb N^*,  
\]
we have:
\begin{align*}
& \limsup_{\ell \to \infty} \int_{\mathbb T^{d}_\ell} {R}_{0,\ell}(x){\rm d}x  \leq \int_{\mathbb R^d} R_{0} ,\\
& \limsup_{\ell \to \infty} \int_{\mathbb T^d_{\ell}} |\nabla \sqrt{R}_{0,\ell} |^2  \leq \int_{\mathbb R^d} |\nabla \sqrt{R}_{0} |^2,   \\
& \limsup_{\ell \to \infty} \int_{\mathbb T^d_{\ell}} R_{0,\ell}|y|^2 \leq \int_{\mathbb R^d} R_0 |y|^2.
\end{align*}
\end{proposition}

\begin{proof}
We note that 
\[
S^{0}_{\ell,\theta,\iota} \Tend \theta 0 \(\sqrt{R_0} \chi_{\ell}\)\ast \zeta_{\iota}=:
S^0_{\ell,\iota} \text{ in $ C^1(\mathbb T^{d}_{\ell}).$} 
\]
Since all the integrals involved in our proposition are continuous in $S^{0}_{\ell,\theta,\iota}$ for the $C^{1}$-topology, 
we may only prove the claimed inequalities by replacing $S^{0}_{\ell,\theta,\iota}$ with $S^{0}_{\ell,\iota}.$

\smallbreak

Standard arguments with the convolution -- combined with explicit computations of the truncation -- entail that, for arbitrary $\iota >0$:
\begin{align*}
& \limsup_{\ell \to \infty} \int_{\mathbb T^d_{\ell}} |S^{0}_{\ell,\iota}|^2 {\rm d}x    \leqslant  \int_{\mathbb R^d} R_0 ,\\
& \limsup_{\ell \to \infty} \int_{\mathbb T^d_{\ell}} |\nabla S^{0}_{\ell,\iota}|^2 \leqslant  \int_{\mathbb R^d} |\nabla \sqrt{R_0}|^2.
\end{align*} 
Then, by a convexity argument and duality formulas for the convolution, we obtain that
\begin{align*}
\int_{\mathbb T^{d}_{\ell}} S^0_{\ell,\iota}|y|^2 & = \int_{\mathbb T^d_{\ell}}  |[\sqrt{R_0} \chi_\ell] \ast \zeta_{\iota}|^2 |y|^2  \leq \int_{\mathbb T^{d}_{\ell}} [|\sqrt{R_0} \chi_{\ell}|^2 * \zeta_{\iota}] |y|^2 \\
& \leq \int_{\mathbb T^{d}_{\ell}} |\sqrt{R_0} \chi_{\ell}|^2 ((1+\iota)|y|^2 + C\iota), 
\end{align*}
for an absolute constant $C.$
Consequently, we obtain again that, for arbitrary $\iota >0$,
\[
\limsup_{\ell \to \infty} \int_{\mathbb T^{d}_{\ell}} |S^0_{\ell,\iota}|^2 
|y|^2 \leq   \int_{\mathbb R^d} R_0 (1+\iota)|y|^2 + C \iota^2 \int_{\mathbb R^d} R_0.
\]
It thus suffices to consider a sequence $\iota_\ell\to 0$. 
\end{proof}
Note that applying Lemma~\ref{lem:SigmaLlogL} to
\begin{equation*}
 \mathbf 1_{[-\ell,\ell]^d}\sqrt{ R_{0,\ell}} ,
\end{equation*}
viewed as a function on $\R^d$, we infer from the above proposition
that
$  \int_{\mathbb T^{d}_{\ell}} R_{0,\ell} \left|\log R_{0,\ell}\right|$
is bounded uniformly in $\ell$. 

In what follows, we consider that
$(\sqrt{R_{0,\ell}},\Lambda_{0,\ell})_{\ell \in \mathbb N^*}$ is the
sequence of initial data constructed in the previous
proposition. Invoking Theorem \ref{theo:RU-drag} with
these data for 
arbitrary $\ell \in \mathbb N^*$, we obtain a sequence
$(\sqrt{R_{\ell}},U_{\ell})_{\ell \in \mathbb N^*}$ such that  
for arbitrary $\ell \in \mathbb N^*,$ the pair
$(\sqrt{R_{\ell}},U_{\ell})$ is a global weak solution to \eqref{RU-drag} on the
torus $\T_\ell^d$. 
We
denote also 
\[
r_{0,\ell} := \dfrac{1}{\ell +\( \displaystyle \int_{\mathbb T^d_{\ell}}
  \log(R_{0,\ell}) \mathbf{1}_{R_{0,\ell} < 1}\)^2}\,, 
\quad 
r_{1,\ell} := \dfrac{1}{\ell} ,\quad \eps_{\ell} = \eps  + \dfrac{1}{\ell},
\]
and of course, these values affect the above mentioned sequence of
solutions $(\sqrt{R_{\ell}},U_{\ell})_\ell$.
These choices ensure that the associated sequence of initial energies
$\Edrag$ (barotropic. entropies $\EBDdrag$) converge to the energy $\EE$
(resp. entropy $\EBD$) of $(\sqrt{R_0},\Lambda_0).$ As a matter of fact, the
somehow intricate choice for $r_{0,\ell}$ is motivated by this
property, to obtain
\begin{equation*}
 r_{0,\ell}  \int_{\T^d_\ell} \log (R_{0,\ell} )
 \mathbf{1}_{R_{0,\ell} < 1} \Tend \ell \infty 0.
\end{equation*}

\subsubsection{Weak solutions with drag forces are renormalized solutions}
Given $\ell \in \mathbb N^*,$ we first obtain that the weak solution
we constructed in the previous step is a renormalized solution as
stated in Definition~\ref{def:renorm}.  
To start with, we note that, in the case with drag and when $\Omega$ is a torus,  item i) in Definition~\ref{def:renorm} gathers all the regularity properties inherited from 
the energy and entropy estimates in Theorem~\ref{theo:RU-drag}. The
only point that deserves more details is 
the construction of the tensor $\mathbf{T}_{N,\ell}.$ We set: 
\[
\mathbf{T}_{N,\ell} = \sqrt{R_{\ell}} \nabla U_{\ell}.
\]
This tensor is well defined (at least in $\mathcal D'((0,\infty)
\times \mathbb T^{d}_{\ell})$) since, 
thanks to the energy/entropy estimates,   we have $U_{\ell} \in
L^2_{\rm loc}((0,\infty) \times  \mathbb T_{\ell}^{d})$ 
and $\sqrt{R_\ell} \in L^2_{\rm loc}((0,\infty); H^1(\mathbb T_{\ell}^d)).$
Furthermore, we control the symmetric part  
(resp.\ the skew-symmetric part) of $\mathbf{T}_{N,\ell}$  with the
energy dissipation (resp.\ the BD-entropy dissipation) so that we obtain
the expected $L^2_{\rm loc}((0,\infty);L^2(\mathbb T^d_{\ell}))$
regularity. 

\smallbreak

We proceed with item ii) of the definition, the last one being an
obvious corollary to the time regularity  
of $(\sqrt{R_\ell},U_\ell)$ as stated in Definition \ref{def_sol_withdrag}. 
By definition, the pair $(\sqrt{R_\ell},\sqrt{R_\ell}U_{\ell})$ solves
the continuity equation \eqref{eq_renorm_continuity},
identifying the right-hand side of \eqref{RU-drag1-bis} as $\Div
\mathbf T_{N,\ell}.$ 
The compatibility conditions for the tensor $\mathbf{S}_{K,\ell}$
can be seen as a definition.  

\smallbreak 

The main point of the construction is to obtain the momentum equation
in terms of renormalized solution  
\eqref{eq_renorm_momentum}. We give here only the main ideas of the
computation and refer the reader 
to \cite[Section 3]{LacroixVasseur} for more details.
In order to multiply the equation with $\varphi'(U_{\ell}),$ the first
step is to regularize the momentum equation by truncating large and
small values of $\sqrt{R_{\ell}}$ in order to take advantage of the
good integrability properties 
of $R_{\ell}^{1/4} U_{\ell}.$ To this end, we first remark that the
continuity equation reads: 
\[
\partial_t \sqrt{R_{\ell}} + \dfrac{2}{\tau^{2}} R^{1/4}_{\ell} U_{\ell} \cdot\nabla R^{1/4}_{\ell} + \dfrac{1}{2\tau^2} \sqrt{R_{\ell}} \Div U_{\ell} = 0.
\]
Applying the bounds on $\nabla R_{\ell}^{1/4}$ stemming from
\eqref{eq:equivJungel} we obtain  $\partial_t \sqrt{R_{\ell}} \in
L^2_{\rm loc}((0,\infty) \times \mathbb T^{d}_{\ell}).$  Moreover, we also
know that 
$\nabla \sqrt{R_{\ell}} \in L^{\infty}_{\rm loc}((0,\infty);L^2(\mathbb T^d_{\ell})).$
Consequently, for arbitrary $\phi \in C^{1}_c(0,\infty),$  $\phi(R_{\ell}) = \phi(\sqrt{R_{\ell}}^{2})$ enjoys the same time and space integrability. 
On the other hand, we remark that the momentum equation satisfied by 
$R_{\ell} U_{\ell}$ reads:
\[
\partial_t (R_{\ell} U_{\ell}) + \dfrac{1}{\tau^2}\Div (R_{\ell} U_{\ell} \otimes U_{\ell}) = \Div(\sqrt{R_{\ell}}\mathbf S_{\ell}) - F_{\ell},
\] 
where 
\begin{align*}
\mathbf S_{\ell} & =  \dfrac{\nu}{\tau^2} \sqrt{R_{\ell}}\mathbb D(U_\ell) + \dfrac{\eps_{\ell}}{2\tau^2 }(\nabla^2 \sqrt{R_{\ell}} - 
4 \nabla R_{\ell}^{1/4} \otimes \nabla R_{\ell}^{1/4} )+ \left( \dfrac{\nu \dot{\tau}}{\tau} - 1\right) \sqrt{R}_{\ell} \mathbf{I}_{d}, \\
F_{\ell} & = \dfrac{r_{0,\ell}}{\tau^2} U + \dfrac{r_{1,\ell}}{\tau^2} R_{\ell} |U_{\ell}|^2 U_{\ell} + 2y R_{\ell}. 
\end{align*}
Here we denoted by $\mathbf{I}_{d}$ the identity matrix. 
Since  $\sqrt{R}_{\ell} \in L^{2}_{\rm loc}((0,\infty);H^2(\mathbb
T^d_{\ell})) \subset L^2_{\rm loc}((0,\infty);L^{\infty}(\mathbb
T^d_{\ell}))$ ($d \leq 3$),  we have $F_{\ell} \in L^{4/3}_{\rm
  loc}((0,\infty) \times \mathbb T^d_{\ell})$ 
and $\sqrt{R_{\ell}}\mathbf S_{\ell} \in L^1_{\rm loc}((0,\infty)
\times \mathbb T^{d}_{\ell}).$
On the left-hand side of the equation, we have:
\begin{align*}
R_{\ell} U_{\ell} &= \sqrt{R}_{\ell} (\sqrt{R_{\ell}} U_{\ell}) \in L^{2}_{\rm loc}((0,\infty) \times \mathbb T^d_{\ell})\,, \\
R_{\ell} U_{\ell} \otimes U_{\ell} &= \sqrt{R}_{\ell} \, R_{\ell}^{1/4} U_{\ell} \otimes R_{\ell}^{1/4} U_{\ell} \in L^{1}_{\rm loc}((0,\infty) ; L^2(\mathbb T^d_{\ell})).
\end{align*}
We thus have sufficient regularity to multiply the momentum equation
with $\phi(R_{\ell}).$  We obtain:
\begin{align*}
& \partial_{t} (\phi(R_{\ell}) R_{\ell} U_{\ell}) + \dfrac{1}{\tau^2}{\Div} ( R_{\ell}U_{\ell} \otimes \phi(R_{\ell}) U_{\ell})\\
& \quad = {\Div}(\phi(R_\ell) \sqrt{R_\ell}\mathbf S_{\ell}) + \phi(R_{\ell}) F_{\ell} - \sqrt{R_{\ell}}\mathbf{S}_{\ell} \cdot \nabla \phi(R_{\ell}) 
 + (\partial_{t} \phi(R_{\ell}) + U_{\ell} \cdot \nabla \phi(R_{\ell}))R_{\ell} U_{\ell} . 
\end{align*}
At this point, we remark that we may also multiply the continuity
equation \eqref{eq_renorm_continuity} with  
a suitable function of $\sqrt{R}_{\ell}$ in order to replace it with
\[
\partial_{t} \phi(R_{\ell}) + U_{\ell} \cdot \nabla \phi(R_{\ell}) = - \dfrac{1}{\tau^2} \phi'(R_{\ell}) \sqrt{R_{\ell}}{\rm Trace} \mathbf{T}_{N,\ell}.
\]
Introducing $V_{\ell} = \phi(R_{\ell})U_{\ell},$ we have finally, 
\begin{align*}
& \partial_{t} ( R_{\ell}V_{\ell}) + \dfrac{1}{\tau^2}{\Div} ( R_{\ell}U_{\ell} \otimes V_{\ell})\\
& \quad = {\Div}(\phi(R_\ell) \sqrt{R_\ell}\mathbf S_{\ell}) + \phi(R_{\ell}) F_{\ell} - \sqrt{R_{\ell}}\mathbf{S}_{\ell} \cdot \nabla \phi(R_{\ell}) 
 - \dfrac{1}{\tau^2} R_{\ell} U_{\ell} \phi'(R_{\ell})\sqrt{R_{\ell}} {\rm Trace} \mathbf{T}_{N,\ell}. 
\end{align*}
Since $\phi$ truncates the small and large values of $R_{\ell}$ we may rewrite 
\[
V_{\ell} = R^{1/4}_{\ell} U_{\ell} \dfrac{\phi(R_{\ell})}{R^{1/4}_{\ell}} \in L^4_{\rm loc}((0,\infty) \times \mathbb T^d_{\ell} ). 
\]
We are then in position to multiply the $i$-th equation of the momentum equation by $\varphi'(V_{\ell}).$
With the help of Friedrich's lemma we obtain, on the left-hand side 
\[
\(\partial_{t} ( R_{\ell}V_{\ell}) + \dfrac{1}{\tau^2}{\Div} (
R_{\ell}U_{\ell} \otimes V_{\ell})\) \cdot \varphi'(V_{\ell}) 
 = \partial_{t} (\phi(R_{\ell}) R_{\ell} \varphi(V_{\ell})) + \dfrac{1}{\tau^2}{\Div} ( R_{\ell}U_{\ell} \otimes \varphi(V_{\ell}))) ,
\]
and, on the right-hand side:
\begin{align*}
&  \left( {\Div}(\phi(R_\ell)\sqrt{R_{\ell}} \mathbf S_{\ell}) + \phi(R_{\ell}) F_{\ell} - \sqrt{R_{\ell}}\mathbf{S}_{\ell} \cdot \nabla \phi(R_{\ell}) 
 - \dfrac{1}{\tau^2} R_{\ell} U_{\ell} \phi'(R_{\ell})\sqrt{R_{\ell}} {\rm Trace} \mathbf{T}_{N,\ell} \right)\cdot \varphi'(V_{\ell})\\
& = 
{\Div}(\phi(R_\ell)\sqrt{R_{\ell}} \mathbf S_{\ell}\cdot  \varphi'(V_{\ell})) + \phi(R_{\ell})\varphi'(V_{\ell}) \cdot F_{\ell} -
 [\mathbf{S}_{\ell} \cdot \nabla \phi(R_{\ell})] \cdot \varphi'(V_{\ell}) \\
 &\quad
 - \dfrac{1}{\tau^2} \phi'(R_{\ell})\sqrt{R_{\ell}} R_{\ell} U_{\ell}  \cdot  \varphi'(V_{\ell}) {\rm Trace} \mathbf{T}_{N,\ell} 
 - \phi(R_{\ell})  \sqrt{R_{\ell}}\mathbf S_{\ell} : \varphi''(V_{\ell}) \nabla V_{\ell}.
\end{align*}
To obtain \eqref{eq_renorm_momentum}, it remains to approximate the
constant $1$ with a suitable sequence of functions
$(\phi_m)_{m\in\mathbb N}.$ This construction is performed in  
\cite{LacroixVasseur} and \cite{VasseurYu}. We emphasize that, in this
case with drag forces: 
\begin{align*}
 f_{\varphi} &=  \varphi''(U_{\ell})  \mathbf S_{\ell} : \sqrt{R_{\ell}}\nabla U_{\ell}  \in L^1_{\rm loc}((0,\infty) \times \mathbb T^d_{\ell}),\\
 \|f_{\varphi} \|_{L^1_{\rm loc}((0,\infty) \times \mathbb T^d_{\ell})} & \leq \|\varphi''\|_{L^{\infty}([0,\infty))} \left(  \Edrag(R^0_{\ell},U^{0}_{\ell})  + \EBDdrag(R^0_{\ell},U^0_{\ell})\right).
\end{align*}
Finally, the compatibility condition concerning $\mathbf{T}_{N,\ell}$
is obtained by noting that 
for arbitrary $\varphi \in W^{2,\infty}(\mathbb R^d)$ and $j,k \in \{ 1 , \dots, d \}$, we have:
\[
\varphi'(U_{\ell}){R_{\ell}} \partial_ j U_{\ell,k} = \partial_j ( {R_{\ell}} \varphi'(U_{\ell})U_{\ell,k})  -   2\sqrt{R_{\ell}} U_{\ell,k} \varphi'(U_{\ell})  \partial_j \sqrt{R_{\ell}} 
- R_{\ell} U_{\ell,k} \varphi''(U_{\ell}) \partial_{j} U_{\ell},
\]
which is obtained standardly by first regularizing  $\sqrt{R}_{\ell}$
and $U_{\ell}$. So, we have:
\[
\sqrt{R_{\ell}} \varphi'(U_{\ell})\mathbf{T}_{N,\ell,j,k} =
\partial_{j} ( {R_{\ell}}U_{\ell}\varphi'(U_{\ell})U_{\ell,k}) -
2\sqrt{R_{\ell}} U_{\ell,} \partial_{j}\sqrt{R_{\ell}} +
g_{j,k,\varphi} ,
\]
with $g_{j,k,\varphi} \in L^{2}_{\rm
  loc}((0,\infty);L^1(\mathbb{T}^{d}_{\ell}))$ satisfying 
\[
\|g_{j,k,\varphi} \|_{ L^{1}_{\rm loc}((0,\infty) \times
  \mathbb{T}^{d}_{\ell})} \leq \|\varphi''\|_{L^{\infty}([0,\infty))}
\left(  \Edrag(R^0_{\ell},U^{0}_{\ell})  +
  \EBDdrag(R^0_{\ell},U^0_{\ell})\right). 
\]

\subsubsection{Compactness of renormalized solutions and conclusion}
We are now able to prove our main result Theorem \ref{theo:main}.
Since, in any case ({\em i.e.} with or without drag) renormalized solutions to \eqref{RU} are weak solutions 
as defined in Definition \ref{def:weak} (see \cite[Section 4]{LacroixVasseur}), we only show that, when we let the parameter $\ell \to \infty,$ we can extract a subsequence from $(\sqrt{R_{\ell}},\sqrt{R_{\ell}}U_{\ell})_{\ell \in \mathbb N^*}$ 
that converges to a renormalized solution to \eqref{RU} on the whole space $\mathbb R^d.$

First, thanks to the energy and entropy estimates on the one hand, and the 
choice of initial data on the other hand, the sequences $\{ (\sqrt{R_{\ell}},\sqrt{R_{\ell}}U_{\ell},\mathbf{T}_{N,\ell})\}_{\ell}$ 
are uniformly bounded  in the following spaces, respectively:
\begin{align*}
& \sqrt{R_{\ell}} \text{ in } L^{\infty}_{\rm loc}(0,\infty; H^1_{\rm loc} (\R^d)), \\
&  \sqrt R_{\ell}U_{\ell} \text{ in } L^{\infty}_{\rm loc}(0,\infty; L^2_{\rm loc} (\R^d)),\\ 
&\mathbf T_{N,\ell} \text{ in } L^2_{\rm loc}(0,\infty; L^2_{\rm loc}(\R^d)) .
\end{align*}
Furthermore, by the choice of our initial data, we have:
\begin{multline*}
\limsup_{\ell \to \infty} \left( \|\sqrt{R}_{\ell}\|_{L^{\infty}_{\rm loc}(0,\infty; H^1 (\T_{\ell}^d))} +  
 \|\sqrt R_{\ell}U_{\ell}\|_{L^{\infty}_{\rm loc}(0,\infty; L^2 (\T_{\ell}^d))}
+ \|\mathbf T_{N,\ell}\|_{L^2_{\rm loc}(0,\infty; L^2(\mathbb T_{\ell}^d))} \right) \\
\leq C(\sqrt{R_0},\Lambda_0).
\end{multline*} 
Consequently, by a standard Cantor extraction argument, we can construct 
\begin{align*}
& \sqrt{R} \text{ in } L^{\infty}_{\rm loc}(0,\infty; H^1 (\R^d)), \\
&  \sqrt RU \text{ in } L^{\infty}_{\rm loc}(0,\infty; L^2 (\R^d)),\\ 
&\mathbf T_{N} \text{ in } L^2_{\rm loc}(0,\infty; L^2(\mathbb R^d)) ,
\end{align*}
so that, without relabelling the subsequences:
\begin{align*}
& \sqrt{R_{\ell}} \rightharpoonup \sqrt{R} \text{ in } L^{\infty}_{\rm loc}(0,\infty; H^1_{\rm loc} (\R^d))-w*, \\
&  \sqrt R_{\ell}U_{\ell} \rightharpoonup \sqrt{R} U  \text{ in } L^{\infty}_{\rm loc}(0,\infty; L^2_{\rm loc} (\R^d))-w*,\\ 
&\mathbf T_{N,\ell}  \rightharpoonup \mathbf{T}_N  \text{ in } L^2_{\rm loc}(0,\infty; L^2_{\rm loc}(\mathbb R^d))-w .
\end{align*}
In addition,  we have also momentum and (if $\eps>0$) second order
bounds for $\sqrt{R}_{\ell}$ uniformly in $\ell$ 
so that $\sqrt{R}$ enjoys the further estimates:
\[
\epsilon \nabla^{2} \sqrt{R} \in L^2_{\rm loc}(0,\infty;L^2(\mathbb R^d)) \quad \sqrt{\epsilon} \nabla R^{1/4}  \in L^4_{\rm loc}(0,\infty;L^4(\mathbb R^d)) \quad 
\<y\> \sqrt{R} \in L^{\infty}_{\rm loc}(0,\infty,L^2(\mathbb R^d)).
\]
We have now a candidate satisfying item i) of the definition of
renormalized solutions without drag forces on the torus. Furthermore, we can
pass to the weak limit in the energy and entropy estimates on the
torus so that these solutions 
satisfy \eqref{eq:EE+DD<=EE0} and \eqref{eq_EBD_weak}.

We note that the above weak convergences of $\sqrt{R_{\ell}},\sqrt{R_{\ell}}U_{\ell}$ and $\mathbf{T}_{N,\ell}$
are sufficient to pass to the limit in the continuity equation \eqref{eq_renorm_continuity}. Reproducing the arguments for the limits $\eta_1,\eta_2 \to 0$ in the
previous section (see also the proof of Lemma 5.1 in
\cite{LacroixVasseur}), we obtain that  
\[
\sqrt{R}_{\ell} \to \sqrt{R} \text{ in $C([0,\infty);L^2_{\rm loc}(\mathbb R^d))$} .
\]
We note that, since we control the second momentum of
$\sqrt{R}_{\ell}$, the convergence actually holds in 
$ C([0,T];L^2(\mathbb R^d)).$ When $\eps >0,$ by interpolation,
we have also that 
\[
\sqrt{R}_{\ell} \to \sqrt{R} \text{ in }L^4_{\rm
    loc}((0,\infty);H^1_{\rm loc}(\mathbb R^d)) . 
\]
We can then combine the strong convergence of $\sqrt{R_{\ell}}$ and
the weak convergence of $\nabla^2 \sqrt{R_{\ell}}$ to pass to the
limit in the compatibility condition for $\mathbf{S}_{K}.$ 

It remains to pass to the limit in the renormalized momentum equation and the compatibility condition for 
$\mathbf{T}_{N}.$ For this, we can again reproduce the arguments of
\cite{LacroixVasseur} with the only integrability of
$\sqrt{R_{\ell}}.$ We obtain that $R_{\ell} U_{\ell} \to RU$ in
$L^2_{\rm loc}((0,\infty);L^{p}_{\rm loc}(\mathbb R^d))$  for
arbitrary $p < 3/2.$ Introducing $U = RU/R\mathbf{1}_{R>0},$ we
conclude that $R_{\ell} \to R$ and $U_{\ell} \to U$ a.e., and
consequently that $R_{\ell}^{\alpha}\phi(U_{\ell}) \to R^{\alpha}
\phi(U)$  in $L^{p}_{\rm loc}((0,\infty)\times \mathbb R^d)$ for any
bounded  $\phi : \mathbb R^d \to \mathbb R^d,$ $\alpha < 6$   and $p <
6/\alpha.$ 
Given $\varphi \in W^{2,\infty}(\mathbb R^d),$ we remark that the
remainder $f_{\ell,\varphi}$ is a bounded sequence of measures, so
that  
we can extract a weakly converging sequence. The above convergences
are then sufficient to pass to the limit in the renormalized momentum
equations with $\varphi$ satisfied by  $(\sqrt{R_{\ell}}, U_{\ell})$
and obtain \eqref{eq_renorm_momentum}. 
We proceed similarly to pass to the limit in the renormalized compatibility 
condition for $\mathbf T_{N,\ell}$ and obtain the renormalized
compatibility condition for $\mathbf{T}_N.$  This ends the proof.

\section{Global weak solutions to isothermal Korteweg equation}
\label{sec:korteweg}

In this section, we explain how to prove
Proposition~\ref{prop:korteweg}. The idea is the same as in
\cite[Proposition~15]{AnMa09} in the barotropic case, and we present
the specificities of the isothermal case.
\smallbreak

Formally, Proposition~\ref{prop:korteweg} stems from Madelung transform:
consider the solution $\psi\in L^\infty_{\rm loc}(\R;H^1(\R^d))$ to the logarithmic
Schr\"odinger equation
\begin{equation}
  \label{eq:logNLS}
  i\eps \partial_t\psi +\frac{\eps^2}{2}\Delta \psi =
  \psi\log|\psi|^2;\quad \psi_{\mid t=0}=\psi_0.
\end{equation}
Then $(\r,j)= \(|\psi|^2,\eps\IM (\bar\psi\nabla\psi)\)$ is a natural
candidate for the conclusions of
Proposition~\ref{prop:korteweg}. 
Indeed, we compute
\begin{equation*}
  \partial_t\r  = 2\RE \bar \psi\partial_t \psi = -\eps \IM
                 \(\bar\psi\Delta \psi\) = -\Div \(\eps \IM\(\bar
                 \psi\nabla \psi\)\),
\end{equation*}
and, in view of the identity
\begin{equation*}
  \partial_t \nabla \psi = \frac{i\eps}{2}\Delta\nabla \psi
  -\frac{i}{\eps}\nabla \( \psi\log|\psi|^2\),
\end{equation*}
\begin{align*}
  \partial_t j & = \eps\IM \(\nabla \psi \(-\frac{i\eps}{2}\Delta\bar
                 \psi +\frac{i}{\eps}\bar\psi\log|\psi|^2\)\)
+ \eps\IM \(\bar\psi\(\frac{i\eps}{2}\Delta \nabla \bar\psi
                 -\frac{i}{\eps}\nabla \(\psi\log|\psi|^2\)\)\)\\
& = \frac{\eps^2}{4} \nabla \Delta|\psi|^2 -\eps^2 \Div
\(\RE\(\nabla\bar\psi\otimes \nabla\psi\)\) -\nabla |\psi|^2.
\end{align*}
\textcolor{black}{
  The above identities are true in the sense of distributions,
  provided (at least) that $\psi\in H^1(\R^d)$. Therefore, to show
  that $(\r,j)$ is a solution to \eqref{fluide} with $\nu=0$, we
  have to rewrite the term $\Div
\(\RE\(\nabla\bar\psi\otimes \nabla\psi\)\)$. 
}
In view of \cite[Lemma~3]{AnMa09}, for $\psi\in H^1(\R^d)$, there exists $\phi\in
L^\infty(\R^d)$ such that $\psi=\sqrt\r \phi$ a.e. in $\R^d$,
$\sqrt\r\in H^1(\R^d)$, $\nabla\sqrt\r = \RE(\bar\phi\nabla\psi)$, so
that if we set $\sqrt \r u:= \eps\IM(\bar\phi\nabla\psi)$, then
$\sqrt \r u\in L^2(\R^d)$,  $j= \sqrt\r\times \sqrt\r u$ and 
\begin{equation*}
  \eps^2\RE\(\nabla\bar\psi\otimes \nabla\psi\) =\eps^2 \nabla\sqrt\r
  \otimes \nabla\sqrt\r +(\sqrt\r u)\otimes(\sqrt\r u). 
\end{equation*}
In this case,
\begin{equation*}
  \phi(x) = \left\{
    \begin{aligned}
      \frac{\psi(x)}{|\psi(x)|}&\quad \text{if }\psi(x)\not =0,\\
      0 &\quad \text{if }\psi(x) =0,
    \end{aligned}
    \right.
  \end{equation*}
  so the compatibility condition $\sqrt\r u=0$ a.e. on $\{\sqrt\r=0\}$ is
  satisfied. 
Finally, by the definition of $j$,
\begin{equation*}
  \nabla\wedge j = \eps\IM \(\nabla \bar\psi\wedge\nabla \psi\),
\end{equation*}
and \cite[Corollary~13]{AnMa09} yields, for $\psi\in H^1(\R^d)$,
\begin{equation*}
  \nabla\wedge j = 2\nabla\sqrt\r\wedge (\sqrt\r u). 
\end{equation*}
Note
that in the barotropic case considered in
\cite[Proposition~15]{AnMa09}, $p(\r)=\r^\gamma$, $\gamma>1$, instead
of the logarithmic
Schr\"odinger equation \eqref{eq:logNLS}, one faces the more standard
nonlinear Schr\"odinger equation with a power-like nonlinearity,
\begin{equation}\label{eq:NLS}
  i\eps \partial_t\psi +\frac{\eps^2}{2}\Delta \psi =
c_\gamma  |\psi|^{\gamma-1}\psi;\quad \psi_{\mid t=0}=\psi_0,
\end{equation}
for some constant $c_\gamma>0$ whose exact value is irrelevant for the
present discussion. 
\bigbreak

The Cauchy problem for \eqref{eq:logNLS} was solved initially in
\cite{CaHa79}  locally in time for
$\psi_0\in L^2(\R^d)$, using the theory of monotone operators. To
obtain a solution with an $H^1$ regularity, as well as the uniqueness
of this solution, in \cite{CaHa79,CaHa80} (see also
\cite{CazCourant}) the authors have to change the sign in front of the
nonlinearity in \eqref{eq:logNLS}, so the Hamiltonian structure of the
equation directly provides a priori estimates. In the case of
\eqref{eq:logNLS}, the formally conserved energy
\begin{equation}\label{eq:ElogNLS}
  E_{\rm logNLS} = \frac{\eps^2}{2}\int_{\R^d}|\nabla \psi|^2
  +\int_{\R^d}|\psi|^2\log|\psi|^2 ,
\end{equation}
is not helpful because the region $\{|\psi|<1\}$ yields a negative
contribution, and cannot be controlled in terms of the
$H^1$-norm. \textcolor{black}{ This is why in the present case, working
  in $H^1$ is not enough, and a (fractional) momentum is considered
  to, $\psi_0\in \F(H^\alpha)$, that is,
  \[ \int_{\R^d} \<x\>^{2\alpha}|\psi_0(x)|^2\dd x<\infty,\]
  for some $0<\alpha\le 1$. Then \eqref{eq:logNLS} has a unique, global
  solution  $\psi\in L^\infty_{\rm
    loc}(\R;H^1\cap \F(H^\alpha))$.   We refer to \cite{CaGa18} for
  details. The first part of Proposition~\ref{prop:korteweg} follows.}

    \bigbreak

    To conclude and prove the second point of
    Proposition~\ref{prop:korteweg}, introduce $\Psi$ given by
    \begin{equation*}
      \psi(t,x) = \frac{1}{\tau(t)^{d/2}}\Psi\(t,\frac{x}{\tau(t)}\)
      \(\frac{\|\r_0\|_{L^1(\R^d)}}{
          \|\Gamma\|_{L^1(\R^d)}}\)^{1/2}\exp\(i\frac{\dot\tau(t)}{\tau(t)}
        \frac{|x|^2}{2\eps} -i\frac{\theta(t)}{\eps}\), 
      \end{equation*}
      where
      \begin{equation*}
        \theta(t) = d\int_0^t \log\tau(s)\dd s -t\log \(\frac{\|\r_0\|_{L^1(\R^d)}}{
          \|\Gamma\|_{L^1(\R^d)}}\).
      \end{equation*}
    It solves (see \cite{CaGa18})
    \begin{equation}\label{eq:Psi}
      i\eps\partial_t \Psi +\frac{\eps^2}{2\tau(t)^2}\Delta\Psi =
      \Psi\log|\Psi|^2 +|y|^2\Psi;\quad \Psi(0,y) = \psi_0(y)
      \(\frac{\|\r_0\|_{L^1(\R^d)}}{  \|\Gamma\|_{L^1(\R^d)}}\)^{1/2}. 
    \end{equation}
    We check
    \begin{equation*}
      \r(t,x) = |\psi(t,x)|^2 =
      \frac{1}{\tau(t)^d}\left|\Psi\(t,\frac{x}{\tau(t)}\)\right|^2
      \frac{\|\r_0\|_{L^1(\R^d)}}{ \|\Gamma\|_{L^1(\R^d)}},
    \end{equation*}
    so in view of \eqref{eq:uvKorteweg}, $R=|\Psi|^2$, and
    \begin{align*}
      \sqrt\r u(t,x) &= \frac{j(t,x)}{\sqrt{\r(t,x)}}=\frac{\eps\IM
                     \(\bar\psi\nabla \psi\)(t,x)}{|\psi(t,x)|} \\
      &=
      \(\frac{\eps}{\tau(t)^{1+d/2} }\IM\(
      \frac{\Psi}{|\Psi|}\nabla\Psi\) \(t,\frac{x}{\tau(t)}\)
      +\frac{\dot
      \tau(t)}{\tau(t)}\frac{x}{\tau(t)^{d/2}}\left|\Psi\(t,\frac{x}{\tau(t)}\)\right| \)
      \(\frac{\|\r_0\|_{L^1(\R^d)}}{
        \|\Gamma\|_{L^1(\R^d)}}\)^{1/2},\\
      &=
      \frac{\eps}{\tau(t)^{1+d/2} }\IM\(
      \frac{\Psi}{|\Psi|}\nabla\Psi\) \(t,\frac{x}{\tau(t)}\)\(\frac{\|\r_0\|_{L^1(\R^d)}}{
        \|\Gamma\|_{L^1(\R^d)}}\)^{1/2}
      +\frac{\dot
      \tau(t)}{\tau(t)}x\sqrt{\r(t,x)},
    \end{align*}
    hence, in view of \eqref{eq:uvKorteweg},
    \[\sqrt R U= \eps\IM\(\frac{\bar\Psi}{|\Psi|}\nabla \Psi\).
      \]
In view of \cite{CaGa18}, for $\psi_0\in H^1\cap \F(H^1)$,
\eqref{eq:Psi} has a global solution $\Psi\in L^\infty_{\rm
  loc}(\R;H^1\cap \F(H^1)$, which satisfies
\begin{equation*}
  \frac{\dd}{\dd t}\( \frac{\eps^2}{2\tau(t)^2}\|\nabla\Psi(t)\|_{L^2}^2 +\int_{\R^d}
  |\Psi(t,y)|^2\log |\Psi(t,y)|^2\dd y + \int_{\R^d}
 |y|^2  |\Psi(t,y)|^2\dd y\)=-
 \frac{\eps^2\dot(\tau(t)}{\tau(t)^3}\|\nabla\Psi(t)\|_{L^2}^2.
\end{equation*}
Integrating in time and rewriting the quantities involved in this
relation in terms of $(\sqrt R,\sqrt R U)$, we recover \eqref{eq:EE+DD=EE0}.

\appendix

\section{Proof of identity \eqref{identite1}} \label{sec:identite1}

We recall that, the first step in the computation of \eqref{identite1} is to set $\Phi = \chi \nu \nabla \log R / \tau^2$ 
in \eqref{RU-drag-reg2-weak}. This yields:

\[
\begin{aligned}
& \int_0^{\infty} \!\!\! \int_{\mathbb T^{d}_{\ell}}RU \cdot \partial_t \Phi  
+  \int_0^{\infty} \!\!\! \int_{\mathbb T^{d}_{\ell}}\frac{1}{\tau^2} R U \otimes U : \nabla \Phi 
 \\
&\qquad\quad
=  \int_0^{\infty} \!\!\! \int_{\mathbb T^{d}_{\ell}}R (2y \cdot \Phi - \Div \Phi)   
+  r_0 \int_0^{\infty} \!\!\!  \int_{\mathbb T^{d}_{\ell}} \frac{1}{\tau^2}   U \cdot \Phi 
+  r_1 \int_0^{\infty} \!\!\! \int_{\mathbb T^{d}_{\ell}} \frac{1}{\tau^2}  R|U|^2  U \cdot \Phi \\
&\qquad\qquad 
+ \eps^2 \int_0^{\infty} \!\!\! \int_{\mathbb T^{d}_{\ell}} \frac{1}{\tau^2} \left[ \dfrac{\Delta \sqrt{R}}{\sqrt{R}} \Div ( R \Phi ) \right] \\
&\qquad\qquad  
+ \nu  \int_0^{\infty} \!\!\! \int_{\mathbb T^{d}_{\ell}}\frac{1}{\tau^2}  R \D U : \nabla \Phi   
+   \nu  \int_0^{\infty} \!\!\! \int_{\mathbb T^{d}_{\ell}} \frac{\dot \tau}{\tau}  R \Div \Phi \\
&\qquad\qquad   
+ \delta_1  \int_0^{\infty} \!\!\! \int_{\mathbb T^{d}_{\ell}} \frac{1}{\tau^2}  \nabla U : \nabla R \otimes \Phi 
+  \delta_2  \int_0^{\infty} \!\!\! \int_{\mathbb T^{d}_{\ell}}\frac{1}{\tau^2}  \Delta U \cdot \Delta \Phi 
\\
&\qquad\qquad  
+ \eta_1  \int_0^{\infty} \!\!\! \int_{\mathbb T^{d}_{\ell}} R^{-\alpha} \Div \Phi   
+ \eta_2  \int_0^{\infty} \!\!\!  \int_{\mathbb T^{d}_{\ell}}\frac{1}{\tau^2} \Delta^{s+1} R \Delta^s
\left[ \nabla R \cdot \Phi + R \Div \Phi \right].
\end{aligned}
\]

We number the integrals on the right-hand side $I_1$ to $I_9$  successively:
\begin{align*}
I_1 &= \int_0^{\infty} \dfrac{\chi \nu }{\tau^2} \int_{\mathbb T^d_{\ell}} \left( 4 |\nabla \sqrt{R}|^2 - 2dR \right),\, &
I_2 &=  r_0 \int_0^{\infty} \dfrac{\chi \nu }{\tau^4} \int_{\mathbb T^d_{\ell}} U \cdot \nabla \log R,\\
I_3 &= r_1 \int_0^{\infty} \dfrac{\chi \nu }{\tau^4} \int_{\mathbb T^d_{\ell}}  |U|^2 U \cdot \nabla R, &
I_4 &= \eps^2 \int_0^{\infty} \dfrac{\chi \nu }{\tau^4} \int_{\mathbb T^d_{\ell}}  R |\nabla^2 \log R |^2, \\
I_5 &= \int_0^{\infty} \dfrac{\chi \nu^2 }{\tau^4} \int_{\mathbb T^d_{\ell}} R \mathbb D U : \nabla^2 \log R  , &
I_6 & = -  \int_0^{\infty} \dfrac{\chi \nu^2 \dot{\tau} }{\tau^3} \int_{\mathbb T^d_{\ell}} 4|\nabla \sqrt{R}|^2 , \\
I_7 & =  \delta_1  \int_0^{\infty} \dfrac{\chi \nu}{\tau^4} \int_{\mathbb T^{d}_{\ell}}\nabla U : \nabla R \otimes \nabla \log R ,& 
I_8 &= \delta_2  \int_0^{\infty} \dfrac{\chi \nu  }{\tau^4}  \int_{\mathbb T^d_{\ell}} \Delta U \cdot \nabla \Delta \log R  , & \\
I_9 &= \eta_1 \int_0^{\infty} \dfrac{4 \chi \nu  }{\alpha \tau^2}  \int_{\mathbb T^d_{\ell}} \left| \nabla \sqrt{R^{-\alpha}}\right|^2 ,&
I_{10} &=  \eta_2 \int_0^{\infty} \dfrac{ \chi \nu  }{ \tau^4}  \int_{\mathbb T^d_{\ell}} |\Delta^{s+1} R|^2.
\end{align*} 

While, we rewrite the left-hand side:
\begin{align*}
LHS & =-  \left\langle \dfrac{\textrm{d}}{\textrm{d}t} \left[\dfrac{\nu}{\tau^2} \int_{\mathbb T^d_{\ell}} RU \cdot \nabla \log R  \right] , \chi \right \rangle
- \int_0^{\infty} \dfrac{2\chi \nu \dot{\tau}}{\tau^3} \int_{\mathbb T^d_{\ell}} RU \cdot \nabla \log R   \\
   & \quad + \int_{0}^{\infty} \dfrac{\chi \nu }{\tau^2} \int_{\mathbb T^d_{\ell}} RU \cdot \nabla \partial_t \log R +  \int_0^{\infty}  \frac{\chi \nu }{\tau^4}  \int_{\mathbb T^{d}_{\ell}}R U \otimes U : \nabla^2 \log R,
  \end{align*}
where we denote with brackets the duality in the sense of distributions. We proceed by computing the third term (denoted $L_1$)
in the right-hand side of this identity. For this, we remark that differentiating the continuity equation \eqref{RU-drag-reg1}, we obtain (in $L^{2}_{loc}(\mathbb R^+ ; L^2(\mathbb T^d_{\ell}))$):
\[
\partial_t ( R \nabla \log R  ) = - \dfrac{1}{\tau^2} \Div (R \nabla \log R \otimes U ) - \dfrac{1}{\tau^2}  \Div(R \nabla U)  + \dfrac{\delta_1}{\tau^2} \Delta \nabla R, 
\]
splitting the left-hand side of this identity and calling again the continuity equation, we conclude that:
\begin{align*}
R \partial_t\nabla  \log R  &  = \dfrac{1}{\tau^2} \Div(RU) \nabla \log R  - \dfrac{\delta_1}{\tau^2} \Delta R \nabla \log R   \\
& \quad - \dfrac{1}{\tau^2} \Div (R \nabla \log R \otimes U ) - \dfrac{1}{\tau^2}  \Div(R \nabla^{\top} U)  + \dfrac{\delta_1}{\tau^2} \Delta \nabla R.
\end{align*}
We infer then that,  a.e. (in $(0,\infty)$), we have:
\begin{align*}
\int_{\mathbb T^{d}_{\ell}} R U \cdot  \partial_t\nabla  \log R   
& = - \dfrac{1}{\tau^2} \int_{\mathbb T^{d}_{\ell}} RU \otimes U : \nabla^2 \log R  + \dfrac{1}{\tau^2} \int_{\mathbb T^d_{\ell}} R \nabla U^{\top} : \nabla U \\
& \quad  - \dfrac{\delta_1}{\tau^2} \int_{\mathbb T^d_{\ell}} \dfrac{\Delta R}{R} \Div (RU).
\end{align*}
Plugging this identity into $LHS$, we obtain:
\begin{align*}
LHS & =-  \left\langle \dfrac{\textrm{d}}{\textrm{d}t} \left[\dfrac{\nu}{\tau^2} \int_{\mathbb T^d_{\ell}} RU \cdot \nabla \log R  \right] , \chi \right\rangle
- \int_0^{\infty} \dfrac{2\chi \nu \dot{\tau}}{\tau^3} \int_{\mathbb T^d_{\ell}} RU \cdot \nabla \log R   \\
   & \quad + \int_{0}^{\infty} \dfrac{\chi \nu }{\tau^4} \int_{\mathbb T^d_{\ell}} R \nabla U : \nabla^{\top} U - \delta_1  \int_{0}^{\infty} \dfrac{\chi \nu }{\tau^4} \int_{\mathbb T^d_{\ell}} \dfrac{\Delta R}{R} \Div (RU).
  \end{align*}
Finally, combining the computations of the right-hand side and left-hand side, we re-interpret our identity as:
\begin{align*}
& \dfrac{\textrm{d}}{\textrm{d}t} \left[\dfrac{\nu}{\tau^2} \int_{\mathbb T^d_{\ell}} RU \cdot \nabla \log R  \right]
+ \dfrac{2 \nu \dot{\tau}}{\tau^3} \int_{\mathbb T^d_{\ell}} RU \cdot \nabla \log R   \\
& \quad 
+ \dfrac{\eps^2\nu }{\tau^4} \int_{\mathbb T^d_{\ell}} R |\nabla^2\log(R)|^2  
+ \left( \dfrac{\nu}{\tau^2} - \dfrac{\nu^2 \dot{\tau}}{\tau^3} \right) \int_{\mathbb T^d_{\ell}} 4 |\nabla \sqrt{R}|^2  \\
& \quad 
+ \dfrac{4 \eta_1 \nu }{\alpha}  \int_{\mathbb T^d_{\ell}} \left| \nabla \sqrt{R^{-\alpha}}\right|^2 
+ \dfrac{\eta_2 \nu}{\tau^4}   \int_{\mathbb T^d_{\ell}} |\Delta^{s+1} R|^2 \\
& \qquad 
  = \dfrac{2d\nu}{\tau^2} \int_{\mathbb T^d_{\ell}} R 
- \dfrac{r_0 \nu}{\tau^4} \int_{\mathbb T^d_{\ell}} U \cdot \nabla \log R - \dfrac{r_1 \nu}{\tau^4} \int_{\mathbb T^d_{\ell}}  |U|^2 U \cdot \nabla R \\
&  \qquad \qquad 
-  \dfrac{\nu^2}{\tau^4} \int_{\mathbb T^d_{\ell}} R \mathbb D U : \nabla^2 \log R  \\
& \qquad \qquad
- \dfrac{\delta_1 \nu}{ \tau^4}\int_{\mathbb T^{d}_{\ell}}\nabla U : \nabla R \otimes \nabla \log R  
- \dfrac{\delta_2\nu}{\tau^4}  \int_{\mathbb T^d_{\ell}} \Delta U \cdot \nabla \Delta \log R  \\
& \qquad \qquad 
- \dfrac{\delta_1 \nu}{ \tau^4}\int_{\mathbb T^{d}_{\ell}} \dfrac{\Delta R}{R} \Div (RU) 
+ \dfrac{\nu}{\tau^4} \int_{\mathbb T^d_{\ell}}  \nabla U:\nabla^{\top}U   .
\end{align*}  
This completes the proof.

\bibliographystyle{abbrv}
\bibliography{biblio}

\end{document}